\documentclass[10pt,reqno]{amsart}
\usepackage[final]{optional}
\usepackage[british,american]{babel}
\usepackage{bbm}
\usepackage{mathrsfs}
\usepackage{amsfonts, amsmath, wasysym,nccmath}

\usepackage{amssymb,amsthm,paralist}
\usepackage{latexsym,nicefrac}
\usepackage[usenames]{color}

\usepackage{url}

\definecolor{darkgreen}{rgb}{0,0.5,0}
\definecolor{darkred}{rgb}{0.7,0,0}
\usepackage[colorlinks, citecolor=darkgreen, linkcolor=darkred]{hyperref}
\usepackage{esint}
\usepackage{bibgerm}
\usepackage[normalem]{ulem}

\theoremstyle{plain}
\newtheorem{lemma}{Lemma}[section]
\newtheorem{thm}[lemma]{Theorem}
\newtheorem{prop}[lemma]{Proposition}
\newtheorem{cor}[lemma]{Corollary}

\theoremstyle{definition}

\newcommand{\ptheta}{\partial_\theta}
\newcommand{\pnu}{\partial_\nu}
\newcommand{\pnug}{\partial_{\nu_g}}

\newtheorem{rmk}[lemma]{Remark}

\numberwithin{equation}{section}

\newcommand{\al}{\alpha}

\newcommand{\de}{\delta}

\newcommand{\Om}{\Omega}

\newcommand{\La}{\Lambda}

\newcommand{\si}{\sigma}

\newcommand{\Si}{\Sigma}

\newcommand{\vph}{\varphi}
\newcommand{\vphi}{\varphi}

\newcommand{\ep}{\varepsilon}

\newcommand{\peps}{\partial_{\eps}}

\newcommand{\De}{\Delta}

\newcommand{\R}{\ensuremath{{\mathbb R}}}
\newcommand{\N}{\ensuremath{{\mathbb N}}}

\newcommand{\upto}{\uparrow}
\newcommand{\embed}{\hookrightarrow}

\newcommand{\grad}{\nabla}

\DeclareMathOperator{\inj}{inj}

\newcommand{\norm}[1]{\Vert#1\Vert} 
\DeclareMathOperator{\Div}{div}

\newcommand{\beq}{\begin{equation}}
\newcommand{\eeq}{\end{equation}}
\newcommand{\beqs}{\begin{equation*}}
\newcommand{\eeqs}{\end{equation*}}
\newcommand{\beqa}{\begin{equation}\begin{aligned}}
\newcommand{\eeqa}{\end{aligned}\end{equation}}
\newcommand{\beqas}{\begin{equation*}\begin{aligned}}
\newcommand{\eeqas}{\end{aligned}\end{equation*}}
\newcommand{\brmk}{\begin{rmk}}
\newcommand{\ermk}{\end{rmk}}
\newcommand{\partref}[1]{\hbox{(\csname @roman\endcsname{\ref{#1}})}}
\newcommand{\half}{\frac{1}{2}}
\newcommand{\mhalf}{\mfrac{1}{2}}
\newcommand{\thalf}{\tfrac{1}{2}}
\newcommand{\dd}{\mathrm{d}} 
\newcommand{\Id}{\text{Id}}

\newcommand{\leqs}{\lesssim}

\newcommand{\pt}{\partial_t}
\newcommand{\ps}{\partial_s}

\newcommand{\M}{\ensuremath{{\mathcal M}}}

\newcommand{\abs}[1]{\vert#1\vert}

\newcommand{\eps}{\varepsilon}
\newcommand{\na}{\nabla}

\newcommand{\dist}{\text{dist}}

\newcommand{\Col}{\mathcal C}
\newcommand{\Cyl}{{\mathscr{C}}}

\newcommand{\Eh}{E_{\half}} 
\newcommand{\pSi}{{\partial \Sigma}}
\newcommand{\MM}{\mathcal{M}} 
\newcommand{\Rea}{\text{Re}}
\newcommand{\Ima}{\text{Im}}
\newcommand{\HH}{\mathcal{H}} 
\newcommand{\png}{{\partial_{\nu_g}}}
\newcommand{\UU}{{\mathcal U}} 
\newcommand{\Ecut}{E_\varphi}

\newcommand{\supp}{\text{supp}}
\newcommand{\DD}{\mathbb{D}}
\newcommand{\ptaug}{{\partial_{\tau_g}}}
\newcommand{\Eupp}{\hat E}

\newcommand{\hf}{\frac{1}{2}} 
\newcommand{\p}{\partial} 
\newcommand{\greq}{\geqslant} 
\newcommand{\subs}{\subseteq} 
\newcommand{\inner}[2]{\langle #1, #2\rangle} 
\newcommand{\la}{\lambda} 
\newcommand{\err}{\text{err}}
\newcommand{\peta}{\partial_\eta}
\newcommand{\ptau}{\partial_\tau}
\newcommand{\MMiota}{\MM_{\iota_0}}
\newcommand{\MMtwoiota}{\MM_{2\iota_0}}

\title{\sc Flowing to free boundary minimal surfaces}
\author{Melanie Rupflin, Michael Struwe and Christopher Wright}
\date{\today}
\begin{document}
\maketitle
\begin{abstract}
We introduce a flow that is designed to flow maps $u:\Si\to \R^n$ which map the boundary of a general domain surface $\Si$ into a given (not necessarily connected) submanifold $N\hookrightarrow \R^n$ towards a free boundary (branched) minimal immersion supported by $N$. In the case when $\Si$ is the unit disc $D$, this task can be achieved by means of the Plateau-flow introduced in the work \cite{S} of the second author. When $\Si\neq D$, however, also the conformal type of the domain metric plays a role and it no longer suffices to deform the trace of the given map into a half-harmonic map as in \cite{S}. In order to overcome this issue, here we combine ideas of the Plateau-flow from \cite{S} with ideas of the Teichm\"uller harmonic flow from \cite{Rupflin-Topping}, in order to flow both an initial map $u_0$ with trace $u_0\colon\pSi\to N$ and an initial domain metric $g_0$ in a way that produces, as time tends to infinity, a half-harmonic map from $\pSi$ into $N$ whose harmonic extension is conformal and hence is a (branched) minimal immersion.  
\end{abstract}

\section{Introduction}
\subsection{Half-harmonic maps}
Let $N\hookrightarrow \R^n$ be a smooth, not necessarily connected,  closed submanifold of any dimension and let  $\Si$ be a closed surface with boundary $\partial \Si\neq \emptyset$. Given a metric $g$ on $\Si$, following \cite{DaLio-Riviere} and \cite{DaLio-Pigati}, we define the half-energy $E_{\half}(u,g)$ of a function $u\in H^\half(\partial \Si; N)$ as the Dirichlet energy of its harmonic extension, i.e.
\beq\label{def:half-energy}
    E_{\half}(u,g):= E_g(u_g):=\half\int_{\Si}\abs{\dd u_g}_g^2\,\dd v_g.
\eeq
Here and in the following  $E_g$ is the standard Dirichlet energy with respect to the metric $g$ and $u_g\colon\Si\to\R^n$ denotes the harmonic extension with respect to the metric $g$ of a given map $u\in H^\half(\pSi; N)=H^\half((\partial \Si, \dd s_g);N)$, i.e. the unique function $u_g:\Si\to \R^n$ with $\De_g u_g=0$ and trace $u_g\vert_{\pSi}=u$.

We furthermore denote by $\pi:N_{\eta}\to N$ the nearest point projection onto $N$, which is well-defined and smooth on the tubular neighbourhood 
\beqs
    N_{\eta}:= \{y \in \R^n:\dist_{\R^n}(y,N) < \eta\}
\eeqs
of $N$ if $\eta>0$ is chosen small enough, and recall that for any $p\in N$ the orthogonal projection $P_p:\R^n\to T_pN$ of the ambient $\R^n$ onto the tangent space $T_pN$ to $N$ at $p$ can be equivalently written as $P_p=\dd\pi(p)$.

As in \cite{DaLio-Riviere} and \cite{DaLio-Pigati}, we call a map $u\in H^\half(\partial \Si; N)$ half-harmonic if it is a critical point of $\Eh(u,g)$ in the sense that $\tfrac{\dd}{\dd \ep}\big|_{\ep=0} E_{\half}(\pi(u+\eps v),g)=0$ for every $v\in H^\half(\pSi; \R^n)$. As 
\beq\label{eq:energy_var-map}
    \mfrac{\dd}{\dd \ep}\Big|_{\ep=0} E_{\half}(\pi(u+\eps v),g)=\int_{\pSi}\peps \pi(u+\eps v) |_{\ep=0} \cdot\png u_g \dd s_g =\int_{\pSi} v\cdot P_u(\png u_g)\dd s_g,
\eeq
where $\nu_g$ is the outwards unit normal of $(\Si,g)$ along $\pSi$, this is equivalent to asking that the harmonic extension $u_g$ of $u$ satisfies the equation 
\beq\label{eq:half-harmonic}
    P_u(\png u_g)=0\ \hbox{ on }\pSi.
\eeq
Half-harmonic maps are related to free boundary minimal surfaces supported by $N$, that is, parametrisations $u\colon\Si\to\R^n$ of minimal surfaces with trace $u(\pSi)\subset N$ that meet the supporting surface $N$ orthogonally. Recall that a map $u_g\colon\Si\to\R^n$ whose components are harmonic functions parametrises a (possibly branched) minimal surface if and only if $u_g$ is (weakly) conformal, i.e. so that $u_g^*g_{\R^N}=\rho^2 g$ for a function $\rho\geq 0$. Orthogonality being implied by \eqref{eq:half-harmonic}, the harmonic extension of a half-harmonic map $u\in H^\half(\partial\Si;N)$ thus parametrises a free boundary minimal surface whenever this extension is (weakly) conformal. (In the following we will always allow branch points.)

We note that if $\Si=D$, then as observed in \cite{DaLio-Martinazzi-Riviere}, \cite{DaLio-Pigati} and \cite{Millot-Sire}, the equation \eqref{eq:half-harmonic} is in fact a sufficient condition for weak conformality of the harmonic function $u_g$. While for surfaces $\Si$ of higher genus or higher connectivity this condition is not sufficient, we can exploit that maps $u_g\colon\Si\to\R^n$ are weakly conformal if and only if the half-energy is critical with respect to variations of the metric. This was shown by Da Lio and Pigati \cite{DaLio-Pigati} in the setting of the half-energy and of half-harmonic maps considered here. Since we will use the corresponding properties of the variation of the half-energy in our construction of the flow, we briefly recall how this key property can be obtained. 

To this end, we note that the first variation of the Dirichlet energy $E_g$ (for fixed metric $g$) along maps $v_\eps=u_{g_\eps}$ that are obtained as harmonic extensions of a fixed map $u:\pSi\to \R^n$ with respect to a family of metrics $g_\eps$ with $g_{\eps=0}=g$ is always so that 
\begin{equation*}
    \frac{\dd}{\dd\eps} E_{g}(v_{\eps})=\int_\Si \inner{\na_g(\peps v_\eps)}{\na_g u_g}_g \dd v_g =-\int_{\Si} \peps v_{\eps}\cdot \Delta_g u_g \dd v_g+\int_{\p\Si}\peps v_\eps \cdot \png u_g \dd s_g=0
\end{equation*}
as $u_g$ is harmonic and $\peps v_\eps\vert_{\partial \Si}=0$. Here and in the following we use the convention that all derivatives with respect to $\eps$ are evaluated at $\eps=0$. The variation of the half-energy with respect to the metric hence reduces to the variation of the Dirichlet energy of the fixed map $u_g$ with respect to a varying metric $g_\eps$, which we recall is given by 
\beq\label{eq:energy_var_metric}
    \frac{\dd}{\dd\eps} E_{\half}(u,g_\eps) =\frac{\dd}{\dd\eps} E_{g_\eps}(u_g) =-\half\int_{\Si}\langle \peps g_\eps,k(u_g,g)\rangle_g \dd v_g,
\eeq
see for instance \cite[Lemma 3.4.1]{Baird_Wood_Textbook}. Here and in the following 
\beq\label{def:stress}
    k(v,g)=v^*g_{\R^n}-\half \abs{\dd v}_g^2 g
\eeq 
denotes the stress-energy tensor of maps $v:\Si\to \R^n$, which we note is always trace-free and of course vanishes if and only if $v$ is conformal.

\subsection{Gradient flow}
In view of the above, in order to define a flow that will evolve an initial (parametrised) surface to a free boundary minimal surface it is natural to consider a gradient flow of the half-energy. 

In the special case of the disc, where the moduli space of metrics is given by a single point, this was achieved by Wettstein in \cite{Wettstein_1,Wettstein_2,Wettstein_3} and by the second author in \cite{S}. In fact, in this case we see an amazing similarity of the salient features of the heat flow for half-harmonic maps, called ``Plateau flow'', and the heat flow of harmonic maps from a closed surface into smooth closed target manifolds laid out in \cite{Struwe-85}. In particular, for any given initial data $u_0\in H^\half(\partial\Si;N)$, it was shown in \cite{S} that there is a unique global weak solution to the Plateau flow whose energy is non-increasing and which is smooth for $t>0$ away from finitely many points in space-time where non-trivial disc-type free boundary minimal surfaces "bubble off". Furthermore, the results of \cite{S} ensure that along a sequence of times going to infinity, the solution converges weakly to a disc-type free boundary minimal surface, again away from finitely many "bubble points" where non-trivial disc-type free boundary minimal surfaces emerge after rescaling. As recently shown by the third author \cite{Wright}, uniqueness in this case holds even in an a priori much larger class of weak solutions to the Plateau flow, which in particular includes all weak solutions with non-increasing energy, again in exact analogy with what is known in the case of the harmonic map heat flow from \cite{Freire} and \cite{Rupflin_08}.

In the present paper, we now consider the corresponding problem for surfaces $\Si$ of higher genus and/or connectivity. We introduce a coupled flow which evolves both an initial map $u_0\colon\partial\Si\to N$ and an initial domain metric $g_0$ by (a projection of) the negative gradient of $E_\hf$ so as to yield a critical point of the half-energy with respect to both $u$ and $g$, hence inducing a free boundary minimal surface.

\begin{rmk}
	In the following, we will only consider domain surfaces $\Si$ which are orientable as for non-orientable $\Si$ all results can be recovered by passing to the orientable double cover and working with metrics and function spaces which are invariant under the non-trivial covering space transformation.
\end{rmk}
Since the energy is conformally invariant, we can restrict the class of admissible metrics to consist of a unique representative for each conformal class of metric. For this we make use of standard uniformisation results, see e.g. \cite{Osgood_Phillips_Sarnak}. For (orientable) $\Si\neq D$ of general type we will work with the unique representative $g$ of each conformal class for which the boundary curves are geodesics in $(\Si,g)$ and which is hyperbolic (i.e. has Gauss curvature $K_g\equiv -1$). When $\Si$ is a cylinder we instead choose as our unique representative the flat metric $g$  ($K_g\equiv 0$) for which the boundary curves are geodesics and $(\Si,g)$ has unit area. We denote by $\MM(\Sigma)$ the space of such constant curvature metrics and note that our definition of the flow will ensure that for an initial metric $g_0\in\MM(\Si)$, the evolving metrics always will be in $\MM(\Si)$.

We recall that at any $g\in\MM(\Si)$ the tangent space splits $L^2(\Si,g)$-orthogonally as 
\beq\label{eq:split}
    T_{g}\MM(\Si)=\{L_X g\}+H(g),
\eeq
where $\{L_X g\}$ is the set of Lie-derivatives generated by vector fields that are parallel to $\pSi$ on $\pSi$, giving rise to $1$-parameter families of diffeomorphisms of $\Si$, and where the  \textit{horizontal space} $H(g)$ consists of all symmetric $(0,2)$-tensors $h$ which are trace-free and divergence-free and which satisfy $h(\nu_g,\tau_g)=0$ on $\partial \Si$,  $\tau_g$ a unit tangent vector field along $(\pSi,\dd s_g)$, c.f. \cite{Tromba}.

\begin{rmk}\label{rmk:conformal}
    We note that for a half-harmonic map $u:(\p\Si,g) \rightarrow N$ the tensor $k(u_g,g)$ is always an element of $H(g)$. Indeed since $u_g$ is harmonic $k(u_g,g)$ is divergence free, while \eqref{eq:half-harmonic} implies that $k(u_g,g)(\nu_g,\tau_g)=\inner{\pnug u_g}{\ptaug u_g}$ vanishes on the boundary since $u$ maps into $N$. Hence the harmonic extension of a half-harmonic map is conformal if and only if $P_g^H(k(u_g,g)) = 0$, where we define $P_g^{H}$ as the $L^2(\Si,g)$-orthogonal projection from the space of symmetric $(0,2)$-tensors to $H(g)$. 
\end{rmk}

Following the construction of the Teichm\"uller harmonic map flow in \cite{Rupflin-Topping}, we now exploit the fact that the energy is invariant under simultaneous pull-back of both the map and the metric (by the same diffeomorphism) to restrict the movement of the metric component to be $L^2$-orthogonal to the space of tensors $\{L_X g\}$ that is generated by the action of diffeomorphisms on the metric. We hence define our flow as 
\beqa\label{def:flow-abstract}
    \partial_t u&= -\na^{L^2}_{u} E_\half(u,g)= -P_u(\png u_g)\\
    \partial_t g&=-P_g^{H}\big(\na_g^{L^2} E_\half(u,g)\big)=\half  P^H_g(k(u_g,g)).
\eeqa
The evolution of the map component $u=u(t)$ is hence described by a variant of the Plateau flow 
studied by the second author in \cite{S}, albeit now considered on a general surface $\Si$ with a time-dependent metric $g(t)$ rather than on a disc with fixed metric; the evolution of the metric $g=g(t)$ on the other hand can be viewed as an evolution equation in the infinite dimensional manifold $\M(\Si)$ of constant curvature metrics with geodesic boundary curves described above. We observe that in the case when $\Si=D$ we have $H(g)=\{0\}$ for any metric $g$ on $D$, and so our coupled gradient flow \eqref{def:flow-abstract} reduces to the Plateau flow for maps $u(t)$ on the disc (with fixed metric) considered in \cite{S}.

\subsection{Main results}
Our first main result establishes the existence of solutions to this new geometric flow and gives a description of the potential singularities that the components of this system of equations might form. 

\begin{thm}\label{thm:1}
    Let $\Si\neq D$ be any orientable surface with boundary and let $\M(\Si)$ be the set of metrics with constant curvature and geodesic boundary considered above. Then to any initial metric $g_0\in \MM(\Si)$ and any initial map $u_0\in H^\half((\pSi,g_0);N)$ there exists a weak solution $(u,g)$ of the coupled flow
    \beq\label{eq:flow}
        \partial_t u =-P_u(\png u_g),\qquad \partial_t g=\half P^H_g(k(u_g,g))
    \eeq
    which has non-increasing energy and is defined on a maximal interval $[0,T_\infty)$ where $T_\infty=\infty$ unless the domain metrics degenerate in finite time, i.e. unless 
    $$\inj(\Si,g(t))\to 0 \text{ as } t\upto T_\infty \text{ for a finite } T_\infty.$$ Furthermore,
    \begin{enumerate}
        \item Away from a finite number of singular times $0<T_i^s<T_\infty$, both the map and metric component of the flow are smooth and the energy decays according to 
        \beq\label{eq:energy-decay}
            \frac{d}{dt}E_{\half}(u,g)=-\norm{P_u(\png u_g)}_{L^2(\pSi,g)}^2-\frac14 \norm{P_g^{H}(k(u_g,g))}_{L^2(\Si,g)}^2.
        \eeq
        \item Across each $T_i^s$ the flow of metrics remains regular in the sense that $g(t)$ is Lip\-schitz continuous in time with respect to any $C^k$-metric in space. 
        \item Any such singular time $T_i^s$ is characterised by the \emph{bubbling-off} of a finite number of minimal discs of the map component, exactly as in \cite{S}, see also  Section \ref{sect:bubble}.
    \end{enumerate}
\end{thm}

Here and in the following we say that $(u,g)$ is a weak solution of the flow to initial data $(u_0,g_0)$ on an interval $[0,T)$ if 
\vspace{-\baselineskip}
\begin{itemize}
    \item the metric component is a continuous curve $g:[0,T)\to (\M^3(\Si),\dist_{H^3})$ with $g(0)=g_0$ which is differentiable at  a.e. $t\in [0,T)$ and so that the second equation in \eqref{eq:flow} is satisfied at a.e. such $t$. Here $\MM^3$ denotes the set of constant curvature metrics with coefficients in $H^3$ as discussed in Section \ref{sect:main} below. 
    \item the map component is given by a $u\in L^\infty([0,T); H^{\half}((\pSi,\dd s_g);N)$ which is so that $\pt u\in L_{loc}^2([0,T);L^2(\p\Si,g))$ and so that  the first equation of \eqref{eq:flow} is satisfied in the sense of distributions.
\end{itemize}

In the absence of singularities at infinity, the following result ensures that the flow deforms the given initial map into a free boundary minimal surface as desired.

\begin{thm}\label{thm:2}
    Suppose that $(u,g)$ is a global weak solution of \eqref{eq:flow} for which $\inj(\Si,g)$ remains bounded away from zero and for which energy does not concentrate as $t\to \infty $, i.e. so that $\limsup_{t\to \infty} \sup_{x\in \pSi} E_g(u_g(t),B_{r}^{g(t)}(x))\to 0$ as $r\to 0$. 
    
    Then there exist $t_j\to \infty$ so that after pull-back by diffeomorphisms the pairs $(u(t_j),g(t_j))$ converge smoothly to a limiting pair $(u^*,g^*)\in C^\infty(\pSi;N)\times \MM(\Si)$ which, if $u^*$ is non-constant, is so that $u^*_{g^*}:(\Si,g^*)\to \R^n$ is (weakly) conformal, harmonic and so that $u^*_{g^*}(\Si)$ meets $N$ orthogonally; that is, $u^*_{g^*}(\Si)$ represents a (possibly branched) free boundary minimal surface supported by $N$.
\end{thm}

\begin{rmk}\label{rmk:Plateau_Problem}
    We note that in the special case where $N$ is given by a collection of $k$ disjoint closed curves $\Gamma_1,\ldots,\Gamma_k$, $k$ the number of boundary components $\si_i$ of $\Si$, and where each $u_0\vert_{\si_i}$ is a map into $\Gamma_i$ with non-zero degree, this property is preserved along regular solutions of the flow. In this case the minimal surface obtained in the above theorem can be thought of as a solution of a generalised version of the Douglas-Plateau problem, 
    compare also the discussion in Sections 1.5 and 1.6 of \cite{S}.
\end{rmk}

\textbf{Outline of the paper.}  We discuss the key steps of the proof of short time existence of solutions to the flow  in Section \ref{sect:main} and carry out the required detailed analysis of the equations satisfied by the metric and map  components in the subsequent Sections \ref{sect:metric} and \ref{sect:map}. In Section \ref{sect:higher_reg}, we prove higher regularity estimates which enable us to deduce the smoothness of the solution until either the metric degenerates or energy concentrates. Combined with the analysis of finite time singularities, which is carried out in Section \ref{sect:bubble} and closely follows \cite{S}, this allows us to complete the proof of our first main result, Theorem \ref{thm:1}. Finally, in Section \ref{sect:asym} we study the asymptotic behaviour, proving Theorem \ref{thm:2}.

\textbf{Acknowledgment:}
Christopher Wright was supported by the Engineering and Physical Sciences Research
Council. For the purpose of open access, the authors have applied a CC BY public copyright licence to any author accepted manuscript arising from this submission.
\section{Short time existence of solutions}\label{sect:main}
In this section, we explain the key steps needed to prove local existence of solutions of \eqref{eq:flow} for given smooth initial data $(u_0,g_0)$. This proof will be based on an iteration argument that is carried out on a sufficiently small interval $[0,T]$ and in suitable Sobolev spaces $X_m(T)$ of maps of regularity $H^m$ and  curves of metrics in $\MM^{m+1}$ of regularity $C^1_tH^{m+1}$ as made precise below. This argument will be based on results about  (related) equations for the metric and map component which are stated in a series of lemmas in the present section, but whose proofs are postponed to Section \ref{sect:metric} (for results about the metric) and Section \ref{sect:map} (for results about the map). 

In the following, we will focus mainly on the case of domains of general type as the analysis of the metric component simplifies very significantly if $\Si$ is a cylinder, as in this case the equation for the metric reduces to a well-controlled ODE. We will discuss this special case in Section \ref{subsec:cylinder} and for now simply note that all results stated below are applicable also for $\Si$ a cylinder and the corresponding space $\MM(\Si)$ of flat unit area metrics. 

\subsection{Set-up}
Given a surface $\Si$ with boundary, we recall that we may think of this domain as a fixed subset of the Schottky double $\hat\Si$, which is the closed surface that we obtain as quotient of $\Si\times\{-1,1\}$ under the identification $(p,1)\sim(p,-1)$ for $p\in\partial \Si$. Moreover, we note that if $g$ is any hyperbolic metric on $\Si$ for which the boundary curves are geodesics, then $g$ can be extended by (even) symmetry to a hyperbolic metric $\hat g$ on $\hat\Si$. 

Following the approach of Tromba \cite{Tromba}, given any $k\in \N_{\ge 3}$ and any fixed finite (and symmetric) set of (smooth) coordinate charts on $\hat\Si$ we can consider the set  $\MM^{k}=\MM^{k}(\Si)$ of metrics $g$ on $\Si$ which we obtain as restrictions of hyperbolic metrics $\hat g$ on $\hat\Si$ which have the above even symmetry across $\pSi$ and whose coefficients are of class $H^{k}$ in these charts. We measure the distance between two elements of $\MM^{k}$ by letting
\beqs
    \dist_{H^{k}}(g_1,g_2):=\inf\int_1^2\norm{\p_s g(s)}_{H^{k}(\Si,g(s))}\dd s,
\eeqs
where the infimum is taken over the set of all $C^1$-paths $g=g(s)\in\MM^{k}(\Si)$, $1\le s\le 2$, that connect $g_1=g(1)$ and $g_2=g(2)$. Here and in the following all Sobolev norms are computed using the Levi-Civit\`a connection with respect to the indicated metric. 

Given any $m\geq 2$ and a curve of metrics $g\in C^1([0,T],\MM^{k})$ we will work with maps that are contained in the space 
\beqa\label{def:Xm}
	X_m(T,g) = \{u:\pSi\times [0,T]\to \R^n& :u_{g} \in L^\infty ([0,T]; \ H^m(\Si,g)) \\&  \text{ and } \p_t (u_{g})\in 
     L^\infty ([0,T]; H^1(\Si,g)) \},
\eeqa
which we equip with the corresponding norms 
\beq\label{def:norm_Xm}
    \norm{u}_{X_m(T,g)}^2 := \norm{u_{g}}_{L^\infty([0,T];H^m(\Si,g))}^2 + \norm{\p_t (u_{g})}_{L^\infty([0,T]; H^1(\Si,g))}^2.
\eeq

While these norms depend on the specific choice of $g$, we will see that the spaces $X_m(T,g)$ themselves are indeed independent of the choice of $g$ as the change of harmonic extensions and Sobolev norms along such curves of metrics is well controlled, compare Lemma \ref{lemma:metric_variations}. We hence often drop the reference to $g$ in the notation for $X_m(T,g)$ and simply write $X_m(T)$, and will at times also use the shorthand $X_{m,loc}(T):= \bigcup_{T'<T} X_{m}(T')$. We furthermore note that in the proof of short-time existence we will for the most part only need to consider curves of metrics which are contained in a small neighbourhood $\UU$ of the initial metric $g_0$ and whose velocity is uniformly bounded. For such curves of metrics, these norms are uniformly equivalent, compare Lemma \ref{lemma:U-basics} below, which will allow us to work with respect to the fixed norm $\norm{\cdot}_{X_m(T,g_0)}$ in the relevant fixed point argument.

We furthermore note that we can always work with the spaces $\MM^{m+1}$ which we obtain by fixing coordinate charts in which the initial metric is smooth, and will see that this property is preserved along the flow. We can hence in the following always assume that the initial metric is an element of $\MM^{m+1}$ for any $m\in \N$.

In the following we also use the convention that all constants are allowed to depend on the setting we consider, i.e. the fixed (topology of the) domain surface $\Si$, the fixed manifold $N\hookrightarrow \R^n$ and later on also on the fixed extensions $P_\cdot$ and $P_\cdot^\perp$ of the orthogonal projections onto $T_\cdot N$ respectively $T_\cdot ^\perp N$  chosen in Remark \ref{rmk:extension-projection}, without further mentioning this. 
\subsection{Evolution of the metric component}
We first consider the problem of solving the equation 
\beq\label{eq:metric_lemma}
    \partial_t g=\half P^H_g(k(v_g,g)) \text{ with } g(0)=g_0
\eeq
for a given map $v\in X_2(T,g_0)$, $k$ the stress-energy tensor defined in \eqref{def:stress}. In Section \ref{sect:metric} we will prove

\begin{prop}\label{prop:step1}
    For any $m\ge 2$ and any $\iota_0>0$ there exist constants $C,\delta>0$ so that  the following holds true for any $T>0$, any $g_0\in \MM^{m+1}$ with $\inj(\Si,g_0)\geq 2\iota_0$ and the neighbourhood $\UU$ of $g_0$ in $\MM^{m+1}$ described in Lemma \ref{lemma:U-basics}. 

    For any $v\in X_2(T)$ there exists a unique solution $g_v\in C^{1}([0,T_0],\MM^{m+1})$ of \eqref{eq:metric_lemma} which is defined and so that $g_v(t)\in\UU$, and hence in particular 
    \beq\label{est:inj-below}
        \inj(\Si,g_v(t))\geq \half \inj(\Si,g_0)\geq \iota_0,
    \eeq
    at least on the interval $I=[0,T_0]$ for $T_0:=\min(T,\delta/\Eupp)$, $\Eupp$ chosen so that  
    \beq\label{ass:Ev-apriori} 
        \sup_{t\in [0,T]}E_{\half}(v(t),g_0)\leq \Eupp.
    \eeq
    Moreover, for every $t,\tilde t\in [0,T_0]$ we have
    \beq\label{est:ptg_apriori_1}
        \norm{\pt g_v(t)}_{H^{m+1}(\Si,g_0)}\leq C \Eh(v(t),g_0),
    \eeq
    and 
    \beq\label{est:ptg_apriori_2}
        \norm{\pt g_v(t)-\pt g_v(\tilde t)}_{H^{m+1}(\Si,g_0)}\leq C\abs{t-\tilde t}\big[ \Eupp^\half \norm{\pt v_{g_0}}_{L^\infty(I;H^1(\Si,g_0))}+\Eupp \big].
    \eeq
    Moreover, the map $v\mapsto g_v$ is Lipschitz the sense that the estimates 
    \beq\label{est:g_diff_apriori}
        \dist_{H^{m+1}}(g_v(t),g_{\tilde v}(t))\leq C \Eupp^\half t \norm{(v-\tilde v)_{g_0}}_{L^\infty(I;H^1(\Si,g_0))}
    \eeq
    and 
    \beqa\label{est:ptg_diff_apriori}
        \norm{\pt (g_v-g_{\tilde v})(t)}_{H^{m+1}(\Si,g_0)}&\leq C \Eupp^\half \norm{(v-\tilde v)_{g_0}(t)}_{H^1(\Si,g_0)}\\
        &\quad+ C\Eupp^\frac{3}{2} t \norm{(v-\tilde v)_{g_0}}_{L^\infty(I;H^1( \Si,g_0))}
    \eeqa 
    hold for all $t\in [0,T_0]$ and all for all $v$ and $\tilde v$ which satisfy \eqref{ass:Ev-apriori} for a given $\Eupp$.
\end{prop} 

\begin{rmk}\label{rmk:def_UU}
	The neighbourhood $\UU=\UU^{m+1}(g_0)$ in Lemma \ref{lemma:U-basics} is given by a ball in $\MM^{m+1}$ around $g_0$ whose radius only depends on $\iota_0$ and which is chosen so that metrics in $\UU$ are uniformly equivalent and induce uniformly equivalent $H^{m+1}$-norms. Hence the above estimates all remain valid also if some (or all) of the norms are computed with respect to another metric of $\UU$, such as $g(t)$, instead of $g_0$. 
	
	We can apply the above lemma iteratively to see that unless the injectivity radius tends to zero as $t$ approaches some time $T_1<T$,  the solution of \eqref{eq:metric_lemma} is guaranteed to exist on the whole interval $[0,T]$ where $v$ is defined. Moreover, the above lemma shows that  that the estimate
	\beq\label{est:velocity-by-energy}
        \norm{\pt g_v(t)}_{H^{m+1}(\Si,g(t))}\leq C \Eh(v(t),g(t)) 
	\eeq   
	holds true with a constant $C$ that only depends on a lower bound on the injectivity radius of $(\Si,g(t))$ at that specific time $t$ and the exponent $m$.
\end{rmk}
We note that for the proof of short-time existence of a solution of our coupled flow \eqref{eq:flow}, a weaker version of the above lemma, in which the constants (and the size of the neighbourhood) are allowed to depend on the initial metric $g_0$ itself, instead of just on $\iota_0$, would suffice, but that we choose to formulate Proposition \ref{prop:step1} and \eqref{est:velocity-by-energy} in the above more precise form to make them applicable also for the long-term analysis of the flow. Namely we will use that the above proposition immediately implies the following. 

\begin{rmk}\label{rmk:uniform_metric_time}
    Let $(u,g)(t)$, $t\in [0,T]$, be any weak solution of our flow \eqref{eq:flow} with non-increasing energy. Then the metric component remains in the neighbourhood $\UU(g(0))$ described in Lemma \ref{lemma:U-basics} at least on  a time interval of the form $[0,\min(T,T_0)]$, for a number $T_0=T_0(\hat E,\iota_0)>0$ that only depends on an upper bound on the initial energy and a lower bound $2\iota_0$ on the injectivity radius of the initial metric.  
\end{rmk}

\subsection{Regularized equation for the map component}
Instead of directly proving a result similar to Proposition \ref{prop:step1} also for the map component, we follow the approach of the second author from \cite[Lemma 5.2]{S}, and first consider the simpler, and in particular linear, problem of determining a solution $u=u_{\ep,g,v}$ of the regularized equation 
\begin{equation}\label{eq:ep_v_map}
	\p_t u = -(\ep + P_{v})\p_{\nu_{g}}u_{g}\ \hbox{ on }\pSi
\end{equation}
for given $\ep>0$ and given curves of maps $v:\pSi\to N_\eta$ and metrics $g$. 

\begin{rmk}\label{rmk:extension-projection}
    Here and in the following we use fixed extensions $P_\cdot, P_\cdot^\perp \in C_b^\infty(\R^n;\R^{n\times n}) $ of the orthogonal projections $P_p:\R^n\to T_pN$, $P_p^\perp:\R^n\to T_p^\perp N$ which are chosen so that 
    \beq\label{def:ext-nbhd}
        P_p=P_{\pi(p)} \text{ and } P_p^\perp=P_{\pi(p)}^\perp \text{ for all } p\in N_\eta
    \eeq 
    on a small tubular neighbourhood $N_\eta$ where $\eta>0$ is chosen so that the nearest point projection is well defined and smooth on the closure $\bar N_\eta$ of $N_\eta$. We note that this choice of extension in particular ensures that $P_v$ and $P_v^\perp$ are orthogonal projections with $P_v+P_v^\perp=\textit{Id}$ for maps $v$ whose image is contained in this tubular neighbourhood $N_\eta$.
\end{rmk}

We note that we should not expect solutions of the above equation \eqref{eq:ep_v_map} to remain in $N$ even if $v$ maps into $N$, since we no longer evaluate the projection $P$ at the points $u(x,t)$ and since the term $\ep\p_{\nu_g} u_g$ in general will not be tangential to $N$. Thus, in order to later be able to apply fixed point arguments, we consider \eqref{eq:ep_v_map} not only for maps whose image is contained in $N$, but allow the maps $v$ to take values in this (fixed) tubular neighbourhood $N_{\eta}$.

Since the term $-\ep\p_{\nu_{g}}u_{g}$ compensates for the degeneracy introduced by the projection operator, we can use a Galerkin approximation to establish local existence for \eqref{eq:ep_v_map}. While this proof of existence of solutions to the regularised equation \eqref{eq:ep_v_map} crucially uses that $\eps>0$, we will be able to obtain $H^m$-bounds for these solutions which are independent of $\eps$. This will be essential in later arguments where we will send $\eps$ to zero in order to obtain a solution of our original equation \eqref{def:flow-abstract}.

The relevant results on the regularised equation \eqref{eq:ep_v_map}, which will be proven in Section  \ref{subsec:map-lin-reg}, can be summarised by the following proposition. 

\begin{prop}\label{prop:ep_map_solve}
    Let $\ep\in (0,\hf]$, $m\ge 4$ and $T>0$. Then for any given $u_0\in H^{m-\hf}(\p\Si;\R^n)$, any $g_0\in\MM^{m+1}$, any $g\in C^1([0,T];\MM^{m+1})$ with $g(0)=g_0$, and any $v:[0,T]\times\p\Si\rightarrow N_\eta$ with $v\in X_m(T)$, there exists a unique solution $u=u_{\ep,g,v}\in X_m(T)$ of \eqref{eq:ep_v_map} with $u(0) = u_0$, and this solution is furthermore so that $u\in L^2([0,T];H^{m}(\pSi,g))$ 
    and so that the estimates
    \beqa\label{est:Hm_u_epsilon_v_eq}
        \norm{u_{g}}_{L^\infty([0,T];H^m(\Si,g)))}^2 \leq e^{CT}\norm{(u_0)_{g_0}}_{H^m(\Si,g_0)}^2
    \eeqa
    and 
    \beqa\label{est:Hm_u_epsilon_v_eq-extra}
        \eps \norm{\p_\nu \nabla^{m-1} u_g}_{L^2([0,T];L^2(\p\Si,g))}^2 \leq 2 e^{CT}\norm{(u_0)_{g_0}}_{H^m(\Si,g_0)}^2
    \eeqa
    hold for a constant $C$ that only depends on $m$ and numbers $0 < \iota_0, M, \La < \infty$, which are chosen so that for all $t\in [0,T]$,
    \beq\label{ass:const_lemma_eps} 
        \inj(g(t))\geq \iota_0,\quad \norm{\pt g(t)}_{H^{m}(\Si,g(t))}\leq M,\quad \norm{(v(t))_{g(t)}}_{H^m(\Si,g(t))}\leq \La.
    \eeq
    We stress in particular that $C$ is \emph{independent} of $\eps$. 
\end{prop}

In addition to this result, we will also show that for every $\eps\in (0,\half]$, the map $(v,g)\mapsto u_{\ep,v,g}$ is Lipschitz in the sense described in Lemma \ref{lemma:ep_Lip_est}. 

Combining these results on the map and metric component will then allow us to prove that for suitably small $T=T(\eps,u_0,g_0,m)>0$, the map
$
    \Psi=\Psi_{\ep,u_0,g_0}: v\mapsto u_{\ep,v,g_v}
$ 
is a contraction on a suitably chosen subset of $X_m(T)$, see Lemma \ref{lemma:FPT} for details. We can hence deduce that for each $\eps\in (0,\half]$, the regularised system 
\begin{equation}\label{eq:flow-ep_coupled} 
	\p_t u_\eps = -(\ep + P_{u_\eps})\p_{\nu_{g_\eps}}u_{g_\eps}, \quad \p_t g_\eps = \hf P_{g_\eps}^H(k(u_{g_\eps},g_\eps))
\end{equation}
can be solved at least on a small, a priori $\eps$-\textit{dependent}, time interval interval $[0,T_\eps]$. 

The fact that the estimate \eqref{est:Hm_u_epsilon_v_eq} on the evolution of the $H^m$ norm of the map component is $\eps$-independent, combined with the uniform control on the metric component obtained in Proposition \ref{prop:step1}, will then allow us to establish that these solutions of the regularised system \eqref{eq:flow-ep_coupled} indeed exist, and remain well controlled, on an $\eps$-\textit{independent} time interval. To be more precise, in Section \ref{subsec:reg-coupled} we will show the following proposition. 

\begin{prop}\label{prop:ep_system}
    For any $m\ge 4$, any $\iota_0>0$ and any $\La_0<\infty$, there exists a time $T>0$ and a constant $C$ so that for any initial metric $g_0\in\MM^{m+1}$ with $\inj(\Si,g_0)\geq 2\iota_0$ and any initial map $u_0 \in H^{m-\hf}((\p\Si,g_0);N)$ with $\norm{(u_0)_{g_0}}_{H^m(\Si,g_0)}\leq \La_0$ the following holds true. 

    For every $0<\ep\le\hf$ there exists a unique solution $(u_\eps,g_\eps)\in X_m(T)\times C^{1}([0,T];\M^{m+1})$ of the regularised system \eqref{eq:flow-ep_coupled} with $u_\eps(0) = u_0$, $g_\eps(0)=g_0$ and on this $\eps$-\textit{independent} interval $[0,T]$ the map component remains bounded by 
    \beq\label{est:Hm-uniform-prop}
        \norm{(u_\eps)_{g_\eps}}_{H^m(\Si,g_\eps)}\leq 2 \La_0 
    \eeq
    while the metric component $g_\eps$ 
    remains in the neighbourhood $\UU\subset\MM^{m+1}$ of $g_0$ obtained in Lemma \ref{lemma:U-basics} and is so that 
    \beq\label{est:C11-uniform-prop} 
        \norm{g_\eps}_{C^{1,1}([0,T];\MM^{m+1})}\leq C. 
    \eeq 
\end{prop}

As neither $T>0$ nor $C$ depend on $\eps \in (0,\hf]$, we can exploit the above uniform bounds on the solutions $(u_\eps,g_\eps)$ of the regularised flow \eqref{eq:flow-ep_coupled} and pass to the limit along suitable $\ep\downarrow 0$ to obtain a solution of our original problem on the interval $[0,T]$ for $T=T(\La_0,\iota_0,m)>0$ as in the above Proposition \ref{prop:ep_system}. Iteratively applying this short-time existence result which, 
in view of Remark \ref{rmk:uniform_metric_time} and the fact that the energy is non-increasing along regular solutions of the flow, 
yields a solution for an additional time interval whose length is bounded away from zero 
unless the injectivity radius goes to zero or the $H^m$-norm goes to infinity,
we can hence deduce

\begin{prop}\label{prop:Hm_plateau_existence}
	For every initial data $(u_0,g_0)\in H^{m-\hf}(\p\Si;N)\times \MM^{m+1}$, $m\geq 4$, there exists a solution $(u,g) \in X_{m,loc}(T_m^*)\times C_{loc}^1([0,T_m^*);\MM^{m+1})$ to the coupled flow \eqref{eq:flow} on a maximal time interval $[0,T_m^*)$, where $T_m^*=\infty$ unless either 
	\beq\label{def:Tm*}
		\lim_{t\uparrow T_m^*}\inj(g(t)) = 0 \text{ or } \lim_{t\uparrow T_m^*}\norm{u(t)}_{H^m(\Si,g(t))}=\infty.
	\eeq
\end{prop}

In Section \ref{sect:higher_reg}  we will establish that the $H^m$-norm remains indeed controlled for as long as there is no concentration of energy. Combined with the analysis of potential finite time singularities of the map component that is carried out in Section \ref{sect:bubble}, following \cite{S}, this will complete the proof of our main Theorem \ref{thm:1}.

\section{Analysis of the metric component}\label{sect:metric}
We will first carry out the analysis of the metric component in detail in the more difficult case where our (orientable) domain surface $\Si$ is neither a cylinder nor a disc. We will discuss the adaptions, and the significant simplifications, of this analysis for the case where $\Si$ is cylinder later in Section \ref{subsec:cylinder} and again recall that in the case where $\Si$ is a disc, the flow reduces to the Plateau flow with fixed metric analysed by the second author in \cite{S}.  

\subsection{Basic properties of set $\MM(\Si)$ of hyperbolic metrics}\label{subsec:metric-basic} 
Let $\Si$ be an orientable surface of general type and let $\MM(\Si)$ be the set of hyperbolic metrics for which the boundary curves are geodesics as considered above. In this section we collect some useful properties of these metrics which will be used not only in the analysis of horizontal curves as considered in Proposition \ref{prop:step1}, but also later on in the paper. To lighten the notation we will often use the short-hand 
\beq\label{def:Miota}
    \MMiota:= \{g\in \MM: \inj(\Si, g)\geq \iota_0\} \text{ for }\iota_0>0 
\eeq
as well as $\MMiota^{m+1}:= \{g\in \MM^{m+1}: \inj(\Si, g)\geq \iota_0\} $ and note the following.

\begin{rmk}\label{rmk:compact-moduli}
    It is easy to see that the subset of moduli space which corresponds to metrics with $\inj(g)\geq \iota_0$, $\iota_0>0$ any fixed number, is compact. Hence the usual Sobolev embedding theorems (for both functions and tensors) are all applicable on surfaces $(\Si,g)$, $g\in \MMiota$, with constants that do not depend on the specific choice of metric $g\in \MMiota$, but only on $\iota_0>0$ (and of course the exponents of the involved Sobolev spaces and as usual the topology of $\Si$). We also stress that all Sobolev spaces and norms are to be understood with respect to the corresponding metric $g$ (at the relevant time if $g$ is time dependent) unless specified otherwise. We also note that while Sobolev space $H^s(\pSi,g)$ and $H^s(\Si,g)$ of any order $s$ are defined for all metrics $g\in \MM^{m+1}$, we will only ever consider such spaces for exponents $s\leq m$ to ensure that the spaces that we obtain are independent of the choice of metric $g\in \MM^{m+1}$ and note that the dependence on $g$ of the corresponding norms can be controlled as described in Lemma \ref{lemma:metric_variations}.
\end{rmk}
We furthermore recall that for such hyperbolic metrics on $\Si$ a neighbourhood $\Col(\si_i,g)$ of each boundary curve $\si_i$ of $\Si$ is isometric to the hyperbolic cylinder
\beq\label{eq:collar}
    \big([0,X(\ell))\times S^1,\rho_{\ell}^2 (\dd s^2+\dd \theta^2)\big), 
\eeq
for $\ell=L_{g}(\si_i)$ the length of $\si_i$ and with $\si_i$ corresponding to $\{0\}\times S^1$, where
\beq \label{eq:collar-X-rho}
    X(\ell)=\tfrac{2\pi}{\ell}[\tfrac\pi2-\arctan(\sinh\tfrac{\ell}{2})], \qquad \rho_\ell(s)=\tfrac{\ell}{2\pi} [ \cos(\tfrac{\ell s}{2\pi})]^{-1}.
\eeq
By passing to the double $\hat\Si$ and using the corresponding version of the collar lemma for all simple closed geodesics, or exploiting the above Remark \ref{rmk:compact-moduli}, it is furthermore easy to see that the lengths $L_{g}(\si_i)$ of the boundary curves are bounded both away from zero and from above uniformly for all metrics $g\in\MMiota$. 

As a result, we obtain the following uniform control on the behaviour of the metric on the neighbourhood $\Col(\pSi,g) := \bigcup_i\Col(\si_i,g)$ of the boundary of such surfaces $(\Si,g)$, $g\in \MMiota$.

\begin{rmk}\label{collar:uniform} 
    For any $g\in \MMiota$, $\iota_0>0$ any given number, estimates of the form 
    \beq
        0<c\leq X(\ell_i)\leq C \text{ and } 0<c \leq \rho_{\ell_i}\leq C \text{ on } [0,X(\ell_i)), \quad  \ell_i:= L_g(\si_i)
    \eeq
    are valid for constants $c>0$ and $C>0$ that only depend on $\iota_0$ (and as usual the topology of $\Si$), while for any $k\in \N$ the estimate $\norm{\rho_{\ell_i}}_{C^k(\Col(\si_i,g))} \leq C$ holds true for a constant that additionally depends on $k$. 
    \\
    For each such surface, we can hence in particular choose a cut-off function $\varphi\in C_c^\infty(\Col(\pSi,g))$ with $\varphi\equiv 1$ in a neighbourhood of $\pSi$ which is so that $\norm{\varphi}_{C^k(\Col(\pSi,g))}\leq C=C(\iota_0,k)$ and so that $\dist_g(\supp(\na \varphi), \pSi)\geq c(\iota_0)>0$. 
    \\
    On occasion we will also want to extend the unit tangent $\tau_g$ and the unit normal $\nu_g$ from $(\pSi, g)$ to this collar neighbourhood and will always do this by choosing the vector fields which are given in collar coordinates as the fixed rescalings of the standard coordinate vector fields
    \beq\label{def:extension-tau-nu}
        \tau:=\rho_{\ell_i}(0)^{-1} \frac{\partial}{\partial \theta} \text{ and } \nu:= - \rho_{\ell_i}(0)^{-1} \frac{\partial}{\partial s}.
    \eeq
    We note that while this does not yield vector fields with unit length, this choice of extension is convenient since it implies $\pnu^2U+\ptau^2 U=0$ for any function $U$ which is harmonic with respect to $g$.  
\end{rmk}

We note that for (orientable) surfaces $\Si$ different from the disc or the cylinder, it has already been pointed out by Tromba in \cite{Tromba} that the map $g\mapsto H(g)$ is not integrable in the sense of Frobenius's theorem; that is, there exists no submanifold $\MM_H$ of $\MM(\Si)$ (or of $\MMiota^{m+1}(\Si)$ for any finite $m\ge 2$) whose tangent spaces $T_g\MM_H$ coincide with the corresponding horizontal spaces 
\beq\label{def:H(g)} 
    H(g):= \{h\in \Gamma^{sym}(T^*\Si\otimes T^*\Si): \text{tr}_g(h)=0, \quad \Div_g(k)=0, \quad h(\nu_g,\tau_g)\vert_\pSi=0\}.
\eeq

In practice this means that curves which are horizontal, i.e. whose velocity $\pt g$ at each time is given by an element of the corresponding horizontal space $H(g(t))$ and which start at a given initial metric $g_0$ will not be constrained to a finite dimensional manifold. By contrast, in the case where  $\Si$ is a cylinder such horizontal curves are constrained to an explicit $1$-dimensional submanifold of $\MM$, compare Section \ref{subsec:cylinder} below. 

We also stress that the equation for  the metric component of our coupled flow \eqref{eq:flow} cannot be viewed as an equation on Teichm\"uller space but that it is important to keep track of the metric $g(t)$ itself as different representatives $\tilde g(t) =f_t^*g(t)$ of the same curve in Teichm\"uller space lead to different PDEs for the map component.

For the proof of short-time existence we will be working with metrics which are contained in a small neighbourhood of a given $g_0\in \M^{m+1}$ which is chosen as in the following lemma. 
\begin{lemma}\label{lemma:U-basics}
    For every $\iota_0>0$ and $m\geq 2$ there exists $r_0=r_0(\iota_0,m)>0$ so that for each $g_0\in \MMtwoiota^{m+1}$ 
    all metrics in the neighbourhood
    \beq\label{def:U}
        \UU=\UU_{g_0,r_0}^{m+1}:=\{g\in  \MM^{m+1}: \dist_{H^{m+1}}(g_0,g)\leq r_0\}
    \eeq
    of $g_0$ are uniformly equivalent in the sense that 
    \beq\label{est:U-equiv} 
        \frac14 g\leq \tilde g\leq 4 g \text{ for all } g,\tilde g\in \UU
    \eeq
    and hence in particular so that 
    \beq\label{est:U-inj-bound} 
        \inj(\Si,g)\geq \half \inj(\Si,g_0) \geq \iota_0 \qquad\text{ for all } g\in \UU.
    \eeq
    Furthermore, all metrics in $\UU$ induce uniformly equivalent Sobolev norms in the sense that there exists a constant $C=C(m,\iota_0)$ so that the estimates 
    \beq\label{est:equiv-Hm-new}
        \norm{v}_{H^k(\Si,g)}\leq C\norm{v}_{H^k(\Si,\tilde g)} \qquad \text{ and }\qquad  \norm{h}_{H^k(\Si,g)}\leq C\norm{h}_{H^k(\Si,\tilde g)} 
    \eeq
    hold true for all $g,\tilde g\in \UU$, all $0\leq k\leq m+1$, all functions $v\in H^k(\Si,\R^n)$ and all $(0,2)$-tensors $h$ on $\Si$ whose coefficients are in $H^k$. 
\end{lemma}

\begin{rmk}\label{rmk:half-energy-uniform}
    As explained in Section \ref{subsec:metric-useful}, such metrics also yield uniformly equivalent half-energies, namely are so that the estimate 
    \beq
        E_\hf(v,g) \leq CE_\hf(v,\tilde{g}) \text{ for all } g,\tilde g\in \UU_{g_0,r_0}^{m+1}
    \eeq
    holds true for all $v\in H^\half(\pSi;\R^n)$ for a constant $C=C(\iota_0)$.
\end{rmk}

\begin{proof}[Proof of Lemma \ref{lemma:U-basics}]
    Let $\UU$ be as defined above for a number $r_0>0$ that we determine below. Given any $g_1\in \UU$ we can consider a curve of metrics $g(t), t\in [0,1]$, with $g(0)=g_0$ and $g(1)=g_1$ so that $\int_0^1\norm{\pt g(t)}_{H^{m+1}(\Si,g(t))}\dd t<2r_0$ and we let $t_0\in (0,1]$ be the maximal time so that \eqref{est:U-inj-bound} holds for all metrics $g(t)$ with $t\in [0,t_0]$. 

    As we have pointwise bounds of 
    \begin{equation}
        \abs{\tfrac{\dd}{\dd t} \abs{X}_{g}^2}=\abs{(\pt g)(X,X)}\leq \abs{\pt g}_{g} \abs{X}_{g}^2\leq \norm{\pt g}_{L^\infty(\Si,g)} \abs{X}_{g}^2
    \end{equation}
    for any vector field $X$ on $\Si$, and as the Sobolev-embedding theorem $H^2(\Si,g)\hookrightarrow L^\infty(\Si,g)$ is valid with the same constant $C_0=C_0(\iota_0)$ for all metrics $g(t)$ with $t\in [0,t_0]$, we can integrate the resulting estimate of $\abs{\tfrac{\dd}{\dd t} \log(\abs{X}_{g}^2)}\leq  \norm{\pt g}_{L^\infty(\Si,g)}\leq C_0 \norm{\pt g}_{H^{m+1}(\Si,g)}$ over any subinterval of $[0,t_0]$ and use that $\int_0^1\norm{\pt g}_{H^{m+1}(\Si,g)}<2r_0$ to deduce that 
    \begin{equation}
        g(t)\leq e^{2C_0 r_0} g(\tilde t) \leq 2 g(\tilde t)\text{ for all }t,\tilde t\in [0,t_0],
    \end{equation}
    where the last inequality holds provided $r_0>0$ is small enough. 
    
    We conclude that the lengths of any curve $\si$ in $\Si$ with respect to two such metrics are related by $L_{g(t)}(\si)\leq \sqrt{2} L_{g(\tilde t)}(\si)$. Since the injectivity radius of hyperbolic surfaces $(\Si,g)$ is given by half of the length of the shortest curve which is either closed and non-contractible or so that it connects two points of $\pSi$, we can hence deduce that $\inj(\Si,g(t))\leq \sqrt{2} \inj(\Si,g(\tilde t))$ for all $t,\tilde t\in [0,t_0]$. In particular, 
    $$\inj(\Si,g(t_0))\geq \tfrac{1}{\sqrt 2} \inj(\Si,g_0) > \tfrac{1}{ 2} \inj(\Si,g_0),$$
    so as $t_0$ was chosen as the maximal number $t_0\leq 1$ so that \eqref{est:U-inj-bound} holds on $[0,t_0]$ we must have $t_0=1$. 
    
    The estimates on $g(t)$ and $\inj(\Si,g(t))$ obtained above are hence in particular applicable for $g_1=g(1)$ which immediately yields the first two claims \eqref{est:U-equiv} and \eqref{est:U-inj-bound} of the lemma. 
    
    We can then use that pointwise estimates of the form 
    \begin{equation}\label{est:ddt_grad_k}
    	\abs{\tfrac{\dd}{\dd t} \na_{g(t)}^k v}_{g(t)}\leq  C\sum_{i+j\leq k} \abs{\na_{g(t)}^i \pt g(t)}_{g(t)}\cdot \abs{ \na_{g(t)}^j v}_{g(t)}, \qquad C=C(k)
    \end{equation}
    hold true for every $k\in \{0,\ldots, m\}$, for all sufficiently smooth $v:\Si\to \R^n$ and every curve of hyperbolic metrics. As \eqref{est:U-inj-bound} and Remark \ref{rmk:compact-moduli} ensure that the standard Sobolev embeddings $H^2(\Si,g)\hookrightarrow L^\infty(\Si,g)$ and $H^1(\Si,g)\hookrightarrow L^4(\Si,g)$ are applicable with the same constant for all $g\in \UU$, and as the change in the volume element is controlled by $\norm{\pt g}_{L^\infty(\Si,g)}$, we can hence bound
    \beq\label{est:ddt-Hm-new}
        \abs{\tfrac{\dd}{\dd t} \norm{v}_{H^k(\Si,g))}^2} \leq C \norm{\pt g}_{H^{m+1}(\Si,g)}^2 \norm{v}_{H^k(\Si,g)}^2 \text{ for some } C=C(m,\iota_0,\Si)
    \eeq
    for all $k\in\{0,\ldots,m+1\}$ and all functions $v\in H^k(\Si;\R^n)$ as claimed. Integrated over $t$ this immediately yields the claimed equivalence of norms of functions, and we note that the same argument also applies for tensors. 
\end{proof}

\subsection{Analysis of horizontal curves of hyperbolic metrics}
The key tool needed for the proof of Proposition \ref{prop:step1} is the following Lipschitz property of the projection $P^H$.

\begin{lemma}\label{lemma:projection}
    For any $\iota_0>0$ and $m\geq 2$, there exist constants $C<\infty$ and $r_1\in (0,r_0)$, $r_0>0$ as in the above Lemma \ref{lemma:U-basics}, so that for any $g_0\in\MMtwoiota^{m+1}$ the following assertions hold true for every metric $g$ in the neighbourhood $\UU=\UU_{g_0,r_1}^{m+1}:=\{g\in  \MM^{m+1}: \dist_{H^{m+1}}(g_0,g)\leq r_1\}$ of $g_0$ in $\MM^{m+1}$.

    Let $P_g^H$ be the $L^2(\Si,g)$-orthogonal projection from the space $\Gamma^{sym}_{L^2(\Si,g)}(T^*\Si\otimes T^*\Si)$  of symmetric $(0,2)$-tensors on $\Si$ with finite $L^2(\Si,g)$-norm to the horizontal space $H(g)$. Then for every $h\in\Gamma^{sym}_{L^2(\Si,g)}(T^*\Si\otimes T^*\Si)$ and all $g,\tilde g\in \UU$  we can bound 
    \beq\label{est:PH_bounded} 
        \norm{P_g^H(h)}_{H^{m+1}(\Si,g_0)}\leq C \norm{h}_{L^1(\Si,g_0)} 
    \eeq
    as well as
    \beq\label{est:PH_Lip}
        \norm{P_{ g}^H(h)-P_{\tilde g}^H(h)}_{H^{m+1}(\Si,g_0)}\leq C \dist_{H^{m+1}}(g,\tilde g) \norm{h}_{L^1(\Si,g_0)},
    \eeq
    and for any $C^1$-curve  of metrics $g_\eps$ in $\UU$, we have
    \beq\label{est:PH_diff}
        \norm{\peps P_{g_\eps}^H(h)}_{H^{m+1}(\Si,g_0)}\le C\norm{\peps g_\eps}_{H^{m+1}(\Si,g_0)} \norm{h}_{L^1(\Si,g_0)}.
     \eeq
\end{lemma}
We note that \eqref{est:PH_bounded} in particular ensures that the projection has a unique continuous extension to the space $\Gamma^{sym}_{L^1(\Si,g)}(T^*\Si\otimes T^*\Si)$ of tensors with finite $L^1$ norm, and remark that the above estimates remain valid if we compute some or all of the norms with respect to $g$ since Lemma \ref{lemma:U-basics} ensures that the metrics in $\UU$ and the induced $H^{m+1}$ norms are uniformly equivalent.

\begin{proof}[Proof of Lemma \ref{lemma:projection}]
    As in the introduction we let $\hat \Si$ be the Schottky double of $\Si$ and let $\hat \MM^{m+1}$ be the manifold of all hyperbolic metrics on $\hat\Si$ whose coefficients are in the Sobolev space $H^{m+1}$ with respect to the fixed coordinate charts on $\hat \Si$. Given $\hat g\in \hat \MM^{m+1}$ we let $\hat H(\hat g)=\hat H(\hat\Si,\hat g)$ be the space of all symmetric $(0,2)$-tensors on $\hat\Si$ which are trace free and divergence free and let $P_{\hat g}^{\hat H}$ be the $L^2(\hat\Si,\hat g)$-orthogonal projection from the space of $(0,2)$-tensors to $\hat H$. 
    
    We now make use of Lemma 2.9 from \cite{R-existence}, which ensures that for every $\hat g\in \hat \MM^{m+1}$ there exists a neighbourhood $\hat \UU$ of $\hat g$ in $\hat\MM^{m+1}$ and a constant $C>0$ so that $P_{\hat g}^{\hat H}$ satisfies the analogues of the claims on $P_g^H$ made in Lemma \ref{lemma:projection}. While the results in \cite{R-existence} are stated using a different notion of $H^{m+1}$ norm, there computed using a fixed set of coordinate chart, we note that for each \text{fixed} $g_0$ this norm is equivalent to the $H^{m+1}(\Si,g_0)$ norm which is computed using the Levi-Civit\`a connection we use in the present paper. The results of \cite{R-existence} hence guarantee that there exists a neighbourhood $\hat \UU$ of $\hat g_0$ and a constant $C=C(\hat g_0,m)$ so that for any $L^2$-tensor $\hat h$ on $\hat\Si$, any $\hat g_1, \hat g_2\in \hat \UU$ and any $C^1$-curve of metrics $\hat g_\eps$ in $\hat\UU$ we have
    \begin{align}\label{est:proj_hat_1}
        \norm{P_{\hat g_1}^{\hat H}(\hat h)}_{H^{m+1}(\hat\Si,\hat g_0)} &\leq C \norm{\hat h}_{L^1(\hat \Si,\hat g_0)},\\
        \norm{P_{\hat g_1}^{\hat H}(\hat h)-P_{\hat g_2}^{\hat H}(\hat h)}_{H^{m+1}(\hat\Si,\hat g_0)} &\leq C \dist_{H^{m+1}}(\hat g_1,\hat g_2) \norm{\hat h}_{L^1(\hat \Si,\hat g_0)},\label{est:proj_hat_2}
    \end{align}
    as well as 
    \beq\label{est:proj_hat_3}
        \norm{\peps P_{\hat g_\eps}^{\hat H}(\hat h)}_{H^{m+1}(\hat\Si,g_0)}\le C\norm{\peps\hat g_\eps}_{H^{m+1}(\hat\Si,g_0)}\norm{\hat h}_{L^1(\hat \Si,\hat g_0)};
    \eeq
    compare also Proposition 2.1 of \cite{R-existence}. We now observe that since these claims are invariant under pull-back by diffeomorphism, and since the subset of moduli space that corresponds to metrics with $\inj(\hat \Si,g)\geq 2\hat \iota_0>0$ is compact, we indeed obtain that the above estimates \eqref{est:proj_hat_1}-\eqref{est:proj_hat_3} hold true on a ball $\hat \UU_{\hat g_0,\hat r_1} := \{\hat g\in \hat \MM^{m}: \dist_{H^{m+1}}(\hat g,\hat g_0) \leq \hat r_1\}$ whose radius $\hat r_1>0$ only depends on $m$ and a lower bound $2\hat \iota_0$ on $\inj(\hat\Si, \hat g_0)$ and with a constant $C=C(\hat \iota_0,m)$.  

    To deduce the claims of Lemma \ref{lemma:projection} from these facts we can now argue as follows: given any $h\in \Gamma^{sym}(T^*\Si\otimes T^*\Si)$ we extend $h$ to the disjoint union $\hat\Si=\Si\times \{-1,1\}$ by simply setting $\hat h(p,\pm 1)=h(p)$ for each $p\in\Si$. Away from $\partial \Si$, and hence almost everywhere, this yields a well defined tensor $\hat h$ on the double $\hat\Si$ which is even with respect to $\pSi$ in the sense that 
    \beq\label{eq:even}
        \hat h(p,1)(w_1,w_2)=\hat h(p,-1)(w_1,w_2) \text{ for } p\in \Si\setminus \pSi \text{ and } w_{1,2}\in T_p\Si{\hat =}T_{(p,\pm 1)}\hat\Si.
    \eeq
    In this way we can extend any $L^2$-tensor field on $\Si$ to an $L^2$-tensor field on the double with even symmetry. We do not claim that for smooth $h$ this extension always yields a smooth, or even just continuous, tensor field $\hat h$ as some of the components of $\hat h$ can jump across $\pSi$ in $\hat\Si$.

    To see this, we can use the description and coordinates of the collar neighbourhoods $\Col(\si_i)$ of the boundary curves $\si_i$ which we recalled in Section \ref{subsec:metric-basic}. On the corresponding neighbourhood $\Col(\si_i)\times\{-1,+1\}/\sim$ of $ \si_i$ in the double $(\hat\Si,\hat g)$ we can then work in the coordinates that we obtain by identifying $((s,\theta),\pm 1)$ with $(\pm s,\theta)$. The components of the extended tensor $\hat h$ in these coordinates at $(\pm s,\theta)$ are so that $\hat h_{ss}(\pm s,\theta)=h_{ss}(s,\theta)$ and $\hat h_{\theta\theta}(\pm s,\theta)=h_{\theta\theta}( s,\theta)$ are continuous across $\pSi$ whenever $h$ is continuous, but so that $\hat h_{s\theta}(\pm s,\theta)=\pm h_{s\theta}(s,\theta)$ jumps unless $h_{s\theta}$ vanishes on $\partial \Si$, that is, unless $h(\nu_g,\tau_g)\vert_\pSi\equiv 0$. 
    
    In the special case where we extend the metric tensor $g$ itself in this way, the resulting tensor $\hat g$ is not only continuous but indeed smooth across $\pSi$ since $\rho_\ell(s)=\rho_\ell(-s)$.
    
    Similarly, the extension of any tensor $h\in H(g)$ yields an element of $\hat H(\hat g)$, so in particular a smooth tensor. To see this, we first recall that such a tensor can be written as   $h=\Rea(\Psi)$ for a  holomorphic quadratic differential $\Psi$ on $\Si$ which is real on the boundary in the sense that in collar coordinates $(s,\theta)$ near any $\si_i$ we have $\Psi=\psi (\dd s+i \dd\theta)^2$ for a holomorphic $\psi:[0,X(\ell))\times S^1$ which is real on the circle $\{0\}\times S^1$. The extended tensor $\hat h$ is then given by $\Rea(\hat\psi (\dd s+i\dd\theta)^2)$ for $\hat\psi$ defined by $\Rea(\hat\psi(\pm s,\theta))=\Rea(\psi(s,\theta))$ and $\Ima(\hat\psi(\pm s,\theta))=\pm\Ima(\psi(s,\theta))$. As $\Ima(\psi(0,\theta))=0$ on $\pSi$ this function $\hat\psi$ is holomorphic on $\hat\Si$, and hence in particular smooth across $\partial \Si$, and the extended tensor $\hat h=\Rea(\hat\psi \dd z^2)$ is an element of $\hat H(\hat g)$ as elements of this space can be equivalently characterised as real parts of holomorphic quadratic differentials. 
    
    We now claim that to project a given $h\in \Gamma^{sym}_{L^2}(T^*\Si\otimes T^*\Si)$ onto $H(g)$ we can equivalently first extend $h$ to the even $L^2$-tensor $\hat h$ on the double as described above, then project this tensor onto $\hat H$, and, finally, again restrict $\hat h$ to $\Si$. That is, we claim that
    \beq\label{eq:relation_proj}
        P_g^H(h)=P_{\hat g}^{\hat H}(\hat h)\vert _\Si.
    \eeq
    To see that this holds true, we split $\hat H=\hat H_{even}\oplus\hat H_{odd}$ into the tensors with even or odd symmetry with respect to $\partial \Si$, respectively, that is, into tensors on $\hat\Si$ corresponding to tensors $h$ on $\Si\times\{-1,1\}$ with $h(p,1)=h(p,-1)$ for all $p$ or $h(p,1)=-h(p,-1)$ for all $p$, respectively. As discussed above, the extension of any element of $H(g)$ yields an element of $\hat H_{even}$ and conversely the restriction of elements of $\hat H_{even}$ to $\Si$ is contained in $H(g)$ since the $(s,\theta)$ component of such smooth even tensors on $\hat \Si$ must vanish on $\pSi$. As the extension $\hat h$ of any $h\in  \Gamma^{sym}_{L^2}(T^*\Si\otimes T^*\Si)$ is even we hence deduce that $P_g^H(h)=P_{\hat g}^{\hat H_{even}}(\hat h)\vert_{\Si}$. Combined with the fact that being even forces the extended tensor $\hat h$ to be $L^2(\hat\Si,\hat g)$-orthogonal to $\hat H_{odd}$, this gives \eqref{eq:relation_proj}.
    
    As choosing $r_1:=\frac{1}{\sqrt{2}}\hat r_1$ ensures that the extension of metrics $g\in \UU_{g_0,r_1}$ is contained in the neighbourhood $\hat \UU_{\hat g_0,\hat r_1}$ of $\hat g_0$ on which \eqref{est:proj_hat_1}, \eqref{est:proj_hat_2} and \eqref{est:proj_hat_3} hold, we can hence immediately deduce the claims of Lemma \ref{lemma:projection} from these properties of the projection $P^{\hat H}$ that were proven in \cite{R-existence}. 
\end{proof}

Based on these Lipschitz estimates we can now deduce

\begin{lemma}\label{lemma:simple-ode}
    For any $m\ge 2$ and $\iota_0>0$  there exist constants $C,\delta_1>0$ so that the following holds true. Let $g_0\in \MM^{m+1}$ be any metric with $\inj(\Si,g_0)\geq 2\iota_0$, let $h\in C^1([0,T]; \Gamma^{sym}_{L^1(\Si,g_0)}(T^*\Si\otimes T^*\Si))$ for some $T>0$. Then there is a unique solution $g \in C^{1}([0,T_1];\MM^{m+1})$ of 
    \beq\label{eq:simple-ode}
        \pt g=P_g^H(h) \text{ with } g(0)=g_0
    \eeq
    which is defined and remains in the neighbourhood $\UU=\UU_{g_0,r_1}$ of $g_0$ where both Lemmas \ref{lemma:U-basics} and \ref{lemma:projection} are applicable at least until $T_1:=\min(T, \de_1/M_1)$ for $M_1$ chosen so that $\sup_{[0,T]} \norm{h}_{L^1(\Si,g_0)}\leq M_1$. Furthermore, this solution satisfies 
    \beq\label{est:basic-ode-est}
        \norm{\pt g(t)}_{H^{m+1}(\Si,g_0)}\leq C \norm{h(t)}_{L^1(\Si,g_0)} \text{ for all } t\in [0,T_1]
    \eeq
    and the map $h\mapsto g$ is Lipschitz, in the sense that the solutions $g_{1,2}$ of \eqref{eq:simple-ode} to tensors $h_{1,2}$ with $\sup_{[0,T]} \norm{h_{1,2}}_{L^1(\Si,g_0)}\leq M_1$ for some given $M_1$ satisfy 
    \beq\label{est:lip-g-simple}
        \dist_{H^{m+1}}(g_1(t),g_2(t))\leq C t \norm{h_1-h_2}_{L^\infty([0,T_1];L^1(\Si,g_0))} \text{ for all } 0\leq t\leq T_1.
    \eeq
\end{lemma}
\begin{proof}[Proof of Lemma \ref{lemma:simple-ode}]
    Let $g_0$, $\UU=\UU_{g_0,r_1}^{m+1}$, $h$ and $M_1$ be as in the lemma and let $T_1:=\min(T, \de_1/M_1)$ for $\de_1>0$ to be determined below. 
    
    Given any curve of metrics $g\in C^0([0,T_1];\UU)$ we define
    \beq\label{eq:tilde-g-ODE}
        G_g(t)=g_0+\int_0^t P^H_{g}(h(t'))\dd t'
    \eeq
    and note that a sufficiently small choice of  $\de_1=\de_1(\iota_0,m)>0$  ensures that $G_g(t)\in \UU$ for all  $t\in [0,T_1]$ as the estimate \eqref{est:PH_bounded} of Lemma \ref{lemma:projection} allows us to bound 
    \beq\label{est:dist-simple-ode}
        \dist_{H^{m+1}}(g_0,G_g(t))\leq C \int_0^t\norm{h(t')}_{L^1(\Si,g_0)} \dd t'\leq CM_1t
    \eeq
    for as long as $G_g(t)$ is in $\UU$, hence ensuring that $\dist_{H^{m+1}}(g_0,G_g(t))$ remains strictly less than $r_1$ on all of $[0,T_1]$ if $\de_1<C^{-1}r_1$. 

    To obtain the desired solution of \eqref{eq:simple-ode} we now want to argue that a sufficiently small choice of $\de_1$ ensures that this map $g\mapsto G_g$, which we have just established is mapping  $C^0([0,T_1];\UU)$ to itself, is a contraction. 
      
    Given $g_{1,2}\in C^0([0,T_1],\UU)$  we can always choose metrics $g_s(t)\in \UU, t\in [0,T_1], s\in [1,2]$, so that for each $t$ the curve $s\mapsto g_s(t)$ is continuously differentiable,  interpolates between  $g_1(t)$ and $g_2(t)$ and is so that $\int_1^2\norm{\p_s g_s(t)}_{H^{m+1}(\Si,g_s(t))} \dd s\leq 2\sup_{[0,T_1] }\dist_{H^{m+1}}(g_1,g_2)$,  while for each $s$ the function $t\mapsto g_s(t)$ is continuous away from a finite set of ($s$ independent) times $t_i$.
     
    As the estimate \eqref{est:dist-simple-ode} is applicable also for such piecewise continuous curves of metrics $t\mapsto g(t)\in \UU$, the curves of metrics $t\mapsto G_s(t):=G_{g_s}(t)$, $s\in[1,2]$, which satisfy \eqref{eq:tilde-g-ODE} for $g_s$ instead of $g$ are again contained in $\UU$ for all $t\in [0,T_1]$. We can hence exploit the equivalence of the induced $H^{m+1}$ and $L^1$ norms and apply \eqref{est:PH_diff} to bound 
    \beqas
        \norm{\ps G_s(t)}_{H^{m+1}(\Si,G_s(t))}&\leq  C\norm{\ps G_s(t)}_{H^{m+1}(\Si,g_0)}\leq C\int_0^t \norm{\ps \pt G_s(t')}_{H^{m+1}(\Si,g_0)} \dd t'\\
        &=C\int_0^t \norm{\ps P_{g_s(t')}(h(t'))}_{H^{m+1}(\Si,g_0)} \dd t'\\
        &\leq CM_1\int_0^t\norm{\ps g_s(t')}_{H^{m+1}(\Si,g_0)}\dd t' 
    \eeqas
    for every $t\in [0,T_1]$ and $s\in [1,2]$ and for a constant $C=C(\iota_0,m)$. 
    Upon integrating over $s\in [1,2]$ we hence deduce that 
    $$\sup_{[0,T_1]} \dist_{H^{m+1}}(G_{g_1},G_{g_2})\leq
    C\de_1\sup_{[0,T_1]} \dist_{H^{m+1}}(g_1,g_2)\leq     
    \half  \sup_{[0,T_1]} \dist_{H^{m+1}}(g_1,g_2),$$
    where the last estimate holds 
    after reducing $\de_1=\de_1(\iota_0,m)>0$ if necessary. Thus $g\mapsto G_g$ is indeed a contraction from $C^0([0,T_1],\UU)$ to itself. As $\UU$ is defined as a closed ball of a complete metric space, we hence obtain the existence of a unique solution $g\in C^1([0,T_1],\UU)$ of \eqref{eq:simple-ode} from Banach's fixed point theorem. We furthermore observe that the claimed estimate \eqref{est:basic-ode-est} is an immediate consequence of \eqref{est:PH_bounded}. 

    To prove the second part of the theorem we can argue similarly, except that we now consider the family $g_s$ of curves that solve \eqref{eq:simple-ode} for the family $h_s:=(2-s) h_1+(s-1)h_2$, $s\in [1,2]$, of tensors that interpolates between the two given tensors  $h_{1,2}$. The estimates \eqref{est:PH_bounded} and \eqref{est:PH_diff} from Lemma \ref{lemma:projection} are again applicable and now allow us to bound 
    \beqas
        \norm{\partial_t \partial_s g_s}_{H^{m+1}(\Si,g_0)} &= \norm{\partial_s P_{g_s}(h_s)}_{H^{m+1}(\Si,g_0)}\\
        &\leq C \norm{\partial_s g_s(t)}_{H^{m+1}(\Si,g_0)} M_1+C\norm{h_1-h_2}_{L^1(\Si,g_0)}.
    \eeqas
    This allows us to deduce that 
    $$\sup_{[0,t]}\norm{\ps g_s}_{H^{m+1}(\Si,g_0)}\leq C  \de_1\sup_{[0,t]}\norm{\ps g_s}_{H^{m+1}(\Si,g_0)}+Ct\norm{h_1-h_2}_{L^\infty([0,T];L^1(\Si,g_0))},$$
    for a constant $C=C(\iota_0,m)$. After reducing $\de_1 = \de_1(\iota_0,m) > 0$ further if necessary, we can absorb the first term of the right hand side into the left hand side and integrate the resulting estimate over $s\in [1,2]$ to obtain the final claim \eqref{est:lip-g-simple} of the lemma. 
\end{proof}

For the proof of Proposition \ref{prop:step1}, we can additionally use that the harmonic extension depends continuously on the domain metric as described in detail in Lemma \ref{lemma:metric_variations}. This immediately implies that the stress-energy tensors $k(v_g,g)$ satisfy the following Lipschitz estimates.

\begin{lemma}\label{lemma:extension-basic}
    Let $g_0\in \MMtwoiota^{m+1}$ for some $\iota_0>0$ and let $\UU=\UU^{m+1}_{g_0,r_0}\subset\MM^{m+1} $ be the neighbourhood of $g_0$ considered in Lemma \ref{lemma:U-basics} above. Then for any $v,\tilde v\in H^{\half}(\pSi,g_0)$ and any $g,\tilde g\in \UU$, we can bound 
    \beq\label{est:k-L1-energy}
        \norm{k(v_g,g)}_{L^1(\Si,g_0)}\leq C \Eh(v,g_0)
    \eeq
    and control the difference between the corresponding stress-energy tensors by
    \beqa\label{est:basic_3}
        \norm{k(v_g,g)-k(\tilde v_{\tilde g},\tilde g)}_{L^1(\Si,g_0)} &\leq C\norm{(v-\tilde v)_{g_0}}_{H^1(\Si,g_0)}\big(\Eh(v,g_0)+ \Eh(\tilde v,g_0)\big)^\half\\
        &\quad +C\dist_{H^{m+1}}(g,\tilde g)\cdot \Eh(v,g_0)
    \eeqa
    for a constant $C$ that only depends on $\iota_0$. 
\end{lemma}

We are now finally in a position to complete the proof of Proposition \ref{prop:step1} for surfaces $\Si$ of general type.

\begin{proof}[Proof of Proposition \ref{prop:step1}]
    Let $\iota_0>0$ and $m\geq 2$  be any fixed numbers, let $g_0\in \MMtwoiota^{m+1}$ and let $\UU=\UU_{g_0,r_1}^{m+1}$ be the neighbourhood of $g_0$ obtained in Lemma \ref{lemma:projection}, which we recall is so that Lemmas \ref{lemma:U-basics} and \ref{lemma:extension-basic} are also applicable. 
    
    Given $v\in X_2(T)$, for $X_2(T)$ as defined in \eqref{def:Xm}, we let $\Eupp$ be so that \eqref{ass:Ev-apriori} holds and note that  \eqref{est:k-L1-energy} ensures that 
    $$\norm{k(v(t)_g,g)}_{L^1(\Si,g_0)}\leq C \Eupp\text{ for every } t\in [0,T] \text{ and every }g\in \UU,$$
    where here and in the following constants $C$ only depend on $\iota_0$ and $m$.

    Setting $\de:=\de_1/C$ for this constant $C$ and the number $\de_1>0$ obtained in Lemma \ref{lemma:simple-ode}, we deduce from this lemma that for every  $g\in C^0([0,T_0];\UU)$, $T_0:=\min(T,\de/\Eupp)$, the solution $G_{v,g}$ of 
    \beq\label{def:tilde-g-proof-of-lemma} 
        \pt G_{v,g}=P_{G_{v,g}}(k(v_g,g)) \text{ with } \tilde g(0)=g_0
    \eeq
    is defined and remains in $\UU$ on the whole interval $[0,T_0]$. The Lipschitz estimate \eqref{est:lip-g-simple} obtained in this lemma, combined with \eqref{est:basic_3}, furthermore allows us to see that curves of metrics $G_{1,2}=G_{v,g_{1,2}}$, obtained from \eqref{def:tilde-g-proof-of-lemma} for the fixed map $v$ and for curves of metrics $g_{1,2}\in C^0([0,T_0];\UU)$, satisfy 
    \beqas
        \sup_{[0,T_0]} \dist_{H^{m+1}}(G_1,G_2)&\leq CT_0  \norm{k(v_{g_1},g_1)-k(v_{g_2},g_2)}_{L^\infty([0,T_0];L^1(\Si,g_0))}\\
        &\leq  CT_0 \Eupp\sup_{[0,T_0]} \dist_{H^{m+1}}(g_1,g_2)\leq C\de \sup_{[0,T_0]} \dist_{H^{m+1}}(g_1, g_2)\\
        &\leq \half \sup_{[0,T_0]} \dist_{H^{m+1}}(g_1,g_2)
    \eeqas
    where the last estimate holds after reducing $\de=\de(\iota_0)>0$ if necessary. Banach's Fixed point theorem hence yields the existence of a unique solution $g_v\in C^1([0,T_0];\UU)$ of \eqref{eq:metric_lemma}. 
    
    Thanks to \eqref{est:PH_bounded} and \eqref{est:k-L1-energy}, we can furthermore bound
    $$\norm{\pt g_v(t)}_{H^{m+1}(\Si,g_0)}\leq C \norm{k(v_{g(t)},g(t))}_{L^1(\Si,g_0)}\leq C\Eh(v(t),g_0)$$ 
    as claimed in \eqref{est:ptg_apriori_1}, while a combination of \eqref{est:PH_bounded}, \eqref{est:PH_diff}, \eqref{est:k-L1-energy} and \eqref{est:basic_3} and  \eqref{ass:Ev-apriori} yields
    \beqas
        \norm{\pt g_v(t)-\pt g_v(\tilde t)}_{H^{m+1}(\Si,g_0)} &\leq \norm{\big(P_{g_v(t)}-P_{g_v(\tilde t)}\big)(k(v_{g_v},g_v)(t))}_{H^{m+1}(\Si,g_0)}\\
        &\quad +\norm{P_{g_v(\tilde t)}\big((k(v_{g_v},g_v)(t)-k(v_{g_v},g_v)(\tilde t)\big)}_{H^{m+1}(\Si,g_0)}\\
        &\leq  C \dist_{H^{m+1}}(g_v(t),g_v(\tilde t)) \Eupp +C \Eupp^\half \norm{(v(t)-v(\tilde t))_{g_0}}_{H^1(\Si,g_0)}
    \eeqas 
    for all $t,\tilde t\in [0,T_0]$. As the estimate \eqref{est:ptg_apriori_1}, which we have proven above, ensures that 
    \begin{equation*}
        \dist_{H^{m+1}}(g_v(t),g_v(\tilde t))\leq C \Eupp \abs{t-\tilde t},
    \end{equation*}
    this immediately yields the second a priori bound \eqref{est:ptg_apriori_2} claimed in the lemma. 

    Given $v$ and $\tilde v$ which satisfy \eqref{ass:Ev-apriori}  for some given $\hat E$, we can then combine \eqref{est:lip-g-simple} with \eqref{est:basic_3} to see that for any $t\in [0,T_0]$,
    \beqas
        \sup_{[0,t]}\dist(g_v,g_{\tilde v})&\leq Ct \norm{k(v_{g_v},g_v)-k(\tilde v_{g_{\tilde v}},g_{\tilde v})}_{L^\infty([0,t];L^1(\Si,g_0))}\\
        &\leq Ct\Eupp^\half \norm{(v-\tilde v)_{g_0}}_{L^\infty([0,T];H^1(\Si,g_0))} +Ct \Eupp \sup_{[0,t]}\dist_{H^{m+1}}(g_v,g_{\tilde v}).
    \eeqas
    After reducing $\de>0$ if necessary to ensure that $CT_0\Eupp \leq C\de\leq \half$, we can absorb the last term into the left hand side resulting in the claimed Lipschitz estimate \eqref{est:g_diff_apriori}. 

    Similarly, \eqref{est:PH_Lip}, \eqref{est:PH_bounded} and \eqref{est:basic_3} allow us to bound
    \beqas
        \norm{\pt(g_v-g_{\tilde v})}_{H^{m+1}(\Si,g_0)} &\leq  \norm{(P_{g_v}-P_{g_{\tilde v}})(k(v_{g_v},g_v))}_{H^{m+1}(\Si,g_0)}\\
        &\quad+\norm{P_{g_{\tilde v}}(k(v_{g_v},g_v)-k(\tilde v_{g_{\tilde v}},g_{\tilde v}))}_{H^{m+1}(\Si,g_0)}\\
        &\leq C\dist_{H^{m+1}}(g_v,g_{\tilde v}) \Eupp+C\Eupp^\half \norm{(v-\tilde v)_{g_0}}_{H^1(\Si,g_0)},
    \eeqas
    which, when combined with \eqref{est:g_diff_apriori}, yields the final claim \eqref{est:ptg_diff_apriori} of the lemma. 
\end{proof}

\subsection{Simplified analysis of the metric components for the cylinder}\label{subsec:cylinder}

We finally turn to the case where $\Si$ is a cylinder where the above arguments simplify significantly since the horizontal space $H(g)$ is not only one-dimensional, but also so that $g\mapsto H(g)$ is an \emph{integrable} distribution on the space $\MM(\Si)$ of flat unit area metrics with geodesic boundary curves. This ensures that every horizontal curve that starts at an initial metric $g_0\in \MM$ evolves in a one-dimensional integral manifold $\MM_H(g_0)\subset \MM(\Si)$. We can hence  either prove Proposition \ref{prop:step1} directly, using that the evolution of the metric reduces to an ODE, or simply observe that the above proof still applies but that it can be simplified in the following way. 

Given a flat unit cylinder $(\Si,g_0)$ with geodesic boundary curves, we can always introduce coordinates $(x,\theta)$ in which 
\beq\label{def:cyl-simple} 
    \Si= [-\pi,\pi]\times S^1 \text{ and } g_0= g_{a_0}:=(2\pi)^{-2} (a_0^{-1} \dd x^2+ a_0 \dd\theta^2),
\eeq
where $a_0$ is determined by the length $\ell_0=L_{g_0}(\si_1)=L_{g_0}(\si_2)$ of the boundary curves via $a_0=\ell_0^2$ and related to the injectivity radius by $\inj(g_0)=\frac12\min(a_0,a_0^{-1})^{1/2}$. It can then be easily checked that the one-dimensional submanifold 
\begin{equation*}
    \MM_H = \{g_a := (2\pi)^{-2} (a^{-1} \dd x^2+ a \dd\theta^2): a > 0\} 
\end{equation*}
of $\MM$ is an integral manifold of $g\mapsto H(g)$, i.e. that $T_g\MM_H=H(g)$ for all $g\in \MM_H$. One way to see this is to set $s:= a^{-1}x$ to obtain coordinates $(s,\theta)$ in which $g_a$ is conformal to the standard cylindrical metric, namely given by $a(2\pi)^{-2} (\dd s^2+\dd\theta^2)$, and use that in these coordinates each horizontal tensor can be written as $h=\Rea(\psi(s+i\theta)(\dd s+i\dd\theta)^2)$ for a function $\psi$ which is holomorphic and real on the boundary curves and hence given by $\psi\equiv c$ for some $c\in \R$. Thus, each such $h$ is given by 
\begin{equation*}
    h=c(\dd s^2-\dd\theta^2)=c(a^{-2} \dd x^2-\dd\theta^2)= -c(2\pi)^2\p_a g_a.
\end{equation*} 

Instead of working with metrics $g$ in a full $H^{m+1}$ neighbourhood of $g_0$ in $\MM^{m+1}$, the existence of such an integral manifold means that it suffices to consider metrics in a neighbourhood $\UU(g_{a_0})=\{g_a, \half a_0\leq a\leq 2 a_0\}$ of $g_0=g_{a_0}$ in this explicit one-dimensional set. These metrics trivially satisfy the equivalence properties stated in Lemma \ref{lemma:U-basics} while the properties of the projection onto $H(g)$ for metrics $g\in \UU(g_0)$ are an immediate consequence of the explicit formula 
\beqas
    P_{g_a}^H(h) &= \norm{\partial_a g_a}_{L^2(\Si,g_a)}^{-2} \inner{h}{\partial_a g_a}_{L^2(\Si,g_a)} \partial_a g_a\\
    &= \frac{a^2}{2(2\pi)^4}\inner{h}{a^{-2}\dd x^2 - \dd \theta^2}_{L^2(\Si,g_a)}\left(a^{-2}\dd x^2 - \dd \theta^2\right).
\eeqas
The proof of Proposition \ref{prop:step1} in this simpler case can hence again be obtained based on a fixed point argument as carried out above.

\subsection{Further properties of metrics in $\MM$ and the associated norms and harmonic extensions}\label{subsec:metric-useful}
In the next section we will often need to switch back and forth between norms computed with respect to different metrics. For this, it is helpful to record that while a change of the metric $g$ leads to a change of the harmonic extensions and of the Sobolev norms of maps, these changes are well controlled along horizontal curves of metrics as considered here, as we can combine the estimates on the $H^{m+1}$ norm for the velocity of such curves obtained above with the following basic lemma.  

\begin{lemma}\label{lemma:metric_variations}
	For any integer $m \geq 2$ and any $\iota_0>0$, there exists a constant $C$ so that the following holds true for any curve $g\in C^1([0,T],\MMiota^{m+1})$. 
    
    For any $j\in \{0,\ldots, m\}$ and any $v\in H^j(\Si)$ we can bound \begin{equation}\label{est:metric_variation_sobolev_norm}
        \abs{\tfrac{\dd}{\dd t}\norm{\nabla_{g}^j v}_{L^2(\Si,g(t))}^2} \leq C\norm{\p_t g}_{H^{m}(\Si,g(t))}\norm{v}_{H^j(\Si,g(t))}^2
	\end{equation}
    while for maps $v\in H^{j+\half}(\Si)$ we can also bound 
    \begin{equation}\label{est:boundary_norm_variation}
        \abs{\tfrac{\dd}{\dd t}\norm{\nabla_{g}^j v}_{L^2(\p\Si,g(t))}^2} \leq C\norm{\p_t g}_{H^{m+1}(\Si,g(t))}\sum_{i=0}^j\norm{\nabla_{g}^i v}_{L^2(\p\Si,g(t))}^2.
	\end{equation}
    Furthermore, the change of the harmonic extension of any $f \in H^{m-\hf}(\p\Si)$ is bounded by 
    \begin{equation}\label{est:he_ddt}
        \norm{\p_t [f_{g(t)}]}_{H^m(\Si,g(t))} \leq C\norm{\p_t g(t)}_{H^{m}(\Si,g(t))}\norm{f_{g(t)}}_{H^m(\Si,g(t))},
	\end{equation}
    and the change of the half-energy by
    \begin{equation}\label{est:g_change_E}
        \abs{\tfrac{\dd}{\dd t}E_\hf(f,g(t))}\leq C\norm{\p_t g}_{L^\infty(\Si,g(t))} E_\hf(f,g(t)).
	\end{equation} 
    In particular, if $M$ is chosen so that $\norm{\p_t g }_{H^{m}(\Si,g)}\leq M$ on $[0,T]$ then we can bound 
    \begin{equation}\label{est:harm-ext-timedep-gen} 
		\norm{f_{g(t_0)}}_{H^m(\Si,g(t_0))}^2
		\le e^{C\abs{t_1-t_0}M}\norm{f_{g(t_1)}}_{H^m(\Si,g(t_1))}^2 \text{ for all }t_{0,1}\in [0,T]
	\end{equation} 
	as well as \begin{equation}\label{est:metric_variation_sobolev_norm_integrated}
        \norm{v}_{H^m(\Si,g(t_0))}^2\le e^{C\abs{t_0-t_1}M}\norm{v}_{H^m(\Si,g(t_1))}^2 \text{ for all }t_{0,1}\in [0,T].
	\end{equation}
\end{lemma}

\begin{proof}[Proof of Lemma \ref{lemma:metric_variations}]
    We first note that the estimates \eqref{est:metric_variation_sobolev_norm} and \eqref{est:boundary_norm_variation} follow from the pointwise estimate \eqref{est:ddt_grad_k}, the fact that the change of the volume form $ \p_t \dd v_{g(t)} = \hf\mathrm{tr}(\p_t g)\dd v_g$  is controlled by $\norm{\pt g}_{L^\infty(\Si)}$ and that $H^m(\Si,g)$ embeds into both $W^{m-1,4}(\Si,g)$ and $L^\infty(\Si,g)$. 
   
    The estimate \eqref{est:g_change_E} for the half-energy is a direct consequence of the  formula \eqref{eq:energy_var_metric} for the variation of the half-energy and the estimate  \eqref{est:k-L1-energy} on the stress-energy tensor. It thus remains to control the evolution of the harmonic extension. For this, we can use that $\p_t (f_{g(t)}) = 0$ on $\p\Si$ to apply standard $L^2$ elliptic estimates and so bound
    \beqa
		\norm{\p_t (f_{g(t)})}_{H^m(\Si,g(t))}& \leq C\norm{\Delta_{g(t)} \p_t (f_{g(t)})}_{H^{m-2}(\Si,g(t))}\\
        &= C\norm{-\p_\eps \Delta_{g(t+\eps)} f_{g(t)} \big|_{\eps=0}}_{H^{m-2}(\Si,g(t))}\\
        & \leq  C\norm{\p_t g}_{H^{m}(\Si)}\norm{f_{g(t)}}_{H^{m}(\Si,g(t))}
	\eeqa
    where the second step follows since $\Delta_{g(t)} f_{g(t)} = 0$ for all $t$.
\end{proof}

\section{Analysis of the map component}\label{sect:map}
In this section we complete the analysis of the map component that is required to establish the existence and properties of solutions of our coupled flow \eqref{def:flow-abstract} for initial maps $u_0\in H^{m-\half}(\pSi;N)$, $m\geq 4$, claimed in Proposition \ref{prop:ep_system}. As outlined in Section \ref{sect:main} we will proceed in several steps, first considering the linearised and regularised equation \eqref{eq:ep_v_map} for given curves of maps $v$ and metrics $g$, see Section \ref{subsec:map-lin-reg}, then the non-linear, but still regularised, coupled system of equations \eqref{eq:flow-ep_coupled}, see Section \ref{subsec:reg-coupled},  before using the obtained uniform control on the solutions of these auxiliary problems to deduce the desired short-time existence results for our flow \eqref{eq:flow} in Section \ref{subsec:Hm-loc-exist}.

\subsection{Analysis of the linearised and regularised equation \eqref{eq:ep_v_map} for $u$} \label{subsec:map-lin-reg}
The main goal of this section is to establish the claims made in Proposition \ref{prop:ep_map_solve} about the existence and properties of solutions of \eqref{eq:ep_v_map} for given curves of maps $v:[0,T]\times \pSi\to N_\eta$ and metrics $g$ as described in this proposition. For this it will be convenient to write equation \eqref{eq:ep_v_map} for short as 
\beqs
    \p_t u = Q_{\ep,v} \p_{\nu_{g}}u_{g} \text{ for }  Q_{\ep,v} = -(\ep \Id+ P_{v}),
\eeqs
where we recall that the chosen extension of the projection is so that $P_\cdot=P_{\pi(\cdot)}$ on the closure of $N_\eta$. This in particular ensures that $P_\cdot$ is an orthogonal projection on $\bar N_\eta$ and hence allows us to view $x\mapsto (\eps \Id+P_x)$ as a smooth map from $\bar N_\eta$ into the set of invertible matrices and to use that also $x\mapsto (\eps \Id+P_x)^{-1}$ is smooth on $\bar N_\eta. $   

If it is clear from the context what $\eps$ and $v$ are we will  drop the indices and simply write  $Q=Q_{\ep,v}$ to lighten the notation and similarly, we will often drop references to the metric for geometric quantities and operators, such as volume forms $\dd v_g = \dd v$, $\dd s_g = \dd s$ and gradients $\nabla_g = \nabla$, in situations where we only work with one, possibly time dependent, curve of metrics and continue to use the convention that all Sobolev spaces of maps from $\Si$ (or $\pSi$) and their norms are to be considered with respect to $g$ at the relevant time unless indicated otherwise. 

As in \cite{S}, we want to use Galerkin's method with Steklov eigenfunctions to establish the existence of solutions to this linear equation \eqref{eq:ep_v_map}. To use this method, we first consider the problem for a time-independent metric $g$, to allow us to work with time-independent finite dimensional spaces of functions. To simplify the proof of the regularity of weak solutions, compare Lemma \ref{lemma:ep_new_sol}, we will furthermore initially also only consider maps $v$ which are independent of time, and will later remove both of these restrictions by using a time-discretisation argument. 

\begin{lemma}\label{lemma:ep_time_ind_weak_exist}
    Let $\ep\in (0, 1/2]$, $T>0$,  $g\in\MM^{3}$ and let $v\in H^{1/2}(\pSi;N_\eta)$. Then for any $u_0 \in  H^{\hf}(\p\Si;\R^n)$ there exists a unique weak solution $u:[0,T]\times \pSi\to \R^n$ of \eqref{eq:ep_v_map} with 
    $$u\in H^1([0,T];L^2(\p\Si))\cap L^2([0,T];H^1(\p\Si)) \text{ and } u_g\in  L^\infty([0,T];H^1(\Si)) 
    $$
    whose $L^2(\pSi)$ trace at $t=0$ is given by $u_0$. Furthermore, along any weak solution with this regularity we have  
    \beq\label{eq:energy-dec_ep_approx}
        \frac{\dd}{\dd t}E_\half(u,g)\leq -\half\norm{\p_tu}_{L^2(\p\Si)}^2.
    \eeq
\end{lemma}

\begin{proof}[Proof of Lemma \ref{lemma:ep_time_ind_weak_exist}]
    As $g$ is independent of time, we can consider a fixed orthonormal basis $\vph_0,\vph_1,\ldots$ of $L^2(\pSi)$ which consists of Steklov eigenfunctions corresponding to non-decreasing eigenvalues $0 = \la_0 < \la_1 \leq \la_2 \leq \cdots $, i.e. functions $\vph_j:\pSi\to \R$ with
    \begin{align*}
    	\p_{\nu}\big[ (\vph_j)_g\big] = \la_j \vph_j \text{ on }\p\Si,
    \end{align*}
    and work with the $\R^n$-valued spans
    \beqs
        V_\ell := \big\{u:\pSi\to \R^n: u=\sum_{j=0}^\ell\varphi_j b_j \text{ for some } b_j\in\R^n\big\},\quad \ell \in \N.
    \eeqs
    For any $\ell\in\N$, we consider the functions $u_\ell(t)(x)= \sum_{j=0}^\ell a_{j,\ell}(t)\vph_j(x) \in C^\infty([0,T];V_\ell)$, and their harmonic extensions $U_\ell=(u_\ell)_g$, which are chosen so that  $u_\ell(0)$ is given as the $L^2(\pSi)$-orthonormal projection of $u_0$ onto $V_\ell$ and which are so that
    \beq\label{eq:ep_galerkin_b}
        \int_{\p\Si}\p_t u_\ell\cdot w\dd s = \int_{\p\Si}Q(\p_{\nu} U_\ell) \cdot w\dd s\ \text{ for all }w\in V_\ell
    \eeq
    holds at each time in $[0,T]$. As $\p_{\nu}U_\ell(t)\vert_{\pSi} = \sum_{j=0}^\ell a_{j,\ell}(t)\la_j \vph_j \in V_\ell$ at each such time, and as the $\vph_j$ are orthonormal, we can obtain these functions $u_\ell$ simply by solving the corresponding system of ODEs 
    \beqs
        \p_ta_{j,\ell} = \sum_{i=0}^\ell \la_i\int_{\p\Si}Q(a_{i,\ell})\vph_i\vph_j\dd s,\ 0\le j\le\ell,
    \eeqs
    with the initial condition
    \beqs
        a_{j,\ell}(0) =a_j(u_0):= \int_{\p\Si}u_0\vph_j \dd s\in \R^n,\ 0\le j\le\ell.
    \eeqs
    As \eqref{eq:ep_galerkin_b} is in particular applicable for $v=\p_\nu u_\ell(t)\in V_\ell$ and as $\De U_\ell=0$ we note that the energy of these Galerkin approximates decays according to     
    \begin{align}\label{est:galerkin_uniform_bound}
        \hf\frac{\dd}{\dd t}\norm{\nabla U_\ell}_{L^2(\Si)}^2 &= \int_\Si\inner{\nabla\p_t U_\ell}{\nabla U_\ell}\dd v = \int_{\p\Si}\p_tu_\ell\cdot\p_{\nu}U_\ell\dd s\nonumber\\
        & = \int_{\p\Si}Q(\p_{\nu}U_\ell)\cdot\p_{\nu}U_\ell\dd s = -\ep\norm{\p_{\nu}U_\ell}_{L^2(\p\Si)}^2-\norm{P_{v}\p_{\nu}U_\ell}_{L^2(\p\Si)}^2\\
        & \le -\half\norm{\p_t u_\ell}_{L^2(\p\Si)}^2,\nonumber
    \end{align}
    where the last step follows since $ P_{\cdot}$ is an orthogonal projection on $N_\eta$, and hence 
    \begin{align*}
        \norm{\p_t u_\ell}_{L^2(\p\Si)}^2&=\norm{(\ep + P_{v})\p_{\nu}U_\ell}_{L^2(\p\Si)}^2
        = \ep^2\norm{\p_{\nu} U_\ell}_{L^2(\p\Si)}^2 + (1 + 2\ep)\norm{P_{v}\p_{\nu}U_\ell}_{L^2(\p\Si)}^2,
    \end{align*} 
    and since $\ep\le \hf$. We can hence deduce that 
    \begin{equation}\label{est:galerkin_uniform}
        \sup_{0\le t\le T}\norm{\nabla U_\ell(t)}_{L^2(\Si)}^2 + \norm{\p_t u_\ell}_{L^2([0,T]\times\p\Si)}^2 + \ep\norm{\p_{\nu}U_\ell}_{L^2([0,T]\times\p\Si)}^2\leq 3\norm{\nabla U_\ell(0)}_{L^2(\Si)}^2
    \end{equation} 
    which also ensures that
    \begin{equation*}
        \begin{split}
            \norm{U_\ell(t)}_{L^2(\Si)} &\le C\norm{u_\ell(t)}_{L^2(\p\Si)} \le C\big(\norm{u_\ell(0)}_{L^2(\p\Si)} + T^{\half}\norm{\p_t u_\ell}_{L^2([0,T]\times\p\Si)}\big)\\
            &\le C\norm{U_\ell(0)}_{H^1(\Si)}
        \end{split}
    \end{equation*}
    for every $t\in [0,T]$, for a constant that is allowed to depend on $T$ and a lower bound $\iota_0$ on $\inj(g)$, but that is independent of $\ell$. 
    Hence we have that
    \beq\label{est:H1-Galerkin}
        \norm{ U_\ell}_{L^\infty([0,T];H^1(\Si))}^2+\norm{\p_t u_\ell}_{L^2([0,T]\times \p\Si)}^2 + \ep\norm{\p_{\nu} U_\ell}_{L^2([0,T]\times\p\Si)}^2\le C\norm{U_\ell(0)}_{H^1(\Si)}^2,
    \eeq
    again with a constant $C=C(T,\iota_0)$.  
    
    By construction, $\norm{u_\ell(0)}_{L^2(\p\Si)}\leq\norm{u_0}_{L^2(\p\Si)}$ since $u_\ell(0)$ is the $L^2$-orthogonal projection of $u_0$ onto $V_\ell$. Importantly, the analogous $\dot H^1(\Si)$-estimate $\norm{\nabla U_\ell(0)}_{L^2(\Si)}\leq \norm{\na (u_0)_g}_{L^2(\Si)}$ also holds, since the relation 
    \beqs
        \int_\Si\langle\nabla (\vph_j)_g,\nabla w_g\rangle\dd v = \int_{\p\Si}\p_{\nu}(\vph_j)_g w\,\dd s = \lambda_j\int_{\p\Si}\vph_j w \, \dd s \text{ for every } w\in H^\half(\pSi)
    \eeqs
    ensures that the $\la_j^{-\hf} (\vph_j)_g$, $j \geq 1$, are $\dot H^1(\Si)$-orthonormal and that $U_\ell(0)-a_0(u_0)$ can alternatively be obtained as the $\dot H^1(\Si)$-orthogonal projection of $(u_0)_g$ onto the span of $(\vphi_1)_g,\dots,(\vphi_\ell)_g$.
 
    Combining the resulting uniform bounds 
    \beqas
        \norm{U_\ell(0)}_{H^1(\Si)}&\leq C(\norm{U_\ell(0)}_{L^2(\pSi)}+\norm{\na U_\ell(0)}_{L^2(\Si)})\\
        &\leq C(\norm{u_0}_{L^2(\pSi)}+\norm{\na (u_0)_g}_{L^2(\Si)})\leq C\norm{(u_0)_g}_{H^1(\Si)}
    \eeqas
    on the initial maps $U_\ell(0)$ with the $\ell$-independent estimate \eqref{est:H1-Galerkin} thus allows us to conclude that the functions $U_\ell=(u_\ell)_g$ are uniformly bounded in $L^\infty([0,T];H^1(\Si))$ and, as $\eps>0$ is fixed, that the functions $\p_t u_\ell$ and $\p_{\nu_g}U_\ell$ are uniformly bounded in $L^2([0,T]\times\p\Si)$. This allows us to pass to a subsequence, still indexed by $\ell$, along which $U_\ell$ converges $\text{weak-}*$ in $L^\infty([0,T];H^1(\Si))$ to a function $U=(u)_g$ and along which $\p_t u_\ell$ as well as $\p_{\nu_g}U_\ell$ converge weakly in $L^2([0,T]\times\p\Si)$ to $\p_t u$ and $\p_{\nu_g}U$, respectively.

    This function hence has the regularity asked for in the lemma and since $u_\ell$ satisfies 
    \begin{equation}\label{eq:weak_PDE_const}
        \int_0^T \int_{\p\Si}\p_t u_\ell \cdot w \dd s_g \dd t = \int_0^T \int_{\p\Si} Q(\p_\nu (u_\ell)_g) \cdot w\dd s_g \dd t
    \end{equation}
    for all $w \in C^{\infty}([0,T];V_\ell)$, we find that $u$ satisfies \eqref{eq:weak_PDE_const} for all $w \in \bigcup_{j=1}^\infty C^{\infty}([0,T];V_j)$ and is hence a weak solution of \eqref{eq:ep_v_map} as required.

    We furthermore observe that for solutions $u$ of this regularity, the computation carried out in \eqref{est:galerkin_uniform_bound} yields the claimed estimate \eqref{eq:energy-dec_ep_approx} on the decay of the energy. Finally we note that as our equation is \textit{linear} this in turn ensures the uniqueness of the solution in the considered class of functions.
\end{proof}

We recall that in the case of $\Si = D$ discussed in \cite{S}, improved regularity of these solutions for initial maps satisfying stronger regularity assumptions was obtained by establishing uniform estimates on the Galerkin approximates in higher order Sobolev spaces. In the present setting this argument is no longer applicable since for general surfaces we cannot expect derivatives of $\varphi_k$ to be contained in $V_k$ and hence cannot use these functions as test-functions in \eqref{eq:ep_galerkin_b} above.

However, we can easily circumvent this difficulty by  differentiating \eqref{eq:ep_v_map} in time. Since $v$ and $g$ are time independent, this leaves the equation unchanged allowing us to use a simple iterative argument that combines Lemma \ref{lemma:ep_time_ind_weak_exist} with the fundamental theorem of calculus. 

\begin{lemma}\label{lemma:ep_new_sol}
    Let $\ep\in (0, 1/2]$, $T>0$, $m\geq 2$. Then for any $g \in \MM^{m+1}$, any $v\in H^{m-\hf}((\p\Si,g);N_\eta )$ and any $u_0 \in H^{m-\hf}((\p\Si,g);\R^n)$ the corresponding solution $u$ of \eqref{eq:ep_v_map} obtained in Lemma \ref{lemma:ep_time_ind_weak_exist} is so that 
    $$u\in  L^2([0,T];H^m(\p\Si)) \text{ and 
    } u_g\in  L^\infty([0,T];H^m(\Si)). 
    $$ 
\end{lemma}

We note that this claim on the regularity could be equivalently stated as 
$$u\in L^2([0,T];H^m(\p\Si))\cap X_m(T,g)$$
since $\pt u_g=(Q \pnu u)_g\in L^\infty([0,T];H^1(\Si,g))$ if $u_g\in  L^\infty([0,T];H^2(\Si))$.  

\begin{proof}[Proof of Lemma \ref{lemma:ep_new_sol}]
    We first consider the case $m=2$. We let $w$ be the solution to \eqref{eq:ep_v_map} with initial data $w_0:= Q_{v,\eps}\p_{\nu}(u_0)_{g}\in H^{\hf}(\p\Si;\R^n)$ obtained in Lemma \ref{lemma:ep_time_ind_weak_exist}.
    We then note that 
    \begin{equation*}
        \tilde{u}(x,t) := u_0(x) + \int_0^tw(x,t')\dd t'
    \end{equation*}
    not only satisfies the initial condition $\tilde{u}(0)=u_0$ but also solves \eqref{eq:ep_v_map}. Indeed, since $v$ and $g$, and hence also $Q$ and $\nu$, are independent of time, the harmonic extension $\tilde{u}_g$ satisfies
    \begin{equation*}
        Q\p_{\nu}\tilde{u}_g (t)=Q\p_{\nu}(u_0)_{g}+ \int_0^tQ\p_{\nu}w_g(t')\dd t' = Q\p_{\nu}(u_0)_g+\int_0^t\p_tw(t')\dd t'=w(t)=\p_t\tilde{u}(t).
    \end{equation*}
    By the uniqueness of solutions of \eqref{eq:ep_v_map} established in Lemma \ref{lemma:ep_time_ind_weak_exist}, we thus have $u=\tilde{u}$ which implies that $Q\p_{\nu}u_g=\p_tu=w\in L^\infty([0,T];H^{\hf}(\pSi))\cap L^2([0,T]; H^{1}(\pSi))$. As $x\mapsto (\eps \Id+P_{x})^{-1}$ is smooth and as $v\in H^{3/2}(\pSi)$ we thus find that both $\p_{\nu}u_g$ and $\pt u$ are in $L^\infty([0,T];H^{\frac12}(\pSi))\cap L^2([0,T]; H^{1}(\pSi))$. Elliptic regularity hence allows us to deduce that $u_g\in L^\infty([0,T];H^2(\Si))\cap L^2([0,T]; H^{2}(\pSi))$, and that $\p_t u_g\in L^\infty([0,T];H^1(\Si;\R^n))$, which gives that $u\in X_2(T,g)\cap L^2([0,T]; H^{2}(\pSi))$ as claimed. 
    
    The claim for general $m\geq 3$ can now be obtained by iterating this argument.
\end{proof}

We next establish the following a priori estimates on the $H^m(\Si)$-norm of solutions of equation \eqref{eq:ep_v_map}, where we note that we now also allow $g$ and $v$ to depend on time. 

\begin{lemma}\label{lemma:ep_Hm_apriori}
    Let $m \geq 3$, $T>0$, $g \in C^1([0,T];\MM^{m+1})$ and $v:[0,T]\times\p\Si \rightarrow N_\eta$ with $v \in X_m(T,g)$. Then for any $0\leq \ep \leq \hf$ and any solution $u$ of \eqref{eq:ep_v_map} with 
    \beq\label{ass:reg-for-est}
        u\in  L^2([0,T];H^m(\p\Si)) \text{ and } u_g\in  L^\infty([0,T];H^m(\Si)), 
    \eeq
    we can bound
    \beqa\label{est:time_ind_Hm}\norm{u_g(t)}_{H^m(\Si,g(t))}^2 + \tfrac{\eps}{2} \norm{\p_\nu \nabla^{m-1} u_g}_{L^2([0,t]; L^2(\p\Si,g))}^2 \leq e^{C_{1} t}\norm{(u(0))_{g(0)}}_{H^m(\Si,g(0))}^2
    \eeqa
    for every $t\in [0,T]$ and for a constant $C_{1}=C(\iota_0,m,M)\La^{m}$ where $\La\geq 1$, $M$ and $\iota_0$ are as in \eqref{ass:const_lemma_eps}. 
\end{lemma}

We stress that the constant $C_1$ is \emph{independent} of $\ep$, and it is this key feature of the above lemma which will allow us to obtain a solution of the original problem from solutions of the regularised equation as  $\eps\to 0$ later in Section \ref{subsec:Hm-loc-exist}.

We note that as our domain surface is in general not flat, we cannot expect higher derivatives to commute, which results in many additional error terms when compared with the case of $\Si=D$ treated in \cite{S}, though these will all be of lower order. 

Throughout this proof, and also later on in Section \ref{sect:higher_reg}, we will also often want to exchange the order in which operations such as derivatives, projections and harmonic extensions are applied. To do this efficiently, we have collected the relevant estimates on the resulting commutator terms in Appendix \ref{sect:app}, see in particular Lemma \ref{lemma:commutator-app} and Remark \ref{rmk:extra-commutator-app}.

\begin{proof}[Proof of Lemma \ref{lemma:ep_Hm_apriori}]	
    As \eqref{ass:reg-for-est} ensures that $\p_\nu u_g \in L^2([0,T];H^{m-1}(\pSi,g))$ 
    we have that also $\p_t u=Q_{\eps,v}\p_\nu u_g \in L^2([0,T];H^{m-1}(\pSi,g))$, which gives us sufficient regularity to compute 
    \begin{align*}
		\frac{\dd}{\dd t}\norm{u_g}_{H^m(\Si,g)}^2 &= \frac{\dd}{\dd \ep}\Big|_{\ep=0} \norm{u_{g(\cdot+\eps)}}_{H^{m}(\Si,g(\cdot+\eps))}^2 +2 \sum_{j=0}^m\int_{\Si}\inner{\nabla^j u_g}{\nabla^j(\p_t u)_g}\dd v \\
		&\leq C \norm{\pt g}_{H^m(\Si,g)} \norm{u_g}_{H^m(\Si,g)}^2+2 \sum_{j=0}^m\int_{\Si}\inner{\nabla^j u_g}{\nabla^j(\p_t u)_g}\dd v,
	\end{align*}
    where we use Lemma \ref{lemma:metric_variations} in the second step and recall that $ \norm{\pt g}_{H^m(\Si,g)} \leq M$. The claimed estimate \eqref{est:time_ind_Hm} hence follows once we establish that $J_j:=\int_{\Si}\inner{\nabla^j u_g}{\nabla^j(\p_t u_g)}\dd v $ satisfies
	\begin{equation}\label{est:Hm_a_priori_goal-lot}
		J_j \leq C\big(1 + \norm{v_g}_{H^{m}(\Si)}^m\big)\norm{u_g}_{H^m(\Si)}^2 \text{ for }j=0,1,\dots,m-1
	\end{equation}
    while
    \begin{align}\label{est:Hm_a_priori_goal}
		J_m \leq - \tfrac{\eps}{2} \norm{\p_\nu \nabla^{m-1} u_g}_{L^2(\p\Si)}^2 + C\big(1 + \norm{v_g}_{H^{m}(\Si)}^m\big)\norm{u_g}_{H^m(\Si)}^2 
	\end{align}
    where here and in the following $C$ denotes a constant that is allowed to depend on $\iota_0$, $m$ and as usual our setting, but that is \textit{independent} of $\eps$.  

    The required bound \eqref{est:Hm_a_priori_goal-lot} on the lower order terms immediately follows from the estimate 
    \beqa\label{est:Hj-velocity}
        \norm{\na^j (\pt u)_g}_{L^2(\Si)} &= \norm{\na^j (\eps \pnu u_g+ P_{v}\pnu u_g)_g}_{L^2(\Si)}\\
        &\leq C (1 + \norm{v_g}_{H^{m-1}(\Si)}^{m-1})\norm{u_g}_{H^{m-1}(\Si)}+C\norm{u_g}_{H^m(\Si)}
    \eeqa
    which we can e.g. obtain by observing that $\norm{P_v \na^j\pnu u}_{L^2(\Si)} \leq \norm{\na^j\pnu u}_{L^2(\Si)}\leq \norm{u_g}_{H^m(\Si)}$, compare \eqref{est:grad_normal_commute}, and that each $P_v \na^j\pnu u_g-\na^j (P_{v}\pnu u_g)_g$ can be thought of as a commutator term $C_{j}(v_g,u_g)$ of the form \eqref{def:comm-rmk-2} for which Lemma \ref{lemma:commutator-app} is applicable.

	To estimate the leading order term $J_m$, we integrate by parts to rewrite
	\begin{align*}
		J_m &= \int_{\p\Si}\inner{\p_\nu\nabla^{m-1}u_g}{\nabla^{m-1}(\p_tu)_g}\dd s-\int_{\Si}\langle \Delta\na^{m-1} u_g, \nabla^{m-1}(\pt u)_g\rangle \dd v\\
		&= -\ep\int_{\p\Si}\inner{\p_\nu\nabla^{m-1}u_g}{\nabla^{m-1}(\p_\nu u_g)_g}\dd s - \int_{\p\Si}\inner{\p_\nu\nabla^{m-1}u_g}{\nabla^{m-1}(P_{v}\p_\nu u_g)_g}\dd s\\
        &\quad+ \err_1
	\end{align*}
    and exploit that $u_g$ is harmonic, and hence $\norm{\Delta\na^{m-1} u_g}_{L^2(\Si)}\leq C\norm{u_g}_{H^m(\Si)}$, as well as \eqref{est:Hj-velocity} to bound the resulting error $\err_1:= \int_{\Si}\langle \Delta\na^{m-1} u_g, \nabla^{m-1}(\pt u)_g\rangle \dd v $ by 
    \begin{equation}\label{est:Hm_apriori_err1_1}
		\err_1 \leq C\norm{u_g}_{H^m(\Si)}\norm{\nabla^{m-1}(\p_t u)_g}_{L^2(\Si)} \leq C(1 + \norm{v_g}_{H^{m-1}(\Si)}^{m-1})\norm{u_g}_{H^m(\Si)}^2.
	\end{equation}
	We now further rewrite 
	\beq \label{est:Hm_a_priori_Im}
		J_m = -\ep\norm{\p_\nu \nabla^{m-1}u_g}_{L^2(\p\Si)}^2 - \norm{P_{v}\p_\nu \nabla^{m-1}u_g}_{L^2(\p\Si)}^2+ \err_1 + \err_2 + \err_3 
	\eeq
    for 
    \beqas
		\err_2 &:= \int_{\p\Si}\inner{\p_\nu\nabla^{m-1}u_g}{P_{v}\nabla^{m-1}(\p_\nu u_g)_g - \nabla^{m-1}(P_{v}\p_\nu u_g)_g}\dd s
	\eeqas
    and 
	\beqas
        \err_3 &:= \ep\int_{\p\Si}\inner{\p_\nu\nabla^{m-1}u_g}{\p_\nu\nabla^{m-1}u_g - \nabla^{m-1}(\p_\nu u_g)_g}\dd s \\
        & \quad + \int_{\p\Si}\inner{\p_\nu\nabla^{m-1}u_g}{P_{v}\left(\p_\nu\nabla^{m-1}u_g - \nabla^{m-1}(\p_\nu u_g)_g\right)}\dd s\\
        &= \int_{\p\Si}\inner{\ep \p_\nu\nabla^{m-1}u_g + P_{v}\p_\nu\nabla^{m-1}u_g}{\p_\nu\nabla^{m-1}u_g - \nabla^{m-1}(\p_\nu u_g)_g}\dd s.
    \eeqas
    Thanks to \eqref{est:grad_normal_commute} we can bound 
    \beqas
        \err_3 &\leq C\ep\norm{\p_\nu \nabla^{m-1}u_g}_{L^2(\p\Si)} \norm{u_g}_{H^m(\Si)} + C\norm{P_{v}\p_\nu \nabla^{m-1}u_g}_{L^2(\p\Si)}\norm{u_g}_{H^m(\Si)}\\
        &\leq \eps^2 \norm{\p_\nu \nabla^{m-1}u_g}_{L^2(\p\Si)}^2+\half \norm{P_{v}\p_\nu \nabla^{m-1}u_g}_{L^2(\p\Si)}^2+C \norm{u_g}_{H^m(\Si)}^2
    \eeqas
    where we stress that $C$ is independent of $\eps\in [0,\half]$. To bound $\err_2$, we note that $C_{m-1}(v_g,u_g)= P_{v_g}\nabla^{m-1}(\p_\nu u_g)_g - \nabla^{m-1}(P_{v}\p_\nu u_g)_g$ and $\na C_{m-1}(v_g,u_g)$ are commutator terms of the form \eqref{def:comm-rmk-2}  (for $j=m-1$ and $j=m$, respectively) and that Lemma \ref{lemma:commutator-app} hence in particular implies that
    \beq\label{est:comm-weaker-sect3}
        \norm{C_{m-1}(v_g,u_g)}_{H^1(\Si)}\leq  C(1 + \norm{v_g}_{H^{m}(\Si)}^m)\norm{u_g}_{H^m(\Si)}  . 
    \eeq
    Rewriting $\err_2$ as an integral over the surface and using again that $\De \na^{m-1} u_g$ is of lower order we hence obtain that
    \beqas
        \err_2 & = \int_{\p\Si}\inner{\p_\nu\nabla^{m-1}u_g}{C_{m-1}(v_g,u_g)}\dd s\\
        &\leq \int_{\Si}\inner{\nabla^m u_g}{\nabla C_{m-1}(v_g,u_g)}\dd v + C\norm{\De \na^{m-1} u_g}_{L^2(\Si)}\norm{C_{m-1}(v_g,u_g)}_{L^2(\Si)}\\
		&\leq C\left(1 + \norm{v_g}_{H^{m}(\Si)}^m\right)\norm{u_g}_{H^m(\Si)}^2.
	\eeqas
    Inserting these error estimates into \eqref{est:Hm_a_priori_Im} yields the claimed bound \eqref{est:Hm_a_priori_goal} on $J_m$ and hence completes the proof of the lemma. 
\end{proof}

We can combine the above lemmas with a time-discretisation argument to establish the existence of solutions of \eqref{eq:ep_v_map} of this regularity for time-dependent $g$ and $v$ as claimed in Proposition \ref{prop:ep_map_solve}.

\begin{proof}[Proof of Proposition \ref{prop:ep_map_solve}.]
    As we have already established the claimed a priori estimates in Lemma \ref{lemma:ep_Hm_apriori} above, it remains to show that to any $\eps>0$, $T>0$, $g$, $v$ and $u_0$ as in the proposition there exists a unique solution $u\in X_m(T)\cap L^2([0,T];H^{m}(\pSi)) $ of \eqref{eq:ep_v_map} with $u(0) = u_0$.

    To prove this we want to apply a time-discretisation argument in which we solve \eqref{eq:ep_v_map} for maps and metrics $(\tilde v,\tilde g)$ which are piecewise constant (in time) in place of $(v,g)$. To this end, given any partition  $t_0=0<t_1<\ldots<t_K<t_{K+1}=T$ of the time interval, we can first use that $u_0\in H^{m-\half}(\pSi,g(t_0)=g_0)$ to apply Lemmas \ref{lemma:ep_time_ind_weak_exist} and \ref{lemma:ep_new_sol} on $[0,t_1]$ to obtain $\tilde u\in X_{m}([0,t_1],g(t_0)) \cap L^2([0,t_1];H^{m}(\pSi,g(t_0)))$ which solves \eqref{eq:ep_v_map} with $v(t_0)$ and $g(t_0)$ in place of the time dependent $v$ and $g$. As the spaces $H^{m-\half}(\pSi,g(t))$ all coincide, compare Lemma \ref{lemma:metric_variations}, we obtain that $\tilde u(t_1)\in H^{m-\half}(\pSi,g(t_1))$, so can iterate this argument to obtain $\tilde u\in X_{m}(T,\tilde g) \cap L^2([0,T];H^{m}(\pSi,\tilde g))$ which solves \eqref{eq:ep_v_map} for the piecewise constant in time metrics and maps $(\tilde g,\tilde v)\vert_{[t_{k-1},t_k)}=(g,v)(t_{k-1})$. 

    On each individual interval $[t_{k-1},t_k)$, the assumptions of Lemma \ref{lemma:ep_Hm_apriori} are satisfied so we can control the evolution of $\norm{\tilde u(t)}_{H^{m}(\Si,\tilde g)}$ by the estimate \eqref{est:time_ind_Hm} obtained in this Lemma \ref{lemma:ep_Hm_apriori}. Since the bounds on the $H^m$-norm and injectivity radius in \eqref{ass:const_lemma_eps} hold for the same numbers $\La$ and $\iota_0$ also for $\tilde v$ and $\tilde g$, we hence obtain that 
    \beqa\label{est:pv_const_apriori0}
        \norm{(\tilde u(t))_{\tilde g(t)}}_{H^m(\Si,\tilde g(t))}^2 +\tfrac{\eps}{2} \norm{\p_\nu \nabla^{m-1} u_g}_{L^2([t_{k-1},t]; L^2(\p\Si,\tilde{g}(t_{k-1}))}^2\\
        \le e^{C_1(t-t_{k-1})}\norm{(\tilde u(t_{k-1}))_{\tilde g(t_{k-1})}}_{H^m(\Si, \tilde g(t_{k-1}))}^2
    \eeqa
    for every $t\in [t_{k-1},t_k)$ and for a constant $C_1=C(\iota_0,m,\La)$. 
    
    We note that while the map $\tilde u(t):\p\Si\to \R^n$ is continuous across the times $t_j$, since the metric jumps at $t_j$, its harmonic extensions $(\tilde u(t))_{\tilde g(t)}$ and the corresponding $H^m(\Si,\tilde g(t))$-norms may jump. This jump is however controlled by the estimate \eqref{est:harm-ext-timedep-gen} from Lemma \ref{lemma:metric_variations} which ensures that 
    \begin{equation}\label{est:norm_jump} 
      \norm{(\tilde u(t_k))_{\tilde g(t_{k})}}_{H^m(\Si,\tilde g(t_k))}^2 \le e^{C M\abs{t_k-t_{k-1}}}\lim_{t \nearrow t_{k}}\norm{(\tilde u(t))_{\tilde g(t)}}_{H^m(\Si,\tilde g(t))}^2
    \end{equation} 
    for a constant $C=C(\iota_0,m)$.   
    
    Combined with \eqref{est:time_ind_Hm} we hence deduce that the harmonic extension $\tilde U=(\tilde{u})_{\tilde g}$ obtained from any such  $\tilde u$ is so that 
    \beq\label{est:pv_const_apriori1}
        \norm{\tilde U(t)}_{H^m(\Si,\tilde g(t))}^2 + \tfrac{\ep}{2}\norm{\p_\nu \nabla_{\tilde g}^{m-1}\tilde U}_{L^2([0,t];L^2(\p\Si,\tilde{g}))}^2 \le e^{Ct}\norm{(u_0)_{g_0}}_{H^m(\Si, g_0)}^2
    \eeq
    for every $t\in[0,T]$ and a constant $C=C(\iota_0,m,\La,M)$. Combined with standard elliptic estimates and Lemma \ref{lemma:metric_variations} this hence yields uniform bounds of 
    \beq\label{est:pv_const_apriori1-g0}
        \norm{\tilde U}_{L^\infty([0,T];H^m(\Si,g_0))}^2 + \ep \norm{\nabla^{m} \tilde{U}}_{L^2([0,T]; L^2(\p\Si,g_0))}^2 \le C\norm{(u_0)_{g_0}}_{H^m(\Si, g_0)}^2,
    \eeq
    $C=C(\iota_0,T,m)$ for all $\tilde U=(\tilde u)_{\tilde g}$ obtained by such a time discretisation.     

    Applying this argument for a sequence  $(\tilde g_\ell,\tilde v_\ell)$ which corresponds to partitions whose mesh-size tends to zero, we can hence pass to a subsequence so that the resulting maps $\tilde U_\ell=(\tilde u_\ell)_{\tilde g_\ell}$ converges weak-$*$ in $L^\infty([0,T];H^m(\Si,g_0))$  and weakly in $L^2([0,T];H^{m+\hf}(\Si,g_0))$ to a limiting map $U \in L^\infty([0,T];H^m(\Si,g_0))\cap L^2([0,T];H^{m+\hf}(\Si,g_0))$. Since $\Delta_{\tilde g_{\ell}} U_{\ell}=0$ for every $\ell$ and since $\tilde g_\ell\to g$ in $L^\infty([0,T];\MM^{m+1}(\Si))$ we must have $\Delta_g U=0$, and can hence view $U=u_g$ as  the harmonic extension of the weak limit $u$ of the traces $\tilde{u}_\ell$.
    
    As maps in $X_m(T)$ are continuous on $\pSi\times [0,T]$, compare Lemma \ref{lemma:Xm} below,  we can additionally use that $\tilde v_\ell$ converge to $v$ uniformly on $[0,T]\times \Si$ to pass to the limit in $\p_t u_\ell=-(\eps+P_{\tilde v_\ell})\p_{\nu_{\tilde g_\ell}} U_\ell$ to obtain that $u$ is the desired solution of 
    \beqs
        \p_t u=-(\eps+P_{v})\p_{\nu_g} U= -(\eps+P_{v})\p_{\nu_g} u_g \text{ with } u(0)=u_0. 
    \eeqs
    Finally, to establish the uniqueness of solutions of the given regularity to this linear equation it suffices to exclude the possibility that there is a non-trivial solution that evolves from $u_0=0$ and this follows from Gronwall's inequality as any  solution of \eqref{eq:ep_v_map} satisfies
    \beq\label{est:energy-decay_ep_v_flow}
        \tfrac{\dd}{\dd t} E_g(u) = \tfrac{\dd}{\dd \eps}E_{g(\cdot+\eps)}(u)-\ep\norm{\p_{\nu}u_g}_{L^2(\p\Si)}^2-\norm{P_{v}\p_{\nu}u_g}_{L^2(\p\Si)}^2\leq C(\iota_0) M E_g(u).
    \eeq 
\end{proof}
In addition to the properties of the solutions $u=u_{\eps,v,g}$ of \eqref{eq:ep_v_map} that we have just established, we also show that the dependence of this solution on the curves of maps $v$ and  metrics $g$ is Lipschitz in the following sense. 
\begin{lemma}\label{lemma:ep_Lip_est}
    For any $\iota_0>0$ and $M,\La<\infty$, there exists a constant $C$ so that the following holds true for every $0<\ep\le\hf$ and any $0<T\le 1$. Let $m=3$, let $(u_0,g_0)\in H^{m-\hf}(\p\Si;\R^n)\times\MMtwoiota^{m+1}$, let $\UU=\UU^{m+1}$ be the neighbourhood of $g_0$ from Proposition \ref{prop:step1}, and let $v_{1,2}\in X_m(T)$ and $g_{1,2}\in C^1([0,T];\UU)$ be any pair of maps and metrics such that $v_{1,2}(\p\Si\times[0,T])\subset N_{\eta}$, $(v_{1,2}(0),g_{1,2}(0)) = (u_0,g_0)$ and for which \eqref{ass:const_lemma_eps} holds for this $m=3$, $M$ and $\La$. 
    \\
    Then the corresponding solutions $u_{i}=u_{\eps,v_i,g_i}$, $i=1,2$, of \eqref{eq:ep_v_map} with initial conditions $u_i(0)=u_0$ satisfy
    \beqas
        \sup\limits_{[0,t]}\norm{\nabla_{g_1}(u_2-u_1)_{g_1}}_{L^2(\Si,g_1)}^2 & + \int_0^t\norm{\p_t (u_2-u_1)}_{L^2(\p\Si,g_1)}^2 \dd t\\
        &\le C\ep^{-1}t\sup\limits_{[0,t]} \left( \dist ^2_{H^{3}}(g_1,g_2)+\norm{(v_2-v_1)_{{g_1}}}^2_{H^1(\Si,g_1)}\right)
    \eeqas
    for all $0\le t\le\min\{T,\delta/M\}$, where $\de>0$ is as in Proposition \ref{prop:step1}. 
\end{lemma}
We note that we can of course apply this lemma also for any other choice of $m\geq 3$ since we can assume that the neighbourhoods $\UU^k=\UU^k(g_0)$ are chosen so that $\UU^{k+1}\subset \UU^k$ and since increasing the exponent $m$ simply strengthens the assumptions while leaving the conclusions invariant.
\begin{proof}[Proof of Lemma \ref{lemma:ep_Lip_est}]
    We first note that $w = u_2-u_1$ satisfies an inhomogenous version of equation \eqref{eq:ep_v_map}, namely 
    \begin{equation}
    	\p_t w = -(\ep + P_{v_1})\p_{\nu_1}w_{g_1} + f,
    \end{equation}
    where $\nu_1=\nu_{g_1}$ and where $f = f_1 + f_2 + f_3$ for
    \begin{align*}
    	f_1 &= (\ep + P_{v_1})\p_{\nu_1}\left((u_2)_{g_1} - (u_2)_{g_2}\right),\\
    	f_2 &= (\ep + P_{v_1})\left(\p_{\nu_1} - \p_{\nu_2}\right)(u_2)_{g_2},\\
    	f_3 &= \left(P_{v_1} - P_{v_2}\right)\p_{\nu_2}(u_2)_{g_2}.
    \end{align*}
    We furthermore note that $(u_1)_{g_1}$ and $(u_2)_{g_1}$ are uniformly bounded in $H^m(\Si,g_1)$ since Proposition \ref{prop:ep_map_solve} yields uniform $H^m(\Si,g_i)$ bounds on $(u_i)_{g_i}$ and since the dependence of the norms and the harmonic extension on the metric are controlled by Lemma \ref{lemma:metric_variations}. We can thus apply Lemma \ref{lemma:metric_variations} to bound the error terms by 
    \beqa\label{est:f-lip}
        \norm{f_1}_{L^2(\p\Si,g_1)}&\le 2\norm{\p_{\nu_1}\left((u_2)_{g_1}-(u_2)_{g_2}\right)}_{L^2(\p\Si,g_1)} \le C\dist_{H^{m}}(g_1,g_2)\\
        \norm{f_2}_{L^2(\p\Si,g_1)}&\le 2\norm{\left(\p_{\nu_1}-\p_{\nu_2}\right)(u_2)_{g_2}}_{L^2(\p\Si,g_1)} \le C\dist_{H^{m}}(g_1,g_2)\\
        \norm{f_3}_{L^2(\p\Si,g_1)} &= \norm{\left(P_{v_1}-P_{v_2}\right)\p_{\nu_2}(u_2)_{g_2}}_{L^2(\p\Si,g_1)} \le C\norm{v_2-v_1}_{L^2(\p\Si,g_1)}
    \eeqa
    at any time in $[0,T]$ and for $C=C(\iota_0,\La)$. We can furthermore compute
    \beqa\label{est:ddE-w}
        \tfrac{\dd}{\dd t}E_{\hf}(w,g_1) &=\int_{\Si}\inner{\nabla_{g_1}w_{g_1}}{\nabla_{g_1}\p_tw_{g_1}}_{g_1}\dd v_{g_1}
        +\tfrac{\dd}{\dd \ep}\big|_{\eps=0} E_{\hf}(w,g_1(t+\ep))\\
        &=-\ep\norm{\p_{\nu_1}w_{g_1}}_{L^2(\p\Si,g_1)}^2-\norm{P_{v_1}\p_{\nu_1}w_{g_1}}_{L^2(\p\Si,g_1)}^2\\
        &\quad+\int_{\p\Si}\p_{\nu_1}w_{g_1}\cdot f\dd s_{g_1}+\tfrac{\dd}{\dd \ep}\big|_{\eps=0} E_{\hf}(w,g_1(t+\ep)). 
    \eeqa
    and note that \eqref{est:g_change_E} allows us to bound $\frac{\dd}{\dd \ep}\big|_{\eps=0} E_{\hf}(w,g_1(t+\ep)) \le CE_{\hf}(w,g_1)$ for $C=C(M,\iota_0)$. We can hence apply Young's inequality and use the above estimates \eqref{est:f-lip} on $f=f_1+f_2+f_3$ to deduce that  
    \begin{align*}
        \frac{\dd}{\dd t}E_{\hf}(w,g_1) &\leq -\hf \ep\norm{\p_{\nu_1}w_{g_1}}_{L^2(\p\Si,g_1)}^2 -\norm{P_{v_1}\p_{\nu_1}w_{g_1}}_{L^2(\p\Si,g_1)}^2+ CE_{\hf}(w,g_1)\\
        &\quad+ C\ep^{-1} \left(\dist^2_{C^{3}}(g_1,g_2)+\norm{v_2-v_1}^2_{L^2(\p\Si,g_1)}\right),
    \end{align*}
    so, as $w(0)=0$ and $T\leq 1$, that
    \begin{align*}
        \norm{(\nabla_{g_1}&w_{g_1})(t)}_{L^2(\Si,g_1(t))}^2 + \int_0^t \ep\norm{\p_{\nu_1}w_{g_1}}_{L^2(\p\Si,g_1)}^2 + \norm{P_{v_1}\p_{\nu_1}w_{g_1}}_{L^2(\p\Si,g_1)}^2\dd t'\\
        &\le C\ep^{-1}t\sup\limits_{[0,t]}\left(\dist_{C^{3}}(g_1,g_2)^2 + \norm{v_2-v_1}_{L^2(\p\Si,g_1)}^2\right).
    \end{align*}
    This immediately implies the claim of the lemma since $\abs{\p_t w} \leq \eps\abs{\p_{\nu_1}w_{g_1}}+ \abs{P_{v_1}\p_{\nu_1}w_{g_1}}+ \abs{f}$, for $f$ controlled by \eqref{est:f-lip}, and since $\norm{v_2-v_1}_{L^2(\p\Si,g_1)}\leq C \norm{(v_2-v_1)_{g_1}}_{H^1(\Si,g_1)}$.
\end{proof}

\subsection{Analysis of the regularised coupled flow}\label{subsec:reg-coupled}
We now want to establish the existence of a solution of the non-linear, but still regularised, system \eqref{eq:flow-ep_coupled} based on the results on the metric obtained in Section \ref{sect:metric} and the properties of the solutions $u_{\ep,v,g}$ of the regularised and linearised equation \eqref{eq:ep_v_map}. To this end we show

\begin{lemma}\label{lemma:FPT}
    For any $\ep \in (0,\hf]$,  $m \geq 3$, $\iota_0,d_0>0$ and $E_0,\La_0<\infty$ there exists $T>0$ and $C_0\geq 2$ so that the following holds true for any $g_0\in \MMtwoiota^{m+1}$, the corresponding neighbourhood $\UU$ from Proposition \ref{prop:step1} and any $u_0\in H^{m-\hf}(\p\Si;N_\eta)$ with $\norm{(u_0)_{g_0}}_{H^m(\Si,g_0)}\leq \La_0$, $E_{\hf}(v_0,g_0)\leq E_0$ and $\dist(u_0(\pSi),\partial N_\eta)>d_0$. 
    
    For any $v$ in the set 
    \beqs
        S=\{v\in X_m(T) :  v(0)=u_0, \norm{v_{g_0}}_{X_m(T,g_0)}\le\La:=C_0\La_0, E_{\hf}(v(t),g_0)\le 2E_0\}
    \eeqs
    the unique solution $g_v$ of \eqref{eq:metric_lemma} with $g_v(0)=g_0$ and the unique solution $u_{\ep,g_v,v}$ of \eqref{eq:ep_v_map} with $u(0)=u_0$ are defined on all of $[0,T]$ with $g_v(t)\in\UU$ and $u(t)(\pSi)\subset N_{\eta}$ for $0<t\le T$, and the map 
    $$\Psi=\Psi_{\ep,u_0,g_0}\colon v\mapsto u_{\ep,g_v,v}$$
    maps the set $S$ into itself and is a contraction with respect to the norm
    \begin{equation*}
        \norm{v}_S^2:=\sup\limits_{0\le t\le T}\norm{(v(t))_{g_0}}_{H^1(\Si,g_0)}^2 + \int_0^T\int_{\p\Si}\abs{\p_tv}^2\dd s_{g_0}\dd t.
    \end{equation*}
\end{lemma}
For the proof of this lemma, and in the next section, we will use that the spaces $X_m$ have the following useful properties.
\begin{lemma}\label{lemma:Xm}
    Let $g_0\in \MMiota^{m+1}$ for some $m\geq 2$ and $\iota_0>0$,  and let $X_m(T,g_0)$ be as defined in \eqref{def:Xm}. Then the harmonic extension $v\mapsto v_{g_0}$ is a compact operator from $(X_m(T),\norm{\cdot}_{(X_m(T),g_0)})$ to $C^{0}([0,T];H^{m-1}(\Si,g_0))$ and there exists a constant $C=C(\iota_0)$ so that
    \beq\label{est:C0-norm-bdry}
        \norm{v(t)-v(s)}_{C^0(\p \Si)}\leq C \norm{v}_{X_m}\abs{t-s}^\half
    \eeq
     for any $v\in X_m(T,g_0)$ and for all $t,s\in [0,T]$.
\end{lemma}

\begin{rmk}\label{rmk:v-in-Nde}
    We note in particular that if $v\in X_m(T,g_0)$ is so that $v_0:=v(t=0)$ maps $\p\Si$ into $N_\eta$ and if $\La$ is so that $\norm{v}_{X_m}\leq \La$, then we are guaranteed that $v(t)(\p \Si) \subset N_\eta$ at least for times $t\in [0,T_2]$, where $T_2:= c_0 \La^{-2} \dist_{\R^n}^2(v_0(\p\Si),\p N_\eta)$ for a constant $c_0=c_0(\iota_0)>0$.
\end{rmk}

\begin{proof}[Proof of Lemma \ref{lemma:Xm}]
    The first claim follows by combining the fact that  $W^{1,\infty}([0,T];H^1(\Si,g_0))$ embeds compactly into $ C^0([0,T];L^2(\Si,g_0))$ with the interpolation inequality
    \begin{equation*}
    	\norm{f}_{H^{m-1}(\Si,g_0)}^m \leq C\norm{f}_{L^2(\Si,g_0)}\norm{f}_{H^m(\Si,g_0)}^{m-1}, \quad C=C(\iota_0,m)
    \end{equation*}
    which follows inductively from $\norm{f}_{H^k}^2\leq C \norm{f}_{H^{k-1}}\norm{f}_{H^{k+1}}$, compare also \eqref{app:Hs-interpol}.
  
    To obtain the second claim it suffices to note that the trace operator is continuous from $H^\frac{3}{2}(\Si)$ into $C^0(\p \Si)$ and that for any $0\leq s\leq t\leq T$, we can use \eqref{app:Hs-interpol} to bound 
    \begin{align*}
    	\norm{v(s) - v(t)}_{H^\frac{3}{2}(\Si,g_0)}^2 &\leq C\norm{v(s) - v(t)}_{H^2(\Si,g_0)}\norm{v(t) - v(s)}_{H^1(\Si,g_0)}\\
    	&\leq C\norm{v}_{X_2}\int_s^t\norm{\p_{t} v(t')}_{H^1(\Si,g_0)} \dd t'\leq C\norm{v}_{X_m}^2 \abs{s-t}.
    \end{align*}
\end{proof}
For the proof of Lemma \ref{lemma:FPT} we will furthermore use that
\begin{equation}\label{est:proj_for_CMT}
    \norm{P_{v}\p_\nu u_g}_{H^{1/2}(\pSi)} \leq C(1 + E_\hf(v,g)^\hf)\norm{u_g}_{H^3(\Si)}, 
\end{equation}
and hence that 
\beqa\label{est:pt_for_CMT}
    \norm{(\pt u)_g}_{H^{1}(\Si)} &\leq  C\norm{\pt u}_{H^{1/2}(\pSi)}\\
    &\leq C\eps  \norm{\pnu u_g}_{H^{1/2}(\pSi)}+ C \norm{P_{v}\p_\nu u_g}_{H^{1/2}(\pSi)}\\
    & \leq C(1 + E_\hf(v,g)^\hf)\norm{u_g}_{H^3(\Si)}
\eeqa
along any solution $u=\Psi(v)$ of \eqref{eq:ep_v_map} as considered in Lemma \ref{lemma:FPT}.

We note that to obtain  \eqref{est:proj_for_CMT} we can use that if $f_1\in H^{\half}(\pSi)$ and $f_2\in H^{3/2}(\pSi)$ then the product $f_1 f_2$ is of course in $H^{1/2}(\pSi)$ with 
$$\norm{f_1f_2}_{H^{1/2}(\pSi)}\leq C \norm{f_1}_{H^{1/2}(\pSi)}\norm{f_2}_{H^{3/2}(\pSi)}$$
and that $\norm{P_v}_{H^{1/2}(\pSi)}\leq C(1+\norm{v}_{H^{1/2}(\pSi)})$ since $x\mapsto P_{x}$ is smooth, hence allowing us to bound
\beqas
    \norm{P_{v}\pnu u_g}_{H^{1/2}(\p\Si)} & \leq C(1+\norm{v}_{H^{1/2}(\pSi)}) \norm{\p_{\nu}u_g}_{H^{3/2}(\p\Si)} \\
    & \leq C(1 + E_\half(v,g)^{\half}) \norm{u_g}_{H^3(\Si)}.
\eeqas

\begin{proof}[Proof of Lemma \ref{lemma:FPT}]
    We first recall from Remark \ref{rmk:v-in-Nde} that if we choose
    \beqs
        T\le c_0\La^{-2}d_0^2,
    \eeqs
    then $v([0,T]\times\pSi)\subset N_{\eta}$ for all $v\in S$. 
    
    We then note that Proposition \ref{prop:step1} ensures that if $T>0$ is chosen to be no more than $\delta/(4E_0)$, for $\de$ as in that proposition, then the curve of metrics $g=g_v$ exists and satisfies $g_v(t)\in\UU$ for all $t\in [0,T]$ as well as 
    \beq\label{est:FPT_g}
        \norm{\pt g}_{H^{m+1}(\Si,g)}\le CE_0= :M
    \eeq
    for a constant $M$ that only depends on $E_0$ and $\iota_0$. 
    
    We furthermore recall that Lemma \ref{lemma:metric_variations} ensures that for any $f\in H^{m-\hf}(\pSi)$,
    \begin{equation}\label{est:harm-ext-timedep} 
        \norm{f_{g(t_0)}}_{H^j(\Si,g(t_0))}^2 \le e^{C\abs{t_1-t_0}M}\norm{f_{g(t_1)}}_{H^j(\Si,g(t_1))}^2 \le 2\norm{f_{g(t_1)}}_{H^j(\Si,g(t_1))}^2
    \end{equation} 
    for all $0\le t_{0,1}\le T$ and all $j=1,\ldots, m$ where the last inequality holds after possibly reducing $\de>0$ and ensuring that $T\le\delta/M$. In particular, we have
    \beq\label{est:v_H^m-bound}
        \norm{(v(t))_{g(t)}}_{H^m(\Si,g(t))}\le 2 \La \text{ for all } v\in S \text{ and all } t\in [0,T].
    \eeq
    
    Hence, for such $0<T<\de/M$ and any $v\in S$ we can apply Proposition \ref{prop:ep_map_solve}, with $\La$ replaced by $2\La$, to obtain a unique weak solution $u=u_{\ep,g,v}$ of the linear equation \eqref{eq:ep_v_map} which, after possibly further reducing $\de$, satisfies
    \beqa\label{est:u_H^m-bound}
        \norm{(u(t))_{g_0}}_{H^m(\Si,g_0)} &\leq 2\norm{(u(t))_{g(t)}}_{H^m(\Si,g(t))} \le 2e^{CT}\norm{(u_0)_{g_0}}_{H^m(\Si,g_0)}\\
        &\le 4\norm{(u_0)_{g_0}}_{H^m(\Si,g_0)}\le 4\La_0
    \eeqa
    for all $0<t\le T$.

    From \eqref{est:harm-ext-timedep}, \eqref{est:pt_for_CMT} and \eqref{est:u_H^m-bound} we furthermore obtain
    \beq\label{est:u_t-bound}
        \begin{split}
            \norm{\p_t(u(t))_{g_0}}_{H^1(\Si,g_0)}&=\norm{(\p_t u(t))_{g_0}}_{H^1(\Si,g_0)} \le 2\norm{(\p_t u(t))_{g(t)}}_{H^1(\Si,g(t))}\\
            &
            \le C(1+E_\hf(v(t),g(t))^\hf)\norm{u_{g(t)}}_{H^3(\Si,g(t))}\\
            &\le C(1+E_0^\hf)\La_0.
        \end{split}
    \eeq 
    Combined with \eqref{est:u_H^m-bound}, we thus obtain 
    \beqs
        \norm{u}_{X_m(T,g_0)}\le C(1+E_0^\half)\La_0\leq C_0\La_0=\La,
    \eeqs
    if $C_0=C_0(E_0,\iota_0)>0$ is chosen sufficiently large.
    
    To obtain the desired bound on the energy we can then apply the estimate \eqref{est:energy-decay_ep_v_flow} on the evolution of the energy obtained above which, for sufficiently small $\delta>0$ and $T\le\delta/M$, implies the bound
    \beqs
        E_{\hf}(u(t),g_0)\leq e^{CMT} E_{\hf}(u(t),g(t))\le e^{2CMT}E_{\hf}(u_0,g_0)\le 2E_0.
    \eeqs
    This completes the proof that $u\in S$, which as already observed also ensures that $u([0,T]\times \pSi))\subset N_\eta$.
    
    It remains to show that $\Psi\colon S\to S$ is a contraction with respect to the norm $\norm{\cdot}_{S}$. So we let $v_{1,2}\in S$, and let $g_{1,2}=g_{v_{1,2}}$ as well as $u_{1,2}=\Psi(v_{1,2})$ be the corresponding solutions of \eqref{eq:metric_lemma} and \eqref{eq:ep_v_map}.
    
    Since $w:=u_2-u_1$ is so that $w(0)=0$, and since we can assume that $T\leq 1$, we can bound 
    \begin{align*}
        \norm{w_{g_1}}^2_{L^2(\Si,g_1)}\le C\norm{w}^2_{L^2(\p\Si,g_1)} \le C\bigg(\int_0^t\norm{\p_tw}_{L^2(\p\Si,g_1)}\dd t\bigg)^2 \le C\norm{\p_tw}_{L^2([0,T];L^2(\p\Si,g_1))}^2
    \end{align*}
    at any $0\le t\le T$. As the curves of metrics $g_{1,2}$ induce uniformly equivalent Sobolev norms and harmonic extensions as recalled above we can hence bound 
    \beqas
        \norm{w}^2_S&=\sup\limits_{t \in [0,T]}\norm{w_{g_0}}^2_{H^1(\Si,g_0)} + \norm{\p_t w}^2_{L^2([0,T]\times\p\Si,g_0)}\\
        &\le 2\sup\limits_{t \in [0,T]}\norm{w_{g_1}}^2_{H^1(\Si,g_1)} + 2\norm{\p_t w}^2_{L^2([0,T]\times\p\Si,g_1)}\\
        &\leq 2 \sup\limits_{t \in [0,T]}\norm{\na_{g_1}(u_2-u_1)_{g_1}}^2_{L^2(\Si,g_1)} +C \norm{\p_t w}^2_{L^2([0,T]\times\p\Si,g_1)}\\
        &\leq C\ep^{-1}T\sup\limits_{[0,T]} \left( \dist ^2_{H^{3}}(g_1,g_2)+\norm{(v_2-v_1)_{{g_1}}}^2_{H^1(\Si,g_1)}\right)
    \eeqas
    where the last step follows from the Lipschitz estimate proven in Lemma \ref{lemma:ep_Lip_est}.

    We note that  \eqref{est:g_diff_apriori} implies 
    \begin{equation}\label{est:contraction_term_2}
        \dist_{H^3}(g_1(t),g_2(t)) \leq CT \sup\limits_{0 \leq t \leq T} \norm{(v_2(t)-v_1(t))_{g_0}}_{H^1(\Si,g_0)} \leq CT \norm{v_2-v_1}_{S},
    \end{equation}
    while \eqref{est:harm-ext-timedep} ensures that $\sup\limits_{[0,T]}\norm{(v_2-v_1)_{{g_1}}}^2_{H^1(\Si,g_1)} \leq 2\norm{v_2-v_1}_{S}^2$. Combined we hence deduce that 
    \beqas
        \norm{u_2-u_1}^2_S & \leq C\eps^{-1} T \norm{v_2-v_1}_S^2\leq \half \norm{v_2-v_1}_S^2
    \eeqas
    where the last inequality holds after possibly further reducing $T>0$ to ensure that $C\eps^{-1}T<\half$ for the constant $C=C(\iota_0, E_0, \La_0)$ obtained in the above argument. 
\end{proof}

Since $(S,\norm{\cdot}_{S})$ is a complete metric space, Lemma \ref{lemma:FPT} ensures the existence of a solution of the regularised system \eqref{eq:flow-ep_coupled} for any $\eps>0$ at least on a short time interval as made precise in the following corollary.

\begin{cor}\label{cor:new-FPT}
For every $\eps>0$ and every $\La_0, E_0, \iota_0,\eta>0$ there exists a number $T_\eps>0$ that only depends on $\eps>0$ and these four numbers  $\La_0, E_0, \iota_0, \eta$ so that the solution of \eqref{eq:flow-ep_coupled} is guaranteed to exist at least on the interval $[0,T_\eps]$ whenever we consider initial metrics $\hat g_0\in \MMiota^{m+1}$ and initial maps $\hat u_0$ with 
 $$E(\hat u_0,\hat g_0) \leq E_0, \quad \norm{(\hat u_0)_{\hat g_0}}_{H^m(\Si,\hat{g}_0)}\leq 2\La_0 \text{ and }\dist(\hat u_0(\pSi),N)\leq \thalf \eta.$$
 \end{cor}
 
Thanks to the $\eps$-independent a priori bounds on the evolution of the map and the metric obtained in Proposition \ref{prop:step1} and Lemma \ref{lemma:ep_Hm_apriori}, and the fact that, similar to \eqref{eq:energy-decay}, along solutions $(u_\eps,g_\eps)$ of \eqref{eq:flow-ep_coupled}, the energy decreases according to 
\beq\label{est:energy_decreas_eps}
    \frac{\dd}{\dd t} E_{g_\eps}(u_\eps)=-\ep\norm{\p_{\nu_{g_\eps}}u_\eps}_{L^2(\p\Si,g_\eps)}^2-\norm{P_{u_\eps}\p_{\nu_g}u_\eps}_{L^2(\p\Si)}^2-\norm{\pt g_\eps}_{L^2(\Si,g_\eps)}^2
\eeq
we can now apply the above corollary iteratively to obtain the uniform control on all these approximate solutions claimed in Proposition \ref{prop:ep_system}.

\begin{proof}[Proof of Proposition \ref{prop:ep_system}]
    Let $u_0$ and $g_0$ be as in the Proposition and let $T>0$ be an ($\eps$-independent) number that is fixed below. 
   
    Given any $\eps>0$, we can use that Corollary \ref{cor:new-FPT}, and the fact that $\inj(g)\geq \iota_0$ for 
   all $g\in  \UU$, $\UU=\UU_{g_0,r_0}^{m+1}$ as in Lemma \ref{lemma:U-basics},
   allows us to extend the solution $(u_\eps,g_\eps)$ of \eqref{eq:flow-ep_coupled} to this initial map and metric by a fixed amount of time $T_\eps>0$ for as long as 
    \beq\label{stuff-we-need}
        \norm{(u_\eps)_{g_\eps}}_{H^m(\Si,g_\eps)}\leq 2\La_0, \quad \dist(u_{\ep}(t)(\pSi), N)\leq \half\eta \text{ and } g_\eps(t)\in \UU.
    \eeq
    As $\norm{(u_0)_{g_0}}_{H^m(\Si,g_0)}\leq \La_0$, $\dist(u_{0}(\pSi), N)=0$ and as  $g_0$ is the centre of the ball $\UU=\UU_{g_0,r_0}^{m+1}$, these estimates are of course satisfied (with strict inequalities) at $t=0$. In order to prove that this solution $(u_\eps,g_\eps)$ indeed exists on the ($\eps$-independent) interval $[0,T]$, it hence suffices to exclude the possibility that the metric reaches $\partial \UU$ or that one of the above estimates on $u_\ep$ holds with equality at some time before $T$.

    For the former, we can use that the energy decreases along the flow, compare \eqref{est:energy_decreas_eps}, and apply \eqref{est:velocity-by-energy} to obtain that the estimate
    \beqa\label{est:H-is-nice}
        \norm{\pt g_\eps}_{H^{m+1}(\Si,g_\eps)}\leq C E_0 \text{ for some }C=C(\iota_0)
    \eeqa
    is valid for as long as $g_\eps$ remains in $\UU$.  For $T$ chosen so that $T\leq r_0 (CE_0)^{-1}$, this hence ensures $g_\eps$ cannot reach  the boundary of $\UU=\UU_{g_0,r_0}^{m+1} $ before such a time $T>0$.
    
    Combining these $\eps$-independent  bounds on $\norm{\pt g}_{H^{m+1}(\Si,g)}$ with the $\eps$-independent estimates on the evolution of the $H^m$-norms of $(u_\eps)_{g_\eps}$ obtained in Proposition \ref{prop:ep_map_solve} furthermore allows us to deduce that for suitably chosen  $T=T(E_0,\iota_0,\La_0)>0$ we can bound  
    \beqas
        \norm{(u_\eps)_{g_\eps}}_{L^\infty([0,t];H^m(\Si,g_\eps))}^2 \leq e^{Ct}\norm{(u_0)_{g_0}}_{H^m(\Si,g_0)}^2\leq \tfrac32 \La_0
    \eeqas
    on any interval $[0,t]$, $t\leq T$, on which $(u_\eps,g_\eps)$ exists. 
    
    Combining the resulting uniform $X_m(t,g_\eps)$-bounds with Lemma \ref{lemma:Xm} then allows us to deduce that for suitably small $T=T(E_0,\iota_0,\La_0)>0$ also 
    \begin{equation*}
        \dist(u_\eps(\pSi),N)\leq Ct^{\half} \norm{u_\eps}_{X_m(t,g_\eps)}\leq \tfrac14 \eta
    \end{equation*}
    for any such $t$. 

    For this $\eps$-independent $T=T(E_0,\iota_0,\La_0)$ we hence conclude that none of the solutions $(u_\eps,g_\eps)$, $\eps\in (0,\half)$ can violate \eqref{stuff-we-need} at any time $t\leq T$, and hence that all these solutions are defined and so that \eqref{stuff-we-need} holds on the whole interval $[0,T]$ as claimed in the proposition. 

    We finally remark that the claimed $C^{1,1}$ estimate for the metric component is an immediate consequence of the a priori estimates \eqref{est:ptg_apriori_1} and \eqref{est:ptg_apriori_2} obtained in Proposition \ref{prop:step1} as the estimate $E_{\half}(u_\eps(t),g_0)\leq CE_{\half}(u_\eps(t),g_\eps(t))\leq C E_0$, $C=C(\iota_0)$, remains valid for as long as $g_\eps$ is in $\UU$. 
\end{proof}

\subsection{Local existence of solutions of the original flow \eqref{eq:flow}}\label{subsec:Hm-loc-exist}
We are now finally in a position to establish the existence of solutions of our original flow \eqref{eq:flow} for sufficiently regular initial data based on the $\ep$-independent control on the solutions $(u_\eps,g_\eps)$ of the regularised flow \eqref{eq:flow-ep_coupled} that we have just established. 
\begin{proof}[Proof of Proposition \ref{prop:Hm_plateau_existence}]
    Let $u_0,g_0$ be as in the proposition and let $(u_\eps,g_\eps)$, $\eps\in (0,\half]$ be the corresponding solutions of the regularised system \eqref{eq:flow-ep_coupled}. Proposition \ref{prop:ep_system} ensures that these solutions all  exist at least on the $\eps$-independent interval $[0,T]$, for $T>0$ as in that proposition, that the metric components are uniformly bounded in $C^{1,1}([0,T];\MM^{m+1})$ and contained in $\UU$, which, combined with \eqref{est:Hm-uniform-prop}, ensures that  the maps $U_\ep := (u_\ep)_{g_\ep}$ are uniformly bounded in $L^\infty([0,T];H^m(\Si,g_0))$. As $u_\eps$ solves \eqref{eq:flow-ep_coupled}, these estimates of course also provide uniform bounds on $\pt (u_\eps)_{g_0}$ in $L^\infty([0,T];H^1(\Si,g_0))$, so $u_\eps$ is bounded uniformly in the space $(X_m(T),\norm{\cdot}_{X_m(T,g_0)})$, which embeds compactly into $C^0([0,T];H^{m-1}(\Si,g_0))$, compare Lemma \ref{lemma:Xm}.
     
    We can hence choose $\eps_n\to 0$ so that $g_{\ep_n}$ converges strongly in $C^1([0,T];\MM^{m})$ to a limiting curve of metrics $g\in C^1([0,T], \UU)$ while the maps $u_{\eps_n}$ converge weak-$*$ in $ X_m(T,g_0)$ to a limit $u\in  X_m(T,g_0)$ which is furthermore so that  harmonic extensions $U_{\eps_n}=(u_{\eps_n})_{g_{\eps_n}}$ converge strongly in $C^0([0,T];H^{m-1}(\Si))$ to $u_g$.  
    This allows us to pass to the limit in \eqref{eq:flow-ep_coupled} to conclude that $(u,g)$ solves the original equation \eqref{eq:flow} on this interval $[0,T]$.
    
    We stress that the length of the interval on which this argument applies does not depend on the specific choice of $u_0$ and $g_0$, but only on upper bounds on the initial energy and the initial $H^m(\Si)$ norm, and on a lower bound for the injectivity radius. As the energy is decreasing along the flow, 
we can hence apply this argument iteratively to establish that the solution indeed exists for as long as $\inj(\Si,g)\nrightarrow 0$ and $\norm{u_g}_{H^m(\Si,g)}\nrightarrow \infty$. 
\end{proof}
\section{Higher Regularity and long-time existence of solutions}\label{sect:higher_reg}
In the previous section, we constructed a solution $(u,g)$ to the flow \eqref{eq:flow} which remains well controlled until either $\inj(\Si,g)\rightarrow 0$ or $\norm{u_g}_{H^m(\Si,g)}\rightarrow \infty$.

In this section, we will complete the proof of our first main theorem, Theorem \ref{thm:1}. In pursuit of this, we establish bounds on the $H^m(\Si,g)$ norm of $u_g$ which remain valid for as long both $\inj(g(t))$ and the radius of balls on which a certain (small) amount of energy is concentrated remain bounded away from zero.

To this end, we will in particular use the following lemma, which extends Proposition 3.3 from \cite{S} to maps from general surfaces.

\begin{lemma}\label{lemma:loc_reg}
    For any $\iota_0>0$, there are constants $\de_1>0$,  $c_1\in (0,\half \iota_0)$ and $C>0$ such that the following holds true for any metric $g \in \MM_{\iota_0}$ and any map $u \in H^1(\p\Si;N)$. If $x_0 \in \p\Si$ and $0 < r < c_1$ are so that
    \begin{equation}\label{ass:E_small_f}
        E_g(u_g;B_{2r}^g(x_0)) \leq \de_1
    \end{equation}
    then we can bound 
    \beqa\label{est:H1_boundary}
        \norm{\nabla u_g}_{L^2(\p\Si\cap B_{r}^g(x_0),g)} &\leq C\norm{P_u \p_\nu u_g}_{L^2(\p\Si\cap B_{2r}^g(x_0),g)} \\
        &+ C r^{-\hf} \norm{\nabla u_g}_{L^2(\Si\cap B_{2r}^g(x_0),g)}+ Cr^{-\frac{3}{2}}\norm{u_g}_{L^2(\Si\cap B_{2r}^g(x_0),g)}.
    \eeqa
\end{lemma}

In addition, we will use 
\begin{lemma}\label{lemma:bit-more-regular}
    Let $m\geq 3$ and let $(u,g)\in X_{m}(T)\times C^1_{loc}([0,T];\MM^{m+1})$ be any solution of the coupled flow \eqref{eq:flow}. Then $u_g\in L^2([0,T]; H^{m+\half}(\Si))$.
\end{lemma}

For such solutions, we can hence consider the  quantities 
\begin{align}\label{def:I_k}
	I_k(u,g) := 1 + \hf\int_{\Si}\abs{\nabla^k u_g}^2 \dd v,\quad  I_{k+\hf}(u,g) := 1 + \hf\int_{\p\Si}\abs{\nabla^k u_g}^2 \dd s
\end{align}
for $k=1,\ldots,m$, recalling the convention that quantities are computed with respect to $g$ unless specified otherwise. 

In Section \ref{subsec:higher_reg_est} we will prove the following proposition, which provides the necessary generalisations of Propositions 4.6 and 4.11 of \cite{S} to the present setting of maps from general surfaces.

\begin{prop}\label{prop:higher_reg} 
    For any $\iota_0>0$ and any $\bar E >0$, there exists $\de_0=\de_0(\iota_0,\bar E)\in (0,\de_1]$, $\de_1=\de_1(\iota_0)$ as in Lemma \ref{lemma:loc_reg}, so that the following holds true. Let $(u,g)$ be a solution to the flow \eqref{eq:flow} which satisfies $g(t) \in \MMiota$, $E_\hf(u,g)\leq \bar{E}$  on a time interval $[t_1,t_2]$ and is so that $u_g \in L^\infty([t_1,t_2]; H^s(\Si)) \cap L^2([t_1,t_2];H^{s+\hf}(\Si))$ for some $s \in \hf \N_{\geq 3}$. Let $\La\geq 1$ be so that 
	\begin{equation}\label{ass:prop-higher-reg-Lambda}
        \norm{u_g(t)}_{H^{s-\half}(\Si,g(t))}^2\leq \La \text{ for } t\in [t_1,t_2].
	\end{equation}
	Then the estimate 
	\begin{align}\label{est:higher_reg_est}
		\sup\limits_{t \in [t_1,t_2]}I_{s}(t) + \int_{t_1}^{t_2}I_{s+\hf}(t)\dd t \leq e^{C\abs{t_2-t_1}\La^{\lceil s\rceil }} I_{s}(t_1)
        \end{align}
    holds true for a constant $C$ which for $s>\frac32$ only depends on $\iota_0, s$, while for $s=\frac32$ we have $C=C_0(\iota_0,\bar E)r_0^{-2}$ for $r_0\in (0,\iota_0)$ chosen so that
    \begin{equation}\label{ass:I-evolv}
        E_{g(t)}(u(t); B_{r_0}^{g(t)}(x))\leq \de_0 \text{ for all } x\in \pSi \text{ and all } t\in [t_1,t_2].
    \end{equation}  
\end{prop}

Combined with a standard iteration argument, these a priori estimates allow us to obtain the desired $H^m$ control on the map component as described in the following proposition. 

\begin{prop}\label{prop:smooth_existence}
    For any solution $(u,g) \in X_{m}(T)\times C^1([0,T];\MM(\Si))$, $m\geq 4$, of the flow  \eqref{eq:flow} we can bound   
	\begin{equation}\label{est:higher_reg_Hm}
		\sup\limits_{t \in [0,T]} \norm{u_g(t)}_{H^m(\Si,g)}\leq C(m,\norm{(u_0)_{g_0}}_{H^m(\Si,g_0)},\iota_0, r_0,\bar E),
	\end{equation}
    and 
    \begin{equation}\label{est:higher_reg_smoothing}
		\sup\limits_{t \in [\tau,T]} \norm{u_g(t)}_{H^m(\Si,g)}\leq C(m,\tau,\iota_0, r_0, \bar E) \text{ for any } 0<\tau<T, 
	\end{equation}
    for $\iota_0>0$ and $r_0>0$ chosen so that $\inj(g)\geq \iota_0$ and so that \eqref{ass:I-evolv} holds on $[0,T]$ for $\de_0$ as in Proposition \ref{prop:higher_reg} and $\bar E$  an upper bound on $E(u_0,g_0)$.
\end{prop}
Combined with local energy estimates which are stated and proven in Section \ref{subsec:loc-energy}, this proposition then allows us to deduce the following proposition. 

\begin{prop}\label{prop:existence-all-initial-data}
    To every initial data $(u_0,g_0)\in H^{\half}(\pSi)\times \MM(\Si)$, there exists a weak solution $(u,g)$ of the coupled flow \eqref{eq:flow} which is smooth on an interval $(0,T_*)$ for a maximal time $T_*>0$. Furthermore, if $T_*$ is finite, then we either have that   $\inj(g(t))\to 0$ as $t\upto T^*$ or that $\iota_0:= \inf_{[0,T_*)} \inj(g)>0$ and $\liminf_{t\upto T_*} \sup_{x\in \pSi} E(u(t);B_r^{g(t)}(x))\geq \half \de_0$ for every $r>0$, $\de_0=\de_0(\iota_0, E_0)>0$ as in Proposition \ref{prop:higher_reg}. 
\end{prop}
The results of Section \ref{sect:metric} guarantee that if $\inj(g(t))\nrightarrow 0$ as $t\to T_*$ then the metrics $g(t)$ converge smoothly to a limiting metric $g(T_*)\in \MM(\Si)$ as $t\upto T_*$. As the energy is non-increasing and as $\int_{0}^{T_*}\norm{\pt u_g}_{L^2(\Si,g)}^2 \dd t$ is finite, we hence deduce that the maps $u_{g(t)}(t)$ converge weakly in $H^1(\Si,g(T_*))$ to a limiting map which is harmonic with respect to $g(T_*)$ and can hence be written as $u(T_*)_{g(T_*)}$ for the weak limit  $u(T_*)$ of $u(t)$ in $H^\half(\pSi)$. We can thus apply the above proposition for this new initial data $(u(T_*),g(T_*))$ to restart the flow and hence continue the flow past any such finite time singularities at which the injectivity radius remains bounded away from zero.

Furthermore, we shall see in Section 
\ref{sect:bubble}
that any such singularity must be caused by the bubbling off of a non-constant half-harmonic map from the disc. Since the energy of such maps is bounded from below by a constant $\eps_0>0$ which only depends on $N$, compare Remark \ref{rmk:bubble_energy_lower_bound}, the number of such singular times is a priori bounded in terms of the initial energy. Repeating the above argument a finite number of times, we hence obtain the existence of a weak solution $(u,g)$ for as long as the metric does not degenerate. 

This completes the proof of our main Theorem \ref{thm:1} up to the proof of the auxiliary results that we stated above, which will be carried out in the subsequent Sections \ref{subsect:normal-proj}-\ref{subsec:loc-energy}, and the characterisation of finite-time singularities of the map component, which is discussed in Section \ref{sect:bubble}. 

In these proofs we will always work with metrics $g\in \MMiota$ for a fixed $\iota_0>0$ and with a fixed number $m\geq 3$, so will at times use the following convention to lighten the notation:

\begin{rmk}\label{rmk:less-sim}
    We write for short $a \leqs b$ if $a\leq Cb$ for a number $C$ that only depends on $m$, $\iota_0$ and as always the topology of the surface $\Si$ and the given manifold $N$. 
\end{rmk}

\subsection{Analysis of terms involving the projection 
onto the normal space}\label{subsect:normal-proj}
For the proofs of both Lemma \ref{lemma:bit-more-regular} and Proposition \ref{prop:higher_reg}, we will use that terms that involve the projection onto the normal space have improved regularity properties, as summarized in the following lemma. 

\begin{lemma}\label{lemma:orth-proj-lot}
    Let $m\geq 3$ and $\iota_0>0$ be any fixed numbers and let $g\in \MMiota^{m+1}$ and $u \in H^{m-\hf}(\p\Si;N)$. Then both $P_u^\perp\p_\nu \nabla^{m-1} u_g$ and $\nabla^{m-1} P_u^\perp\p_\nu u_g $ are in $L^2(\p\Si)$ and we can bound 
    \beqa\label{est:proj-improved-est-L2}
        &\norm{P_u^\perp\pnu \nabla u_{g}}_{L^2(\p\Si)}^2 + \norm{\nabla (P_u^\perp\pnu u_{g})_g}_{L^2(\p\Si)}^2\\
        & \qquad \qquad \leqs \norm{\nabla^2 u}_{L^2(\Si)}^2(1+\norm{\nabla u}_{L^4(\Si)}^2) + (1+\norm{\nabla u}_{L^4(\Si)}^2)^3,
    \eeqa
    and, for $k\in \{3,\ldots,m\}$, 
    \beqa\label{est:proj-improved-est-Hk}
        &\norm{P_u^\perp\pnu \nabla^{k-1} u_{g}}_{L^2(\p\Si)}^2+ \norm{\na^{k-1}(P_u^\perp\pnu u_{g})_g}_{L^2(\p\Si)}^2\\
        & \qquad \qquad \leqs I_{3/2}(u) I_{k}(u)+I_{5/2}(u) S_{k-1}(u) + S_{k-1}(u)^{k+1}.
    \eeqa
\end{lemma}

Here and in the following, we use the shorthand $\leqs$ introduced in Remark \ref{rmk:less-sim} and write for short 
\begin{align}\label{def:S_k} 
    S_s(u,g) :=\sum_{\{\tilde{s} \in \hf \N_{\geq 2}: \tilde s\leq s\}} I_{\tilde{s}}(u,g), \quad \text{ for }I_{\tilde s}(u,g) \text{ defined in \eqref{def:I_k},}
\end{align}
where we drop the reference to the metric if we work with respect to fixed (curves of) metrics $g$.

Since $H^{\half}(\Si)\hookrightarrow L^4(\Si)$, the above lemma in particular ensures that for any $k\in \{2,\ldots, m\}$
\beqa\label{est:proj-improved-est-all-k}
    & \norm{P_u^\perp\pnu \nabla^{k-1} u_{g}}_{L^2(\p\Si)}^2+ \norm{\na^{k-1} (P_u^\perp\pnu u_{g})_g}_{L^2(\p\Si)}^2\\
    &\qquad \qquad \leqs (1+\norm{u_g}_{H^{3/2}(\Si)}^2)\norm{u_g}_{H^k(\Si)}^2 + (1+\norm{u_g}_{H^{k-1/2}(\Si)}^2)^{k+1}.
\eeqa 

\begin{proof}[Proof of Lemma \ref{lemma:orth-proj-lot}]
    The main reason for the improved regularity of the above terms is that the projection of $\p_\tau u$ onto $T_u^\perp N$ vanishes identically on $\pSi$ since $u$ maps $\pSi$ into $N$. This will allow us to obtain the claims made in the lemma by writing all relevant expressions in terms of derivatives of the corresponding harmonic extension $(P_u^\perp \ptau u)_g\equiv 0$ and commutator terms, whose improved regularity properties are discussed in the appendix. 
    
    To make this precise, we recall that  Lemma \ref{lemma:commutator-app} and Remark \ref{rmk:extra-commutator-app} ensure that commutator terms $C_{k-1}(u_g)=C_{k-1}(u_g,u_g)$, $2\leq k\leq m$, as considered in \eqref{def:comm-j},  \eqref{def:comm-rmk-1} and \eqref{def:comm-rmk-2} are in $L^2(\pSi)$ and that their norms $\norm{C_{k-1}(u_g)}_{L^2(\pSi)}^2+\norm{C_{k}(u_g)}_{L^{4/3}(\Si)}^2$ are bounded by the right hand side of the claimed estimates \eqref{est:proj-improved-est-L2} (for $k=2$) respectively \eqref{est:proj-improved-est-Hk} (for $k\geq 3$). 
    
    As $\varphi\big [P_u^\perp\pnu \nabla^{k-1} u_{g}-\na^{k-1} (P_u^\perp\pnu u_{g})_g\big]$, $\varphi$ a cut-off function as in Remark \ref{collar:uniform}, is given by such a commutator term $C_{k-1}(u_g,u_g)$ of the form \eqref{def:comm-rmk-1}, it hence suffices to establish the claims for $P_u^\perp\pnu \nabla^{k-1} u_{g}$.
    
    In collar coordinates $(s,\theta)$ on a neighbourhood of the boundary component $\si_i$, this term can be written as a linear combination of terms of the form $P_u^{\perp }\ps^{i}\ptheta^{j} u_{g}$ with $1\leq i\leq k-j$. As the claimed estimate is trivially true if $i<k-j$, it suffices to establish that all terms of the form $P_u^{\perp }\ps^{k-j}\ptheta^{j} u_{g}$, $j=0,\ldots, k-1$ are controlled in $L^2(\si_i,\dd\theta)$ by the right-hand sides of \eqref{est:proj-improved-est-L2} respectively \eqref{est:proj-improved-est-Hk}. 
    
    For $j\geq 1$ we use that $(P_u^\perp \ptheta u)_g$ vanishes identically since  $P_u^\perp \ptheta u\vert_\pSi\equiv 0$. This allows us to rewrite
    \beqas 
        P_u^{\perp}\ps^{k-j}\ptheta^{j} u_{g} &= \ps^{k-j}\ptheta^{j-1} \big[ P_u^{\perp}\ptheta u_{g} -(P_u^{\perp}\ptheta u)_{g}\big] -\big[\ps^{k-j}\ptheta^{j-1}  (P_u^{\perp}\ptheta u_{g}) -P_u^{\perp}\ps^{k-j}\ptheta^{j} u_{g}]
    \eeqas
    on the collar neighbourhood. As these are again commutator terms of the form \eqref{def:comm-rmk-1} respectively \eqref{def:comm-j}, we immediately obtain the required estimate on $\norm{P_u^{\perp}\ps^{k-j}\ptheta^{j} u_{g}}_{L^2(\pSi)}$ from Lemma \ref{lemma:commutator-app} for any $j\geq 1$.
     
    Finally, for $j=0$ we can use $(\ps^2+\ptheta^2)u_g=0$ to rewrite the corresponding term as  $P_u^{\perp}\ps^{k} u_{g}=-P_u^{\perp}\ps^{k-2}\ptheta^{2} u_{g}$ reducing this final case to the case where $j=2$ discussed above. 
\end{proof}

\subsection{Proof of Lemma \ref{lemma:bit-more-regular} and of 
Propositions \ref{prop:higher_reg} and \ref{prop:smooth_existence}}\label{subsec:higher_reg_est}
We now want to use the above improved estimates on terms involving $P_u^\perp$ to 
prove the key a priori bounds claimed in Proposition \ref{prop:smooth_existence}. To this end, we first establish the 
a priori estimates on the evolution of the quantities $I_{s}$, $s\in \half \N_{\geq 3}$, defined in \eqref{def:I_k} that we claimed in Proposition \ref{prop:higher_reg}, then show that this proposition is applicable for all $s\leq m$ by giving the proof of Lemma \ref{lemma:bit-more-regular} and finally derive Proposition \ref{prop:smooth_existence} from these two results.

\begin{proof}[Proof of Proposition \ref{prop:higher_reg}]
    For $(u,g)$ as in the proposition we will write for short $I_s(t)=I_s(u(t),g(t))$ for $t\in [t_1,t_2]$ as well as $S_s(t):= S_s(u(t),g(t))$ and will often drop the reference to $t$ if there is no room for confusion. As $S_{s-\half}\leq C (1+\norm{u_{g}}_{H^{s-1/2}(\Si,g)}^2 )\leq C \La$ on $[t_1,t_2]$ for $\La$ as in \eqref{ass:prop-higher-reg-Lambda}, it suffices to prove that  
    \beq\label{est:diff-version-prop-higher}
        \mfrac{\dd}{\dd t}I_s+\mfrac14 I_{s+\half}\leq CS_{s-\half}^{ {\lceil s\rceil }}S_s \text{ on }  [t_1,t_2]
    \eeq 
    for a constant $C$ as considered in the proposition. To carry out this proof simultaneously for all $s\in \half \N_{\geq 3}$, we write all such exponents as $s=k-\half \beta$ for $k={\lceil s\rceil } \in \N_{\geq 2}$, i.e. the smallest integer greater or equal to $s$, and $\beta=2({\lceil s\rceil }-s)\in \{0,1\}$.

    We first note that the control on the evolution of the metric component established in Lemma \ref{lemma:metric_variations} and \eqref{est:velocity-by-energy} allows us to bound the contribution of the change of the metric to the evolution of $I_s$ by 
    \beqa \label{est:ddeps-Ik-metric}
    \mfrac{\dd}{\dd \eps}\Big\vert_{\eps=0}I_s(u,g(\cdot+\eps))\leq  C\norm{\p_t g}_{H^k(\Si)} S_s \leq C S_s, \quad C=C(\iota_0,\bar E).\eeqa
    Writing for short
    \beq\label{def:I-hat-k}
        \hat I_{s} := \int_{\p\Si}\inner{\p_\nu^{1-\beta }\nabla^{k-1}u_g}{\nabla^{k-1}(P_u \p_\nu u )_g}\dd s
    \eeq
    and recalling that $\pt u=-P_u \pnu u_g$, we then note that for $s=k-\half$ (and hence $\beta=1$) 
    $$\mfrac{\dd}{\dd \eps}\Big\vert_{\eps=0}I_s(u(\cdot +\eps ),g) = \int_{\p\Si}\inner{\nabla^{k-1}u_g}{\nabla^{k-1}(\p_tu)_g}\dd s=-\hat I_s,$$
    while for $s=k\in \N_{\geq 2}$,
    \beqa \label{est:ddeps-I-k}
    \mfrac{\dd}{\dd \eps}\Big\vert_{\eps=0}I_s(u(\cdot+\eps),g) &= \int_\Si \inner{\nabla^k u_g}{\nabla^k (\p_t u)_g} \dd v = \int_{\p\Si}\inner{\p_\nu\nabla^{k-1}u_g}{\nabla^{k-1}(\p_tu)_g}\dd s + \err_1\\
        &=-\hat I_s+\err_1
    \eeqa 
    for $\err_1=-\int_\Si \inner{\Div(\na^{k} u_g)}{\na^{k-1}(P_u(\p_\nu u_g))_g}\dd v$. 
    
    We recall that
     $\norm{\De\na^{k-1}u_g}_{L^2(\Si)}\leqs\norm{u_g}_{H^k(\Si)} $  since $u_g$ is harmonic and that terms which are supported away from $\pSi$, such as $(1-\varphi) \na^{k-1}(P_u(\p_\nu u_g))_g$ are well controlled, compare Remark \ref{rmk:extra-commutator-app}. 
     Here we recall that $\varphi$ is a cut-off function with $\varphi\equiv 1 $ near $\p\Si$ which is supported in the collar neighbourhood $\Col(\p\Si)$, compare 
     Remark \ref{collar:uniform}. We now extend $\nu$ to  $\Col(\p\Si)$ as described in that remark and use that 
     we can view 
     $\vph \big[ \na^{k-1}(P_u(\p_\nu u_g))_g- P_u(\p_\nu \na^{k-1} u_g)\big]$ as a commutator term $C_{k-1}(u_g,u_g)$ of the form  \eqref{def:comm-rmk-1}. As we can trivially bound $\norm{\vph P_u(\p_\nu \na^{k-1} u_g)}_{L^2(\Si)}\leqs \norm{u_g}_{H^k(\Si)}$, this allows us to bound the error term $\err_1$ by 
    \beqa 
        \abs{\err_1}\leqs  \norm{u_g}_{H^k(\Si)}^2+\norm{C_{k-1}(u_g,u_g)}_{L^2(\Si)}^2\leqs I_k+S_{k-\half}^k\leqs S_{k-\half}^{k-1}S_k.
    \eeqa  

    In both cases we hence obtain that $\tfrac{\dd}{\dd t}I_s\leq- \hat I_s+CS_{s-\half}^{k-1}S_s$, reducing the proof of \eqref{est:diff-version-prop-higher} to the proof that 
    \beq\label{claim:hat-Is}
        -\hat I_s\leq -\mfrac14 I_{s+\half}+C S_{s-\half}^{k} S_{s}.
    \eeq
    To show this we write $P_u=\text{Id}-P_u^\perp$ and split $\hat I_s= T_s-T_s^\perp$ for 
    \beqa\label{est:split-hatIk}
        T_s &:= \int_{\p\Si}\inner{\p_\nu^{1-\beta}\nabla^{k-1} u_g}{\nabla^{k-1} (\p_\nu u_g)_g}\dd s\\ 
        T_s^\perp&:= \int_{\p\Si}\inner{\p_\nu^{1-\beta}\nabla^{k-1} u_g}{\nabla^{k-1}\left(P_u^\perp(\p_\nu u_g)\right)_g}\dd s
    \eeqa

    To bound the first term for $s=k\in \N_{\geq 2}$ (corresponding to $\beta=0$), we can use that $\norm{\pnu\na^{k-1}u_g - \na^{k-1}(\p_\nu u_g)_g}_{L^2(\pSi)}^2 \leqs S_{k-\half}$, compare \eqref{est:grad_normal_commute} and \eqref{est:commute-L12}, as well as that 
    $$\mhalf \int_{\p\Si}\abs{\nabla^{k} u_g}^2\dd s\leq  \int_{\p\Si}\abs{\p_\nu\nabla^{k-1}  u_g}^2\dd s+C S_{k-\half},$$
    which can e.g. be seen by viewing these expressions in collar coordinates and exploiting that $\ps^2 u=-\ptheta^2 u$. Combined, this allows us to bound 
    \beqas
        -T_k &\leq -\int_{\p\Si}\abs{\p_\nu\nabla^{k-1}  u_g}^2\dd s+C\norm{\p_\nu\na^{k-1} u_g}_{L^2(\pSi)} S_{k-\half}^\hf \leq - I_{k+\half}+ C S_{k-\half}
    \eeqas
    for every $k\in \N_{\geq 2}$. Similarly, we have
    \begin{align*}
        -T_{k-\half}&\leq -\int_{\p\Si}\inner{\nabla^{k-1} u_g}{\p_\nu\nabla^{k-1} u_g}\dd s + C\norm{\na^{k-1} u_g}_{L^2(\pSi)} S_{k-\half}^\half\\
        &\leq  - \mhalf\int_{\p\Si}\p_\nu\abs{\nabla^{k-1} u_g}^2\dd s + C S_{k-\half} = -\mhalf\int_{\Si}\Delta\abs{\nabla^{k-1}u_g}^2\dd v + C S_{k-\hf}\\
        &\leq -\int_{\Si}\abs{\nabla^{k}u_g}^2\dd v + C S_k^\half S_{k-1}^\half + C S_{k-\half}\\
        &\leq- 2 I_{k} + CI_{k}^\half S_{k-1}^\half+CS_{k-\half}\leq -  I_{k} + C S_{k-\half}. 
	\end{align*}
    We hence conclude that 
    \beq\label{est:Ts} 
        -T_s\leq - I_{s+\half}+C S_{s} \text{ for all } s\geq \tfrac32.
    \eeq
    To bound $T_s^\perp$ for $s\geq \frac52$ we use that  \eqref{est:proj-improved-est-Hk} ensures that for $k\geq 3$ 
    \beqas 
        \norm{\nabla^{k-1}\left(P_u^\perp\p_\nu u_g\right)_g}_{L^2(\pSi)}^2\leqs S_{3/2}I_k+S_{5/2} S_{k-1}+S_{k-1}^{k+1}.
    \eeqas
    For $s=k\geq 3$ this term is in particular bounded by $CS_{s} S_{s-1}^k$ allowing us to estimate 
    $$\abs{T_s^\perp}\leq I_{k+\half}^\half \norm{\nabla^{k-1}\left(P_u^\perp\p_\nu u_g\right)_g}_{L^2(\pSi)}\leq \mfrac{1}{4} I_{s+\half}+C S_{s} S_{s-1}^s$$
    while for $s=k-\half$, $k\geq 3$, we can use that $\frac52\leq s$ and $\frac32\leq s-1$ to bound 
    \beqa
        \abs{T_s^\perp} &\leq I_{k-\half}^\half \norm{\nabla^{k-1}\left(P_u^\perp\p_\nu u_g\right)_g}_{L^2(\pSi)} \leq C  I_{s}^\half (S_{s-1} I_{s+\hf} +S_{s} S_{s-\hf} +S_{s-\hf}^{k+1})^\half\\
        &\leq \mfrac14 I_{s+\half} +C S_{s-\half}^k S_s.
    \eeqa 
    
    Similarly, applying \eqref{est:proj-improved-est-all-k} for $k=2$ allows us to bound
    \beqa
        \abs{T_2^\perp}&\leq \norm{\p_\nu\nabla u_g}_{L^2(\pSi)}\norm{\nabla\left(P_u^\perp\p_\nu u_g\right)_g}_{L^2(\pSi)}\leq CI_{5/2}^{1/2} \big[I_{3/2} I_2+S_{3/2}^3]^\half\\
        &\leq \mfrac14 I_{5/2} +C S_{3/2}^2 S_2.
    \eeqa
    Combined, these estimates on $T_s$ and $T_s^\perp$, $s\geq 2$, imply that
    $$-\hat I_s=-T_s+T_s^\perp \leq -\mfrac14 I_{s+\half}+C S_{s-\half}^{k} S_{s}$$
    for all such $s\geq 2$ and a constant $C=C(\iota_0,s)$, i.e. establish the estimate \eqref{claim:hat-Is} that was required to prove the claim of the proposition for these $s$. 

    It hence remains to prove this estimate \eqref{claim:hat-Is} for $s=\frac32$, where we recall that in this one case we have claimed that such an estimate holds for a constant of the form $C=C_0(\iota_0) r_0^{-2}$, $r_0$ so that \eqref{ass:I-evolv} holds, rather than for a constant $C=C(\iota_0,s)$. 

    While we can continue to use  \eqref{est:Ts} to bound $-T_s$ if $s=\frac32$, in this case we need to analyse the term $T_s^\perp$ more carefully. For this we use the well known fact that
    \beq\label{est:L4-improved}
        \int_{\Si}\abs{\na U}^4\dd v \leq C \sup_{x\in \Si} E_g(U;B_r^g(x)) \norm{U}_{H^2(\Si)}^2+Cr^{-2} E_g(U), \quad C=C(\iota_0)
    \eeq
    for any $U\in H^2(\Si)$ and any $r\in (0,\inj(g))$, which can e.g. be obtained by applying the bound 
    $$\norm{\phi \abs{\na u}}_{L^4(\Si)}^4=\norm{\phi^2 \abs{\na u}^2}_{L^2(\Si)}^2\leq C \norm{\phi^2 \abs{\na u}^2}_{W^{1,1}(\Si)}^2$$
    for a suitable cover of balls and corresponding cut-off functions $\phi$. 

    Choosing $r=r_0\in (0,\frac12 c_1)$, $c_1>0$ as in Lemma \ref{lemma:loc_reg} so that \eqref{ass:I-evolv} holds for $\de_0=\de_0(\iota_0,\bar E) \in (0,\min(1,\de_1))$ chosen below, this allows us to estimate 
    $$\norm{\na u_g}_{L^4(\Si)}\leq C\de_0^{1/4} I_2^{1/4}+Cr_0^{-1/2} I_1^{1/4}$$
    for a constant $C$ which only depends on $\iota_0$. Combined with \eqref{est:proj-improved-est-L2}, and recalling that $I_1 \leq 1 + \bar{E}$, this allows us to bound 
    \beqas
        \norm{\nabla\left(P_u^\perp(\p_\nu u_g)\right)_g}_{L^2(\pSi)} &\leq C(1+\norm{\na u_g}_{L^4(\Si)}) \norm{\na^2 u}_{L^2(\Si)}+C(1+\norm{\na u_g}_{L^4(\Si)})^3\\
        &\leq C \de_0^{1/4} I_2^{3/4} +Cr_0^{-1/2}I_2^\half+Cr_0^{-3/2}
    \eeqas
    for $C=C(\iota_0,\bar{E})$. As the standard interpolation inequality \eqref{app:Hs-interpol} ensures that 
    $$\norm{\nabla u_g}_{L^2(\pSi)}\leq C \norm{u_g}_{H^{3/2}(\Si)}\leq C\norm{u_g}_{H^1(\Si)}^\hf\norm{u_g}_{H^2(\Si)}^\hf\leq C I_2^\frac{1}{4}, \quad C=C(\iota_0,\bar E)$$
    we hence obtain that $T_{\frac{3}{2}}^\perp = \int_{\p\Si}\inner{\nabla u_g}{\nabla (P_u^\perp \p_\nu u_g)_g}\dd s$ is bounded by
    \beqa
        \abs{T_{3/2}^\perp}&\leq  C \de_0^{1/4} I_2+Cr_0^{-1/2} I_2^{3/4}+C r_0^{-3/2} I_2^{1/4}\\
        &\leq (C \de_0^{1/4}+\mfrac18)  I_2+ C  r_0^{-2} \leq \mfrac14 I_2 +C r_0^{-2}
    \eeqa 
    for a constant $C = C(\iota_0,\bar{E})$ and where the last estimate holds provided $\de_0=\de_0(\iota_0,\bar{E})>0$ is chosen small enough. 

    This completes the proof of the bound \eqref{claim:hat-Is} in this final case where $s=\frac32$ and hence completes the proof of the key Proposition \ref{prop:higher_reg}.
\end{proof}

To prove Lemma \ref{lemma:bit-more-regular}, it is now convenient to 
write the equation for the map component as
\beq\label{est:half-heat-flow}
    \pt u+\pnu u_g=P_u^\perp\pnu u_g
\eeq
to exploit the 
 extra regularity of the right hand side established in Section \ref{subsect:normal-proj}.

\begin{proof}[Proof of Lemma \ref{lemma:bit-more-regular}]
    We will use that if a function $u$ satisfies $\pt u+\pnu u_g=f$ for some $f:[0,T]\times \pSi\to \R^n$ with $\na^{m-1} f_g\in L^2([0,T];L^{2}(\pSi,g))$ and some 
   $g\in C^1_{loc}([0,T];\MM^{m+1})$, then 
    \beqas
        I_m((u,g)(t))
        +\int_0^t\norm{\pnu \na^{m-1} u_g}_{L^2(\pSi)}^2\dd t' &\leq 
            I_m((u,g)(0)) +C\int_0^t \norm{u_g(t)}_{H^m(\Si)}^2 \dd t'\\
        & \quad + C \int_0^t \norm{\na^{m-1}f_g}_{L^{2}(\pSi)}^2+ \norm{\na^{m-1}f_g}_{L^{2}(\Si)}^2\dd t'.
    \eeqas
    We note that it suffices to show this for sufficiently regular $u$ and $f$, as the linearity of the problem means that standard approximation arguments are applicable. For such $u$ and $f$ this estimate follows from a short calculation,  using a simplified version of arguments carried out in detail in the proof of Proposition \ref{prop:higher_reg} above,  noting first that $\mfrac{\dd}{\dd \eps}\vert_{\eps=0}I_m(u,g(\cdot+\eps))\leqs \norm{u_g(t)}_{H^m(\Si)}^2$, compare  \eqref{est:ddeps-Ik-metric}, and then using that
   $$\mfrac{\dd}{\dd \eps}\vert_{\eps=0}I_m(u(\cdot+\eps),g) =\int_\pSi \pnu \na^{m-1} u_g\cdot\na^{m-1}  (-\pnu u_g+f)_g \dd s_g +\err$$ for $$\abs{\err}=\abs{\int_\Si \inner{\Div(\na^{m} u_g)}{\na^{m-1}(-\p_\nu u_g+f)_g}\dd v}\leqs \norm{u}_{H^m(\Si)}\cdot (\norm{u}_{H^m(\Si)}+\norm{\na^{m-1}f_g}_{L^2(\Si)} ),$$ compare \eqref{est:ddeps-I-k}.

    As \eqref{est:proj-improved-est-all-k} ensures that $\na^{m-1}(P_u^\perp\pnu u_g)_g\in L^2([0,T];L^2(\pSi,g)) $ for solutions $(u,g)$ of \eqref{eq:flow} as considered in the lemma, we can apply this to deduce that indeed
$u_g\in L^2([0,T];H^{m+\half}(\Si,g))$.
\end{proof}

\begin{proof}[Proof of Proposition \ref{prop:smooth_existence}]
    We note that it suffices to prove this proposition for $T\leq 1$ as the desired estimates at times $t_0>1 $ can be obtained by considering the solution $(u,g)(t+t_0-1)$ on $[0,1]$ and using that the initial energy of this flow is also bounded by $\bar E$ since the energy is non-increasing along the flow.

    As Lemma \ref{lemma:bit-more-regular} ensures that the map component of the solution has the regularity required in Proposition \ref{prop:higher_reg} for any $s\leq m$, we can obtain the first estimate directly by applying this proposition iteratively on $[0,T]$, starting with $s=\frac32$ for which  \eqref{ass:prop-higher-reg-Lambda} is valid for $\La=C(1+\bar{E})$ as the energy is non-increasing along the flow.

    To prove the second estimate \eqref{est:higher_reg_smoothing} we first apply Lemma \ref{lemma:loc_reg} to conclude that on $[0,T]$ 
    \begin{equation*}
    	\norm{u_g}_{H^{3/2}(\Si)}^2+	\norm{\nabla u_g}_{L^2(\p\Si)}^2 \leq C(r_T,\iota_0) \left(\norm{P_u \p_\nu u_g}_{L^2(\p\Si)}^2 + \norm{u_g}_{H^1(\Si)}^2\right) 
    \end{equation*} 
    where we note that this lemma is applicable as $\de_0\leq \de_1$. 
    As the energy decays according to \eqref{eq:energy-decay}  along the flow, we  hence deduce that 
    \beq\label{est:int-I3/2}
        \int_{0}^{T}\int_{\p\Si}\abs{\na u_g}^2 \dd s \dd t \leq C (1+T)\leq C \text{ for } C=C(\iota_0, r_0,\bar E).
    \eeq
    We can now use this bound to select a time $t_1\in [0,\tau/2]$ with  $I_{\frac{3}{2}}(t_1)\leq C_1=\frac{2}{\tau} C(\iota_0,r_0,\bar E).$ Proposition \ref{prop:higher_reg}, applied for $s = \frac{3}{2}$ on $[t_1,T]$, then ensures that
	$$\sup\limits_{[t_1,T]} I_\frac{3}{2} + \int_{t_1}^T I_{2} \dd t \leq C I_{\frac{3}{2}}(t_1) \leq C_2(\iota_0,r_0,\tau,\bar E),$$
    This not only provides a uniform estimate for $I_\frac{3}{2}$ on  $[\tau,T] \subs [t_1,T]$, but also allows us to select the next $t_2\in (\hf \tau,\frac{3}{4}\tau)$ with which we can continue to iteratively apply Proposition \ref{prop:higher_reg} to establish the claimed bound \eqref{est:higher_reg_smoothing}. 
\end{proof}

\subsection{Proof of Lemma \ref{lemma:loc_reg}}\label{subsec:loc-reg-lemma}
We will derive Lemma \ref{lemma:loc_reg} from the following analogue statement about functions defined on Euclidean half-discs $\DD_r^+=\DD_r\cap \HH$, $\HH:= [0,\infty)\times \R$. 
\begin{lemma}\label{lemma:loc_reg_euc}
    For any $c_0\in (0,1)$ there exist constants $C,\de_2>0$ such that the following holds true. Let $r>0$, let $w:\DD_r^+\to \R^n $ be a harmonic function which maps $\DD_r^+\cap \p\HH$ into $N$ and let $\Omega_1 \subs \Omega_0 \subs \DD_r^+$ be so that 
    $$\dist_{g_0}(\p\Omega_1 \setminus \p\HH,\p\Omega_0\setminus\p\HH) \greq c_0r \text{ and } E(w;\Omega_0) \leq \de_2.$$
    Then we can bound 
    \begin{align}\label{est:euc_loc_reg}
        \norm{\nabla w}_{L^2(\Omega_1\cap\p\HH)}^2 &\leq C(\norm{P_w \p_{\nu_\HH} w}_{L^2(\Omega_0\cap \p\HH)}^2 + r^{-1}\norm{\nabla w}_{L^2(\Omega_0)}^2 + r^{-3}\norm{w}_{L^2(\Omega_0)}^2)
    \end{align}
    where all quantities are computed with respect to the euclidean metric. 
\end{lemma}

To derive Lemma \ref{lemma:loc_reg} from this statement, we can work in the collar coordinates $(s,\theta)\in [0,X(\ell_i)]\times S^1$ of the boundary curve $\si_i$ that contains $x_0$, noting that $X(\ell_i)\geq c_3(\iota_0)>0$ and that the conformal factor satisfies  $c_2(\iota_0)\leq \rho\leq C_2(\iota_0)$, compare Remark \ref{collar:uniform}. 

For $c_1:=\half \min( \iota_0,c_3 C_2^{-1})$ we hence know that $B_{2r}^g(x_0)$, $r\in (0,c_1),$ is contained in this collar neighbourhood  and that the subsets $\Om_1\subset \Om_0$  of $\R \times S^1$ which represent  $B_r^g(x_0)\subset B_{2r}^g(x_0)$ are so that 
$$\Om_0\subset B_{2C_2r}^{g_0} \text{ and } \dist_{g_0}(\partial \Om_1\setminus \si_i, \partial \Om_0\setminus \si_i)\geq \inf \rho^{-1} r\geq C_2^{-1} r, \text{ for } g_0=\dd s^2 + \dd\theta^2.$$
We can hence apply Lemma \ref{lemma:loc_reg_euc} for $c_0(\iota_0):= \half C_2^{-2}$, set $\de_1(\iota_0):= \de_2(c_0(\iota_0))$ and exploit the conformal invariance of the energy to deduce that the function $w$ that represents $u_g$ in collar coordinates satisfies \eqref{est:euc_loc_reg}. Combined with the uniform bounds on the conformal factor, this yields the claimed estimate \eqref{est:H1_boundary} for $u_g$. 
 
\begin{proof}[Proof of Lemma \ref{lemma:loc_reg_euc}]
    As all terms in \eqref{est:euc_loc_reg} have the same scaling behaviour  under dilations, we can assume without loss of generality that $r=1$.  Given $\Om_1\subset \Om_0$ as in the lemma, we can then choose a  cut-off function $\vph\in C_c^\infty(\Om_0) $ such that $\vph \equiv 1$ on $\Omega_1$ and $\norm{ \vph}_{C^2} \leq C =C(c_0)$ by taking the convolution of the characteristic function of $\{x : \dist_{g_0}(x,\Omega_1) \leq \tfrac{c_0}{2}\}$ with a mollifying kernel that is supported on $\DD_{c_0/2}$.

    For the rest of the proof we will viewing these subsets of $\DD_1^+\subset  [0,1]\times (-\pi,\pi)\subset \HH$ as subsets of the cylinder $(\Si_0:= [0,1]\times S^1, g_0=  \dd s^2 + \dd \theta^2)$ and carry out all computations using these coordinates $(s,\theta)$ and metric $g_0$. We furthermore write for short $\si_0 = \{0\}\times S^1$ for the boundary curve of $\Si_0$ which contains the curve $\DD_1^+\cap \p\HH$ that is mapped to $N$ under $w$ and denote by $v_{\Si_0}$ the harmonic extension of maps $v:\partial \Si_0\to \R^n$ to $(\Si_0,g_0)$. 

    We first recall that for any function $v:\Om_0\to \R^n$  we can bound
    \begin{equation}\label{est:extra-proof-1}
        \norm{\nabla (\vphi v)_{\Si_0}}_{L^2(\p \Si_0)} \leq C\norm{\p_{\nu_{\Si_0}} (\vphi v)_{\Si_0}}_{L^2(\p \Si_0)}=C\norm{\p_{s} (\vphi v)_{\Si_0}}_{L^2(\p \Si_0)}.
    \end{equation}
    As $\norm{\vph}_{C^2}\leq C$ we can furthermore  bound $\norm{v \na \vph}_{L^2(\p\Si_0)} \leq C\norm{v \na \vph}_{H^1(\Si_0)}\leq C\norm{v }_{H^1(\Om_0)}$ and $\norm{\nabla(\vph v - (\vph v)_{\Si_0})}_{L^2(\p\Si_0)} \leq C\norm{\Delta (\vphi v)}_{L^2(\Si_0)}\leq C\norm{v}_{H^1(\Omega_0)}$, where the last estimate holds provided $v$ is harmonic. Combined we hence get  
    \begin{align*}
        \norm{\vph \nabla v}_{L^2(\p\Si_0)} 
        &\leq \norm{\nabla(\vph v)_{\Si_0}}_{L^2(\p\Si_0)} + \norm{\nabla(\vph v - (\vph v)_{\Si_0})}_{L^2(\p\Si_0)} + C\norm{v }_{H^1(\Om_0)}\\
        &\leq C (\norm{\p_s[(\vph v)_{\Si_0}-\vphi v]}_{L^2(\p\Si_0)} + \norm{\p_s[\vphi v]}_{L^2(\p\Si_0)})+ C\norm{v }_{H^1(\Om_0)}\\
        &\leq C\norm{\vph\p_s v}_{L^2(\p\Si_0)} +C\norm{v }_{H^1(\Om_0)}
    \end{align*}
    for any harmonic $v:\supp(\vphi)\cap \Si_0\to \R^n$. Applied for $w$, which maps $\supp(\vph)\cap \pSi=\supp(\vph)\cap \si_0$ into $N$, this yields 
    \begin{equation}\label{est:cambidge-1}
        \norm{\vph \nabla w}_{L^2(\p\Si_0)}^2 \leq C\norm{\vph P_w^\perp\p_s w}_{L^2(\si_0)}^2 + C\norm{\vph P_w\p_s w}_{L^2(\si_0)}^2 + C\norm{w}_{H^1(\Omega_0)}^2.
    \end{equation}
 
    To bound the first term on the right hand side, we recall that  $w$ is harmonic and that $\vphi$ vanishes on $\pSi\setminus \si_0$ and use the divergence theorem to write 
    \beqa\label{est:nearly-zero}
        \norm{\vph P_w^\perp\p_s w}_{L^2(\si_0)}^2 &= \int_{\p\Si_0}\p_s w \cdot \vph^2 P_w^\perp \p_s w \dd \theta = \int_{\Si_0}\nabla w \cdot \nabla (\vph^2 P_w^\perp \p_s w) \dd v\\
        &= \int_{\Si_0}\p_s w \cdot \p_s (\vph^2 P_w^\perp \p_s w) + \p_\theta w \cdot \p_\theta (\vph^2 P_w^\perp \p_s w) \dd v.
    \eeqa
    We can now use that $\abs{\p_\theta (\vph^2 P_w^\perp \p_s w) - \p_s (\vph^2 P_w^\perp \p_\theta w)}\leq C (\vphi^2\abs{\na w}^2 +\vphi\abs{\na w})$ since  error terms that result from commuting a derivative with the projection are bounded by $C\vphi^2\abs{\na w}^2$ while terms that result from a derivative falling onto a cut-off function are bounded by $C\vphi\abs{\na w}$. As $\p_{ss}u+\p_{\theta\theta} w=0$ we can also bound  $  \abs{\p_s (\vph^2 P_w^\perp \p_s w)+\p_\theta (\vph^2 P_w^\perp \p_\theta w)}\leq C (\vphi^2\abs{\na w}^2+\vphi\abs{\na w})$. 
        
    Denoting the error term which results from inserting these bounds into \eqref{est:nearly-zero} by $\err_1$, we now observe that the main term in the resulting formula of 
    \beqas       
        \norm{\vph P_w^\perp\p_s w}_{L^2(\si_0)}^2 
        &= \int_{\Si_0} -\p_s w \cdot  \p_\theta(\vph^2 P_w^\perp \p_{\theta} w ) + \p_\theta w \cdot  \p_s(\vph^2 P_w^\perp \p_{\theta} w) \dd v + \err_1
    \eeqas
    vanishes identically, since integration by parts (with respect to $\theta$ for the first term and with respect to $s$ for the second) results in integrals over $\Si_0$ which cancel out and since the integrand $\ptheta w \cdot \varphi^2 P_w^\perp \p_{\theta} w$ of the resulting boundary terms  vanishes identically on $\pSi$ as $w$ maps $\pSi\cap \supp(\phi)\subset \si_0$ into $N$.

    We hence conclude that 
    \beqa \label{est:phi-P-w}      
        \norm{\vph P_w^\perp\p_s w}_{L^2(\si_0)}^2 &=\err_1 \leq C\int_{\Si_0}\vph^2 \abs{\nabla w}^3 \dd v + \vph \abs{\nabla w}^2 \dd v\\
        &\leq C\norm{\nabla w}_{L^2(\Omega_0)} \norm{\vph \nabla w}_{L^4(\Si_0)}^2+ C\norm{w}_{H^1(\Omega_0)}^2.
    \eeqa
    We now recall that $H^\hf(\Si_0) \embed L^4(\Si_0)$ and that  standard theory for harmonic functions allows us to bound $\norm{v_{\Si_0}}_{H^\hf(\Si_0)} \leq C\norm{v}_{L^2(\p \Si_0)}$ for any $v:\pSi_0\to \R^n$. As 
    \begin{align*}
        \norm{\nabla(\vph w) - (\nabla(\vph w))_{\Si_0}}_{H^\hf(\Si_0)} &\leq \norm{\nabla(\vph w) - (\nabla(\vph w))_{\Si_0}}_{H^1(\Si_0)}\leq C\norm{\Delta \nabla(\vph w) }_{H^{-1}(\Si_0)}\\
        &\leq C\norm{\Delta (\vph w)}_{L^2(\Si_0)} \leq C\norm{w}_{H^1(\Omega_0)},
    \end{align*}
    we can hence bound  
    \begin{align*}
        \norm{\vph \nabla w}_{L^4(\Si_0)} &\leq \norm{\nabla (\vph w)}_{L^4(\Si_0)} + \norm{\nabla(\vph) w}_{L^4(\Si_0)}\\
        &\leq C\norm{(\nabla(\vph w))_{\Si_0}}_{H^\hf(\Si_0)} + C\norm{\nabla(\vph w) - (\nabla(\vph w))_{\Si_0}}_{H^\hf(\Si_0)} + C\norm{w}_{H^1(\Omega_0)}\\
        &\leq C\norm{\nabla(\vph w)}_{L^2(\p \Si_0)} +C\norm{w}_{H^1(\Omega_0)}\\
        &\leq C\norm{\vph \nabla w}_{L^2(\p \Si_0)} + C\norm{w}_{H^1(\Omega_0)}.
    \end{align*}
    Combined with \eqref{est:cambidge-1}, \eqref{est:phi-P-w} and the assumed smallness of the energy, this yields 
    $$\norm{\vph \nabla w}_{L^2(\p\Si_0)}^2 \leq C\de_2^{\half} \norm{\vph \nabla w}_{L^2(\p\Si_0)}^2 + C\norm{\vph P_w\p_s w}_{L^2(\si_0)}^2+C\norm{w}_{H^1(\Omega_0)}^2,$$ 
    which, for suitably small $\de_2=\de_2(c_0)>0$, implies the claimed estimate.
\end{proof}

\begin{rmk}\label{rmk:bubble_energy_lower_bound}
    We note that the same argument can be carried out 
    on the whole semi-infinite cylinder $[0,\infty)\times S^1$ with $\vphi\equiv 1$
 to see that the estimate  
    \begin{equation*}
        \norm{\nabla w}_{L^2(\{0\}\times S^1)} \leq C\norm{P_w \p_s w}_{L^2(\{0\}\times S^1)}
    \end{equation*}
    holds for any map $w:[0,\infty)\times S^1\to \R^n$ with $w(\{0\}\times S^1)\subset N$ whose \textit{total energy} $E(w,[0,\infty)\times S^1)$ is smaller than the constant $\de_2$ obtained in the above lemma. As the right hand side vanishes for half-harmonic maps 
    and as $[0,\infty)\times S^1$ is conformal to a punctured disc this excludes the existence of any non-trivial half-harmonic maps from the disc with energy less than $\de_2$, providing an alternative proof of Corollary 3.2 of \cite{S} that applies irrespective of whether $TN$ is parallelisable. 
\end{rmk}

\subsection{Local energy estimates and proof of Proposition \ref{prop:existence-all-initial-data}}\label{subsec:loc-energy}
Both in the proof of the existence of solutions for general initial data as well as in the analysis of finite time singularities, we will make use of the following local energy estimates. 

\begin{lemma}\label{lemma:Ecut}
    For any fixed function $\vph\in C_c^1(\Si,[0,1])$ and any solution $(u,g)$ of the flow \eqref{eq:flow} we can control the evolution of  the localised energy
    \begin{equation*}       
        E_\vph(t)=E_\vphi(u(t),g(t)):=\mhalf\int_{\Si} \vph^2 \abs{\nabla u_g(t)}^2 \dd v
    \end{equation*}
    by 
	\beqs\label{est:loc_E_evol}
		\bigg| \tfrac{\dd}{\dd t} \Ecut+\int_{\pSi}\varphi^2 \abs{\p_t u}^2 \dd s \bigg| \leq C\norm{\pt g}_{L^2(\Si)}\Ecut + C(\norm{\pt g}_{L^2(\Si)}+\norm{\partial_t u}_{L^2(\pSi)}) \norm{\grad \vph}_{L^4(\supp(\vph))} \Ecut^\half
	\eeqs
	for a constant $C>0$ that only depends on an   upper bound $\bar{E}$ on the total energy $E(u(t),g(t))$ and a lower bound $\iota_0>0$ on the injectivity radius $\inj(g(t))$ at the corresponding time. 
\end{lemma}
We recall that \eqref{est:PH_bounded} implies that $\norm{\pt g}_{L^2(\Si)}\leq C=C(\iota_0,\bar E)$ and  that \eqref{eq:energy_var-map} and \eqref{eq:energy_var_metric} ensure that 
 the total energy along solutions of \eqref{eq:flow} decays according to \eqref{eq:energy-decay}. We can thus use that  $\norm{\pt(u,g)}_{L^2([t',t];L^2(\pSi)\times L^2(\Si))}^2\leq E(t')-E(t)$ for any $t'<t$.
\begin{proof}[Proof of Lemma \ref{lemma:Ecut}]
	From the equation \eqref{eq:flow} of the flow and \eqref{eq:energy_var_metric}, upon integrating by parts, we get that 
	\beqs
	    \tfrac{\dd}{\dd t} \Ecut+\int_{\pSi}\varphi^2 \abs{\partial_t u}^2 \dd s=-\half \int_{\Si}\inner{\pt g}{k(u_g,g)} \vph^2 \dd v - 2\int_{\Si} \vph \p_t u_g \cdot \inner{\nabla u_g}{\nabla \vph} \dd v.
	\eeqs
	As $\abs{k(u_g,g)}_g \leq C\abs{\dd u_g}_g^2$, and as \eqref{est:PH_bounded} in particular ensures that  $\p_tg=P_g(\pt g)$ is controlled by  $\norm{\pt g}_{L^\infty(\Si)}\leq C(\iota_0) \norm{\pt g}_{L^2(\Si)}$, we hence get
	\beqs
	    \bigg| \tfrac{\dd}{\dd t} \Ecut+\int_{\pSi}\varphi^2 \abs{\partial_t u}^2 \dd s \bigg|\leq C  \norm{\pt g}_{L^2(\Si)}\Ecut +C \Ecut^\half  \norm{\nabla \vph}_{L^4(\Si)} \norm{\pt u_g}_{L^4(\Si)}
	\eeqs
	and it remains to check that $\norm{\pt u_g}_{L^4(\Si)}\leq C (\norm{\pt u}_{L^2(\pSi)}+ \norm{\pt g}_{L^2(\Si)})$ for some $C=C(\iota_0,\bar{E})$. 
	
	To show this, we split $\pt u_g = (\pt u)_g + \tfrac{\dd}{\dd \ep} (u)_{g(\cdot+\ep)}$ and use Lemmas \ref{lemma:metric_variations} and \ref{lemma:projection}, together with the Sobolev embedding $H^\hf(\Si) \embed L^4(\Si)$ and the elliptic regularity estimate $\norm{f_g}_{H^\hf(\Si)} \leq C\norm{f}_{L^2(\p\Si)}$, to estimate
	\begin{align*}
        \norm{\p_t u_g}_{L^4(\Si)} &\leq \norm{(\pt u)_g}_{L^4(\Si)} + \norm{\tfrac{\dd}{\dd \ep} u_{g(\cdot+\ep)}}_{L^4(\Si)} \leq C\norm{(\p_t u)_g}_{H^\hf(\Si)} + C\norm{\tfrac{\dd}{\dd \ep} u_{g(\cdot+\ep)}}_{H^\hf(\Si)}\\
        &\leq C\left(\norm{\p_t u}_{L^2(\p\Si)} +  \norm{\pt g}_{L^2(\Si)}\norm{u_g}_{H^1(\Si)}\right)	\leq C\left(\norm{\p_t u}_{L^2(\p\Si)} +  \norm{\pt g}_{L^2(\Si)}\right).
	\end{align*}
\end{proof}

We can now combine these local energy estimates with the control on the evolution of the metric obtained in Section \ref{sect:metric} and with Propositions \ref{prop:Hm_plateau_existence} and \ref{prop:smooth_existence}, to deduce

\begin{lemma}\label{lemma:Tmin}
    For any $\bar E$, $\iota_0>0$ and $r_0>0$ there exists a time $T_{\min}=T_{\min}(\bar E,\iota_0,r_0)>0$ so that for any initial data $(u_0,g_0)\in C^\infty(\pSi,N)\times \MM(Si)$ with 
    \beq
        \inj(g_0)\geq 2\iota_0 \text{ and } \sup_{x\in \pSi}  E_{g_0}(u_0; B_{4r_0}^{g_0}(x))\leq \mhalf \de_0
    \eeq
    we get that a solution $(u,g)$ of the flow \eqref{eq:flow} exists and remains so that 
    \beq\label{claim:exist-Tmin}
        \inj(g(t))\geq \iota_0 \text{ and } \sup_{x\in \pSi} E_{g(t)}(u_0; B_{r_0}^{g(t)}(x))\leq \de_0 
    \eeq
    at least on the interval $[0,T_{\min}]$. Here $\de_0 = \de_0(\iota_0,\bar E)>0$ is as in Proposition \ref{prop:higher_reg}.
\end{lemma}

\begin{proof}[Proof of Lemma \ref{lemma:Tmin}]
    Given any such initial data, we can use that Propositions \ref{prop:Hm_plateau_existence} and \ref{prop:smooth_existence} yield the existence of a smooth solution $(u,g)$ to the flow \eqref{eq:flow} on a maximal time interval $[0,T_{\max}(u_0,g_0))$ which can only be finite if  $\inj(g(t))\to 0$ or if energy concentrates as $t\upto T_{\max}$.  
	
    As Remark \ref{rmk:uniform_metric_time} ensures that $g(t)$ remains in the neighbourhood $\UU(g_0)$  of the metric $g_0$  considered in Lemma \ref{lemma:U-basics} at least until $\min(T_0,T_{\max})$, $T_0=T_0(\bar E,\iota_0)>0$ as in that remark, we can bound
    $$\inj(g(t))\geq  \mhalf\inj(g_0)\geq \iota_0 \text{ and } \mfrac{1}{4}g(t) \leq g_0 \leq 4g(t) \text{ for all } t \in [0,\min(T_0,T_{\max})).$$

    Given any $x_0 \in \p\Si$, we  then consider the localised energy $E_\vph$ for a cut-off function $\vph\in C_c^\infty( B_{4r_0}^{g_0}(x_0))$ which is chosen so that  $\vph\equiv 1$ on $B_{2 r_0}^{g_0}(x_0)$ and so that $\abs{\nabla_{g_0} \vph} \leq \frac{C}{r_0}$ for a universal constant $C$. The above bounds on the metrics $g(t)$ ensure that $B_{r_0}^{g(t)}(x_0) \subs B_{2 r_0}^{g_0}(x_0)$, and hence that $\vph\equiv 1 $ on $B_{r_0}^{g(t)}(x_0)$, at least for $t\in [0,\min(T_{\max},T_0))$ and that $\norm{\nabla_{g(t)} \vph}_{L^4(\supp(\vph),g)} \leq Cr_0^{-\hf}$  for such times. As we furthermore have that $E_\vphi\leq \bar E$ we can hence apply Lemma \ref{lemma:Ecut} to deduce that for $t\in [0, \min(T_{\max},T_0))$
    \begin{align*}
    	E_{g(t)}(u_0; B_{r_0}^{g(t)}(x_0)) &\leq E_{\vph}(t)\leq E_{\vph}(0) + Cr_0^{-\half}\int_0^t ( \norm{\pt g}_{L^2(\Si)} + \norm{\p_t u}_{L^2(\p\Si)}) \dd t'\\
    	&\leq \mhalf \de_0 + Cr_0^{-\half} t^\hf  \norm{\pt (u,g) }_{L^2([0,t];L^2(\pSi)\times L^2(\Si))} 
    	\leq \mhalf \de_0 + C_1 r_0^{-\half} t^\hf ,
    \end{align*}
    for $C_1=C_1(\iota_0,\bar E)$. Defining $T_{\min}=T_{\min}(\iota,\bar E,r_0):=\min(T_0,\half \de_0 C_1^{-1}r_0^{\half})$ we hence conclude that \eqref{claim:exist-Tmin} remains valid at least on $[0,\min(T_{\min},T_{\max}(u_0,g_0))$, which in turn ensures that $T_{\max}(u_0,g_0)>T_{\min}$ and hence completes the proof of the lemma. 
\end{proof}

We are now finally in the position to complete the proof of Proposition \ref{prop:existence-all-initial-data}.
\begin{proof}[Proof of Proposition \ref{prop:existence-all-initial-data}]
    Given $(u_0,g_0)\in H^{\half}(\pSi;N) \times \MM(\Si)$, we set $\bar{E} := 2E_\hf(u_0,g_0)$, $\iota_0:= \half \inj(g_0)$ and fix $r_0>0$ so that $\sup_{x\in \pSi}  E_{g_0}(u_0; B_{4r_0}^{g_0}(x))\leq \frac14 \de_0$ for $\de_0=\de_0(\iota_0,\bar E)$ as in Proposition \ref{prop:higher_reg}.

    Letting $u_0^{k} \in C^\infty(\p\Si;N)$ be a sequence of maps which converges to $u_0$ in $H^\hf(\p\Si)$, we can then use that  the conditions of Lemma \ref{lemma:Tmin} are satisfied for all sufficiently large $k$. After dropping the first few $u_{k,0}$ if necessary, we hence deduce that the corresponding solutions $(u_k,g_k)$ of \eqref{eq:flow} all exist and are so that \eqref{claim:exist-Tmin} is satisfied on the $k$ independent interval $[0,T_{\min}]$. Proposition \ref{prop:smooth_existence} hence ensures that 
	\begin{equation*}\sup\limits_{k}\norm{(u_k(t))_{g_k(t)}}_{L^\infty([\tau,T_{\min}];H^m(\Si,g_k(t)))} <\infty \text{ for every } \tau>0 \text{ and } m\in\N. 
	\end{equation*}
    Combined with the control on the metric established in Section \ref{sect:metric} and the fact that $(u,g)$ satisfy the evolution equations \eqref{eq:flow}, this yields $k$-independent bounds on all space-time derivatives of $u_k$ and $g_k$ on $[\tau,T_{\min}]\times \Si$, $\tau>0$. Hence the $(u_k,g_k)$ subconverge smoothly locally on $(0,T_{\min}]\times \Si$ and weakly in $H^1([0,T_{\min}]\times \Si)$ to a solution
    $(u,g)$ of \eqref{eq:flow} to initial data $(u_0,g_0)$ which is smooth away from time $t=0$.    
    This completes the proof of the proposition since the characterisation of the maximal time $T_*$ until which this solution remains smooth immediately follows from Lemma \ref{lemma:Tmin}.
\end{proof}

\subsection{Analysis of finite time singularities}\label{sect:bubble} 

We finally discuss how the analysis of finite time singularities from the disc case, carried out by the second author in Section 8 of \cite{S}, can be applied to establish that singularities of the map component of our coupled flow must be caused by the bubbling off of half-harmonic discs.

The local nature of the analysis in \cite{S} and the fact that we can view a neighbourhood of the boundary of the disc $\DD$ as a cylinder $\Cyl(c_1)=[0,c_1]\times S^1$, which is equipped with a metric which is conformal to $g_0:=\dd s^2+\dd \theta^2$, means that we can 
extract from \cite{S} the following result about the behaviour of harmonic functions on $(\Cyl(c_1),g_0)$.

\begin{lemma}\label{lemma:seq_cyl_bubble}
    Let $v_k: \Cyl(c_1) \to \R^n$ be a sequence of harmonic functions from a fixed Euclidean cylinder $(\Cyl(c_1),g_0)$, $c_1>0$, which is bounded in $H^1(\Cyl(c_1))\cap L^\infty(\Cyl(c_1))$. Suppose that there exist radii $r_k \to 0$ and points $x_k=(s_k,\theta_k)\in \Cyl(\half c_1)$ such that
    \begin{equation*}
        \inf_k	E_{g_0}(v_k;\DD_{r_k}(x_k)) >0 
    \end{equation*}
    Then $r_k^{-1} s_k$ is bounded.
    
    If the maps $v_k$ are furthermore so that $v_k(\{0\}\times S^1)\subset N$ and     
    \begin{equation}\label{est:seq_bubble_tension}
    	r_k\norm{P_{v_k} \partial_s v_k }_{L^2(\{0\}\times S^1,\dd \theta)}^2 \rightarrow 0,
    \end{equation}
    then there exist $\hat{\theta}_k\in S^1$ and scales $\hat r_k\to 0$ so that the rescaled maps 
    \begin{equation*}
        \hat{v}_k(s,\theta):=v_k(\hat{r}_k(s,\theta-\hat{\theta}_k)), \quad (s,\theta)\in [0,c_1\hat{r}_k^{-1}] \times (-\pi\hat{r}_k^{-1},\pi \hat{r}_k^{-1})\subset \HH
    \end{equation*}
    subconverge weakly in $H^{3/2}_{loc}(\HH)$ to a limiting harmonic function $v_\infty$ which is non-constant, has finite energy, maps $\p\HH$ into $N$ and is so that  
    \begin{equation*}
        P_{v_\infty} \ps v_\infty \equiv 0 \text{ on } \partial \HH=\{0\}\times \R.
    \end{equation*}
    Hence, upon composing with a conformal bijection $\psi:\DD \rightarrow \HH$, we obtain a parametrisation $v_\infty \circ \psi:\DD\to \R^n$ of a free boundary minimal disc.
\end{lemma}

As solutions of our flow for which the injectivity radius is bounded from below remain well controlled for as long as the condition \eqref{ass:I-evolv} on the smallness of energy remains valid on fixed size balls, we obtain the desired claim on the behaviour of our flow at finite singular times by combining the above result from \cite{S} with the following lemma. 

\begin{lemma}\label{lemma:flow_bubble}
    Let $(u,g)$ be a smooth solution of the flow \eqref{eq:flow} on an interval $[0,T_*)$ with $\iota_0:=\inf_{[0,T_*)}\inj(g)>0$. Suppose that energy concentrates near at least one of the boundary curves $\si_i\subset \pSi$ as $t\upto T_*$ in the sense that 
    \beq\label{ass:energy-conc-lemma}
        \sup_{t\in [0,T_*)}\sup_{x\in \si_i} E_{g(t)}(u_{g(t)}(t);B_{r}^{g(t)}(x)) \geq \mfrac34\de_0 \text{ for every } r>0,
    \eeq
    for the constant $\de_0=\de_0(\bar E, \iota_0)>0$, $\bar E:= E(u(0),g(0))$, obtained in Proposition  \ref{prop:higher_reg}. Then there are times  $t_k \nearrow T_*$ and a number $c_1=c_1(\iota_0)>0$ so that the maps $u_k:\Cyl(c_1)\subset \Cyl(X(L_{g(t_k)}(\si_i)))\to \R^n$ which represent $u(t_k)$ in the collar coordinates of $(\Si,g(t_k))$ around $\si_i$ satisfy the assumptions of Lemma \ref{lemma:seq_cyl_bubble} and hence undergo bubbling as described in this lemma. 
\end{lemma}

\begin{proof}[Proof of Lemma \ref{lemma:flow_bubble}]
    We first recall from Remark \ref{rmk:uniform_metric_time} that the metrics are uniformly equivalent on the interval $I_0:=[\max(0,T_*-T_0),T_*)$, $T_0=T_0(\bar E,\half \iota_0)>0$ as in this remark, namely that $g(t) \leq 4g(t')$ for all $t,t'\in I_0$.

    Given any $r\in (0,\iota_0)$ and $x\in \pSi$, we can hence obtain a function $\phi_{x,r}$ which satisfies
    \begin{equation*}
        \phi_{x,r} \equiv 1 \text{ on } B_{\frac14 r}^{g(t)}(x), \quad \supp(\phi_{x,r} ) \subs B_{2 r}^{g(t)}(x) \text{ and } \norm{\na_{g(t)}\phi_{x,r}}_{L^\infty(\Si,g(t))}\leq Cr^{-1} 
    \end{equation*}
    for all  $t\in I_0$ by fixing any $t_0\in I_0$ and taking a standard cut-off function on $(\Si,g(t_0))$ which vanishes outside of the corresponding ball with radius $r$ and is equal to $ 1$ on $B^{g(t_0)}_{\half r}(x)$. 

    For any such $x$ and $r$, we can hence apply the estimates on the evolution of the localised energy $E_{\phi_{x,r}}$ obtained in Lemma \ref{lemma:Ecut} to deduce that  
    \beq 
        E(u(t');B^{g(t')}_{r}(x)) \geq E(u(t);B^{g(t)}_{\frac14 r}(x))-C\norm{\pt(u,g)}_{L^2([t',t];L^2(\p\Si)\times L^2(\Si))}^2-Cr^{-\half} (t-t')
    \eeq
    for any $t'<t$ with $t,t'\in I_0$ and for a constant $C=C(\iota_0,\bar E)$.

    As the assumption \eqref{ass:energy-conc-lemma}, allows us to choose times $\tilde t_k\upto T_*$, points $\tilde x_k\in \si_i$ and radii $\tilde r_k\to 0$ with $ E(u(\tilde t_k);B^{g(\tilde t_k)}_{\frac14 \tilde r_k}( x_k))\geq \frac12\de_1$ we hence deduce that 
    $$\inf_{I_k} E(u( t);B^{g(t)}_{4 \tilde r_k}( x_k))\geq \mfrac12 \de_1 -C\norm{\pt (u,g)}_{L^2(I_k;L^2(\pSi)\times L^2(\Si))}^2-C c_0\geq \mfrac14 \de_1,$$
    on $I_k:=[\tilde t_k-c_0r_k,\tilde t_k]$, $c_0:=\frac18 C^{-1}\de_1$, where the last estimate holds for all sufficiently large $k$ since $\norm{\pt (u,g)}_{L^2(I_k;L^2(\pSi)\times L^2(\Si))}^2\to 0$ as $k\to \infty$.

    Choosing $t_k\in I_k$ so that $\norm{P_{u(t_k)}\p_\nu (u(t_k))_{g(t_k)}}_{L^2(\p\Si)}^2 \leq 2 \abs{I_k}^{-1} \int_{I_k}\norm{P_{u}\p_\nu u_{g}}_{L^2(\p\Si)}^2 \dd t$ and using that $\int_{I_k}\norm{P_{u}\p_\nu u_{g}}_{L^2(\p\Si)}^2 \dd t = \norm{\pt u}_{L^2(I_k;L^2(\pSi))}^2\to 0$, we hence deduce that 
    \begin{equation*}
        E(u( t_k);B^{g(t_k)}_{4 \tilde r_k}( x_k))\geq \mfrac14\de_1>0 \text{ and } r_k\norm{P_{u(t_k)}\p_\nu (u(t_k))_{g(t_k)}}_{L^2(\p\Si)}^2\rightarrow 0.
    \end{equation*}    
    As the conformal factors $\rho_k$ of the metrics which represent $g(t_k)$ in the collar coordinates around $\si_i$ are bounded below by some constant $c=c(\iota_0)>0$, and as $X(L_{g(t_k)}(\si_i))\geq c_1=c_1(\iota_0)$, compare Remark \ref{collar:uniform}, we hence deduce that the maps $u_k:\Cyl(c_1)\to \R^n$ which represent $u(t_k)_{g(t_k)}$ in these coordinates are indeed so that 
    $$ E_{g_0}(u_k;B_{r_k}) \geq \mfrac14 \de_1 \text{ and } r_k\norm{P_{u_k}\p_s u_k}_{L^2(\{0\}\times S^1,\dd \theta)}^2\to 0 \text{ for }r_k=4c^{-1}\tilde r_k\to 0. $$ 
\end{proof}

\section{Asymptotics}\label{sect:asym}
In this final section, we complete the proof of Theorem \ref{thm:2}, which gives a result on the asymptotic convergence of the flow under the assumption that neither the metric, nor the map, develop singularities as $t \rightarrow \infty$.

For this, we prove the following sequential compactness result, which can be readily applied to the solution of the flow studied above.

\begin{prop}\label{prop:sequence_converge}
	Let $(u_n,g_n) \in C^\infty(\p\Si;N)\times \MM(\Si)$ be a sequence with half-energy uniformly bounded from above, $E_\hf(u_n,g_n) \leq \bar{E} < \infty$, and injectivity radius uniformly bounded from below, $\inj(g_n) \geq \iota_0 > 0$. Suppose further that energy does not concentrate, in the sense that there is an $r>0$ such that
	\begin{equation}\label{est:asm_energy_bound}
		\sup\limits_{n}\sup\limits_{x_0 \in \Si}\int_{B_r^{g_n}(x_0)}\abs{\nabla_{g_n} (u_n)_{g_n}}_{g_n}^2 \dd v_{g_n} < \de_1 = \de_1(\iota_0),
	\end{equation}
    where $\de_1$ is the constant from Lemma \ref{lemma:loc_reg}. Suppose further that $(u_n,g_n)$ are an almost critical sequence for the half-energy functional, in the sense that
	\begin{align}
		\norm{P_{u_n}\p_{\nu_{g_n}} (u_n)_{g_n}}_{L^2(\p\Si,g_n)} &\rightarrow 0\label{est:map_almost_crit}\\
		\norm{P_{g_n}^{H}(k((u_n)_{g_n},g_n)}_{L^2(\Si,g_n)} &\rightarrow 0\label{est:metric_almost_crit}.
	\end{align}
	Then there exist diffeomorphisms $f_n:\Si \rightarrow \Si$ such that along a subsequence,
	\begin{align}
		f_n^* g_n &\rightarrow g_\infty\in \MM(\Si)\label{est:asm_metric_conv} \text{ smoothly}\\
		u_n \circ f_n &\rightarrow u_\infty\label{est:asm_map_conv} \text{ strongly in } H^\hf(\p\Si,g_\infty), \text{ weakly in } H^1(\p\Si,g_\infty)
	\end{align}
	where $(u_\infty,g_\infty)$ is a critical point of the half-energy $E_\hf$, i.e.  $u_\infty$ is a $g_\infty$ half-harmonic map for which $(u_\infty)_{g_\infty}$ is conformal.
\end{prop}
\begin{proof}[Proof of Proposition \ref{prop:sequence_converge}]
	Firstly, since $\inj(g_n) \geq \iota_0 > 0$, we can apply Mumford's compactness theorem, compare \cite{Tromba}, to immediately obtain the diffeomorphisms $f_n$ and the convergence of the metrics \eqref{est:asm_metric_conv} along a suitable subsequence. We then define $\tilde{u}_n := u_n \circ f_n$ and $\tilde{g}_n := f_n^* g_n$, and note that $P_{\tilde{u}_n} \p_{\nu_{\tilde{g}_n}} (\tilde{u}_n)_{\tilde{g}_n} = f_n^*\left(P_{u_n}\p_{\nu_{g_n}} (u_n)_{g_n}\right)$. Hence  \eqref{est:map_almost_crit} implies that 
	\begin{equation}\label{est:u_tilde_ac}
		\norm{P_{\tilde{u}_n}\p_{\nu_{\tilde{g}_n}} (\tilde{u}_n)_{\tilde{g}_n}}_{L^2(\p\Si,\tilde{g}_n)} \rightarrow 0.
	\end{equation}
	Similarly, we see that the local energy bound \eqref{est:asm_energy_bound} holds for the sequence $(\tilde{u}_n,\tilde{g}_n)$, and so combining the estimate \eqref{est:H1_boundary} of Lemma \ref{lemma:loc_reg} together with the convergence \eqref{est:asm_metric_conv}, we have that $\tilde{u}_n$ is a bounded sequence in $H^1(\p\Si,g_\infty)$ and that $\p_{\nu_{g_\infty}}\left(\tilde{u}_n\right)_{g_\infty}$ is a bounded sequence in $L^2(\p\Si,g_\infty)$. Hence, upon passing to a subsequence, we can assume that $\tilde{u}_n \rightharpoonup u_\infty$ weakly in $H^1(\p\Si,g_\infty)$ and $\p_{\nu_{g_\infty}}\left(\tilde{u}_n\right)_{g_\infty} \rightharpoonup \p_{\nu_\infty} (u_\infty)_{g_\infty}$ weakly in $L^2(\p\Si,g_\infty)$. By the compactness of the embedding $H^1(\p\Si) \embed C(\p\Si)$, we further have uniform convergence of $\tilde{u}_n$ and hence also the projections $P_{\tilde{u}_n}$ converge uniformly to $P_{u_\infty}$. We thus conclude that 
    \begin{equation*}
		P_{u_\infty}\p_{\nu_{g_\infty}}(u_\infty)_{g_\infty} = 0,
	\end{equation*}
    i.e. that $u_\infty$ is a half-harmonic map (with respect to $g_\infty$).
	
	It remains to show that $(u_\infty)_{g_\infty}$ is conformal. For this, we first note that pulling back by diffeomorphisms does not affect \eqref{est:metric_almost_crit}, and so
	\begin{equation*}
		\norm{P_{\tilde{g}_n}^{H}(k((\tilde{u}_n)_{\tilde{g}_n},\tilde{g}_n)}_{L^2(\Si)} \rightarrow 0.
	\end{equation*}
	We can now use the estimates from Section \ref{sect:metric} on the stress-energy tensor and horizontal projection to show that $P_{g_\infty}^{H}k((u_\infty)_{g_\infty},g_\infty) = 0$. 

    Specifically, we use \eqref{est:PH_bounded} and \eqref{est:basic_3} to get the convergence
    \begin{align*}
        \norm{P_{g_\infty}^{H}\left(k((u_\infty)_{g_\infty},g_\infty) - k((\tilde{u}_n)_{\tilde{g}_n},\tilde{g}_n)\right)}_{L^2(\Si)} \rightarrow 0
    \end{align*}
    and the bounds \eqref{est:PH_Lip} and \eqref{est:k-L1-energy} to show
    \begin{equation*}
        \norm{(P_{g_\infty}^{H} - P_{\tilde{g}_n}^H) k((\tilde{u}_n)_{\tilde{g}_n},\tilde{g}_n)}_{L^2(\Si)} \rightarrow 0
    \end{equation*}
    which all combine to show $P_{g_\infty}^{H}k((u_\infty)_{u_\infty},g_\infty) = 0$.
    
    As observed in Remark \ref{rmk:conformal}, this immediately implies that $(u_\infty)_{g_\infty}$ is conformal, which completes the proof of the proposition.
\end{proof}

\begin{proof}[Proof of Theorem \ref{thm:2}]
Let  $(u,g)$ be a global solution to the flow \eqref{eq:flow} with $\inj(g(t)) \geq \iota_0$ for which there is no energy concentration as $t\to \infty$ as specified in the theorem. 
Then the energy decay formula \eqref{eq:energy-decay} ensures that there is a sequence of times $t_j$ such that \eqref{est:map_almost_crit} and \eqref{est:metric_almost_crit} hold
for $(u_j,g_j) = (u(t_j),g(t_j))$. Hence we can apply Proposition \ref{prop:sequence_converge} to obtain the required limiting pair $(u^*,g^*)$, the smooth convergence $g(t_j) \rightarrow g^*$ and weak $H^1(\p\Si,g^*)$ convergence of the map component $u(t_j) \rightarrow u^*$. We then note the estimate \eqref{est:higher_reg_smoothing} provides uniform $H^k$ bounds on $u(t_j)$, which allows us to upgrade to smooth convergence and hence conclude the proof.
\end{proof}

\appendix

\section{}

\subsection{Sobolev Inequalities}
Throughout the paper we use standard Sobolev inequalities which we briefly recall for the convenience of the reader. 
As 
$W^{s,p}(\Om)\hookrightarrow L^{q}(\Om)$ for $n$-dimensional domains $\Om$, $p<\frac{n}{s}$ and $q= \frac{np}{n-sp}$ we can use that $W^{s,p}(\Si)\hookrightarrow L^{q}(\Si)$ for $q=\frac{2p}{2-sp}$, 
in particular  $H^{\half}(\Si)\hookrightarrow L^4(\Si)$,
    and  $W^{s,p}(\pSi)\hookrightarrow L^{q}(\p\Si)$ for $q=\frac{p}{1-sp}$. 
    As  $W^{1,p}(\Si)\hookrightarrow W^{1-\frac{1}{p},p}(\pSi)$, we 
can furthermore use that 
$W^{1,p}(\Si)\hookrightarrow L^r(\pSi)$ for $r= \frac{p}{1-(1-\frac1p)p}=\frac{p}{2-p}$ and $1\leq p<2$, in particular $W^{1,\frac43}(\Si)\hookrightarrow L^2(\pSi)$.

We will also use that the following standard interpolation inequality 
\beq \label{app:Hs-interpol}
\norm{f}_{H^s(\Om)}^2\leq \norm{f}_{H^{s+r}(\Om)}\norm{f}_{H^{s-r}(\Om)}, \qquad f\in H^{s+r}(\Om)
\eeq
holds true both on $\Om=\Si$ and on $\Om=\pSi$ for any $0\leq r\leq s$ and any $g\in \MMiota$ for a constant $C$ that only depends on $r,s$, $\iota_0$ and as usual the topology of $\Si$, compare e.g. Theorem 1 of \cite{Brezis_Mironescu_interpolation}.

\subsection{Commutator estimates}\label{sect:app}
While many of the operations we use repeatedly throughout the paper, such as harmonic extensions, derivatives and projections, do not commute, the error terms  resulting from changing the order of operation will in general be of lower order. The purpose of this appendix is to collect and prove the relevant estimates for such commutators which will be used throughout the paper. 

To begin with we note that terms which are obtained from commuting two differential operators $L_{1,2}$ (which are defined in terms of the metric) of total order $j_1+j_2=j\in \N$ can be bounded by 
\beq\label{est:commute-L12}
\norm{L_1L_2 w-L_2L_1 w}_{L^2(\Si)}\leq C \norm{w}_{H^{j-1}(\Si)} \text{ and }
\norm{L_1L_2 w-L_2L_1 w}_{L^2(\pSi)}\leq C \sum_{k\leq j-1}\norm{\na^{k}w}_{L^2(\pSi)}.
\eeq

We also recall that the extensions $\tau$ and $\nu$ of the unit tangent and normal vector fields on $\pSi$ are chosen so that $\peta u_g$, $\eta\in \{\tau,\nu\}$ are again harmonic on the corresponding collar neighbourhood $\Col(\pSi)$ of the boundary, compare Remark \ref{collar:uniform}, and let $\varphi$ be a cut-off function with $\varphi\equiv 1$ near $\pSi$ as in that remark. As the $C^k$ semi-norm of harmonic functions can be controlled just in terms of the energy away from the boundary, we can hence bound the difference between $\peta u_g$ and its harmonic extension by 
\beq\label{est:commute-he-petau}
\norm{\varphi\big[\peta u_g-(\peta u_g)_g)\big]}_{C^k(\Si)}\leq C E(u,g)^\half 
\eeq 
for any $k\in \N$, any $g\in\MMiota$ and a constant $C$ that only depends on $k$ and $\iota_0$. 

Combined, these estimates in particular allow us to bound 
	\begin{equation}\label{est:grad_normal_commute}
		\norm{\nabla^j (\p_{\eta} u_g)_g - \p_{\eta}\nabla^j u_g}_{L^2(\p\Si)} 
        \leq C \sum_{k\leq j}\norm{\na^{k}u_g}_{L^2(\pSi)} \leq 
        C\norm{u_g}_{H^{j+1}(\Si)}
	\end{equation}
where the last estimate is not optimal, but will often be sufficient in applications.

We will need commutator estimates mostly for derivatives and harmonic extensions of functions of the form $w=P_{v_g}(\peta u_g)$
and $w=P_{v_g}^\perp (\peta u_g)$,  where 
 $P_p, P^\perp_p\in C_b^\infty(\R^n;\R^{n\times n})$ denote the (fixed) extensions of the projections onto the tangent and normal space of our support manifold 
chosen in Remark \ref{rmk:extension-projection}.

To state the relevant estimates in a form that makes them applicable to treat both types of terms, we let 
 $A \in C_b^\infty(\R^n;\R^{n\times n})$ be any bounded matrix valued function with bounded derivatives 
and consider commutator terms of the form 
\beq\label{def:comm-j}
    C_j(U,V):= L_1L_2\big(\varphi A(V)(\peta U)\big)-L_1\big(\varphi A(V)(\peta L_2 U)\big)
\eeq
for linear differential operators $L_{1,2}$ of order $j_1\geq 0$, $j_2\geq 1$ with  $j_1+j_2\leq j$ and a cut-off function $\varphi\in C_c^\infty(\Col(\pSi))$ which is chosen as in Remark \ref{collar:uniform} and is in particular so that $\varphi\equiv 1$ near $\pSi$. A short argument, which is carried out below, allows us to prove the following basic, but useful, estimates.

\begin{lemma}\label{lemma:commutator-app}
Let $g\in\MMiota$, $A \in C_b^\infty(\R^n;\R^{n\times n})$ and let $U,V\in H^{j}(\Col(\pSi))$, for some $j\geq 2$,  $\Col(\pSi)=\bigcup_i\Col(\si_i,g)$ the union of the neighbourhoods of the boundary curves $\si_i$ described in Section \ref{subsec:metric-basic}. Then the following holds true for commutator terms 
$C_{j}(U,V)$ and $C_{j-1}(U,V)$ of the form \eqref{def:comm-j} obtained from differential operators $L_{1,2}$ of order $j_{1,2}$ with $j_2\geq 1$ and $j_1+j_2\leq j$ respectively $j_1+j_2\leq j-1$.

 If $j=2$ then $C_2(U,V)\in L^{4/3}(\Si)$ and $C_1(U,V)\in L^2(\pSi)$ and we can bound
\beqa\label{est:app-c2-L43}
  &  \norm{C_2(U,V)}_{L^{4/3}(\Si)}
 +\norm{C_1(U,V)}_{L^{2}(\pSi)} \\
     & \qquad \qquad \leq C (1+\norm{ \na V}_{L^4(\Si)}) \norm{U}_{H^2(\Si)} + C\norm{V}_{H^2(\Si)}\norm{ \na U}_{L^4(\Si)}\\
 & \qquad \qquad + C(1+\norm{ \na V}_{L^4(\Si)}^2)\norm{ \na U}_{L^4(\Si)}
\eeqa
while for functions $U,V\in H^{5/2}(\Si)$ we furthermore have that $C_2(U,V)\in L^{2}(\Si)$ and
\beqa\label{est:app-c2-L2}
    \norm{C_2(U,V)}_{L^{2}(\Si)}&\leq C  (1+\norm{V}_{H^{3/2}(\Si)}) \norm{U}_{H^{5/2}(\Si)}+ C\norm{V}_{H^{5/2}(\Si)}\norm{U}_{H^{3/2}(\Si)}\\
    &\qquad +C(1+\norm{V}_{H^{2}(\Si)}^2)\norm{U}_{H^{3/2}(\Si)}.
\eeqa
If $j\geq 3$ then
 $C_j(U,V)\in L^{2}(\Si)$ and $C_{j-1}(U,V)\in L^2(\pSi)$
and we can bound 
\beqa\label{est:app-cj-L43}
   & \norm{C_j(U,V)}_{L^{4/3}(\Si)}+\norm{C_{j-1}(U,V)}_{L^{2}(\pSi)} \\
    & \qquad\qquad\leq C (1+\norm{V}_{H^{3/2}(\Si)}) \norm{U}_{H^{j}(\Si)} +C \norm{U}_{H^{3/2}(\Si)}\norm{V}_{H^{j}(\Si)}\\
    &\quad \qquad\qquad + C\norm{V}_{H^{5/2}(\Si)}\norm{U}_{H^{j-1}(\Si)} + C\norm{U}_{H^{5/2}(\Si)}\norm{V}_{H^{j-1}(\Si)}\\
    &\quad\qquad \qquad + C(1+\norm{V}_{H^{j-1}(\Si)}^j)\norm{U}_{H^{j-1}(\Si)}
\eeqa
and 
\beqa\label{est:app-cj-L2}
    \norm{C_j(U,V)}_{L^{2}(\Si)} &\leq C(1+\norm{V}_{H^{5/2}(\Si)}) \norm{U}_{H^{j}(\Si)}+ C\norm{U}_{H^{5/2}(\Si)} \norm{V}_{H^{j}(\Si)}\\
    &\quad +C(\norm{V}_{H^{5/2}}+\norm{V}_{H^2(\Si)}^2) \norm{U}_{H^{j-\half}(\Si)}\\
    &\quad + C(\norm{U}_{H^{5/2}(\Si)}+ \norm{V}_{H^2(\Si)}\norm{U}_{H^2(\Si)})\norm{V}_{H^{j-1/2}(\Si)}\\
    &\quad + C(1+\norm{V}_{H^{j-1}(\Si)}^j)\norm{U}_{H^{j-1}(\Si)}
\eeqa
for a constant $C$ that only depends on the operators $L_{1,2}$, $\iota_0$ and $A$.
\end{lemma}

At times, it will suffice to use that the above estimates in particular imply that for $j \geq 3$
\beq \label{est:L^2-commut-weaker}
\norm{C_j(u_g,u_g)}_{L^2(\Si)}^2\leq S_{j-\half}(u,g)^{j}S_j(u,g), \text{ for } S_s(u,g) \text{ as in \eqref{def:S_k}}
\eeq

We can apply the above lemma not only to control error terms that result from commuting derivatives and projections, but also to compare functions of the form $w=P_{v_g}\peta u_g$ with their harmonic extensions. 
To this end, we note that  $\na^{j-2} \Delta (\varphi A(v_g) \peta u_g)$ can be viewed as a commutator term
$C_{j}(v_g,u_g)$ of the above form for 
 $L_1=\na_g^{j-2}
$ and $L_2=\Delta_g$ since $\Delta_g u_g \equiv 0$. 
To bound the difference between such $w=\varphi A(v_g)\peta u_g $ and their harmonic extension $w_g$ we will hence be able to combine the above commutator estimates with the standard
elliptic estimate 
\beq\label{est:diff-extension-sigma}
    \norm{w-w_g}_{W^{k,p}(\Si)}\leq C\norm{\Delta_g w}_{W^{k-2,p}(\Si)}, \quad C=C(k,p,\iota_0), \quad 1<p<\infty
\eeq and the fact that  $W^{1,\frac{4}{3}}(\Si)
\hookrightarrow L^2(\pSi)$. 

This immediately allows us to deduce that the above commutator estimates also apply for 
terms of the form $\na^j( \varphi \big[ A(v_g)\peta u_g-(A(v_g)\peta u_g)_g\big] )$
and, as $\varphi[(\peta u_g)_g-\peta u_g]$ is controlled by \eqref{est:commute-he-petau}, also for terms of the form $\na^j( \varphi \big[ A(v_g)(\peta u_g)_g-(A(v_g)\peta u_g)_g\big] )$. 
Combined with the $C^k$ control on harmonic functions on $\supp(1-\vph)\subset \Si\setminus \pSi$, 
we hence obtain the following useful consequence of the above lemma.

\begin{rmk}\label{rmk:extra-commutator-app}
    The commutator estimates stated in the above Lemma \ref{lemma:commutator-app} are all valid also for commutator terms of the form 
\begin{eqnarray}
    \label{def:comm-rmk-1}
C_j(v_g,u_g)&=&\varphi\big(
    \na_g^{j}\big[(A(v_g)\peta u_g)_g\big]-\na_g^{j_1} \big[A(v_g) \peta \na_g^{j_2} u_g\big]\big), \qquad \, j_1+j_2=j,
   \\
   \label{def:comm-rmk-2}
   C_j(v_g,u_g)&=&\vph \big(
    \na_g^{j}\big[(A(v_g)\peta u_g)_g\big]-\na_g^{j_1} \big[A(v_g) \na_g^{j_2} (\peta  u_g)_g\big]\big), \quad j_1+j_2=j,
\end{eqnarray} 
    and for cut-off terms that are supported away from $\pSi$ of the form 
      \begin{eqnarray}
      \label{def:comm-rmk-3-1}
        C_j(v_g,u_g) &=& \nabla_g^{j_0} \left[(1-\vph) \na_g^{j_1} (A(v_g)\na_g^{j_2}(\p_\eta u_g)_g )\right]
    , \qquad
        j_0+j_1+j_2=j\\
        \label{def:comm-rmk-3-2}
        C_j(v_g,u_g) &=& \nabla_g^{j_0} \left[(1-\vph) \na_g^{j_1} (A(v_g)\p_\eta u_g)_g \right], \qquad\qquad\, \qquad j_0+j_1=j.
    \end{eqnarray}
\end{rmk}

\begin{proof}[Proof of Lemma \ref{lemma:commutator-app}]
To obtain the claim for $j=2$ we can use that $\na C_1$ and $C_2$ satisfy pointwise bounds of
\beqas 
    \abs{C_2(U,V)}+\abs{\na C_1(U,V)} \leq C\big[ \abs{\na^2 U}(1+\abs{\na V}) +\abs{\na^2 V}\abs{\na U}+(1+\abs{\na V}^2)\abs{\na U}\big]
\eeqas
$C$ as in the lemma. The claims in this case can hence be obtained by using that $H^\half(\Si)\hookrightarrow L^4(\Si)$ and $W^{1,\frac{4}{3}}(\Si)\hookrightarrow L^2(\pSi)$ and applying H\"older's inequality.

Similarly, for $j\geq 3$ we have pointwise estimates of 
\beqas 
    \abs{C_j(U,V) }+\abs{\na C_{j-1}(U,V)} &\leq C (1+\abs{\na V}) \abs{\na^{j} U}+\abs{\na^{j} V}\abs{\na U}\\
    &\quad+C(\abs{\na V}^2+\abs{\na^2 V}+1) \abs{\na^{j-1} U}+(\abs{\na U}\abs{\na V}+\abs{\na^2 U})\abs{\na^{j-1} V}\\
    & \quad +C \sum_{\substack{\abs{\al} \leq j+1 \\ 1\leq \al_i \leq j-2}}\abs{\nabla^{\al_1} V}\cdots\abs{\nabla^{\al_{\ell-1}}V} \abs{\nabla^{\al_\ell} U}.
\eeqas
The claims in this case can hence be obtained by using the Sobolev embeddings $H^\half(\Si)\hookrightarrow L^4(\Si)$ and $H^{j-1}(\Si)\hookrightarrow W^{j-2,p}(\Si)$, $p<\infty$, and  $H^{3/2}(\Si)\hookrightarrow L^\infty(\Si)$  as well as H\"older's inequality. 
\end{proof}

\textsc{Melanie Rupflin} and \textsc{Christopher Wright}: Mathematical Institute, University of Oxford, Oxford, United Kingdom

\textsc{Michael Struwe}: Departement Mathematik, ETH Z\"urich, Z\"urich, Switzerland

\end{document}